\documentclass{amsart}
\usepackage{amssymb,mathtools}
\usepackage[british]{babel}
\usepackage{enumitem}
\usepackage[latin1]{inputenc}
\usepackage{soul}
\usepackage{hyperref}

\newtheorem{theorem}{Theorem}
\numberwithin{theorem}{section}
\newtheorem{corollary}[theorem]{Corollary}
\newtheorem{lemma}[theorem]{Lemma}
\newtheorem{proposition}[theorem]{Proposition}
\theoremstyle{definition}
\newtheorem{definition}[theorem]{Definition}
\newtheorem{remark}[theorem]{Remark}
\newtheorem{example}[theorem]{Example}

\numberwithin{equation}{section}

\DeclareRobustCommand{\SkipTocEntry}[5]{}

\DeclareMathOperator{\dom}{dom}
\DeclareMathOperator{\rng}{rng}

\newcommand{\ipp}{\mathsf{IPP}}

\title{Avoiding logical strength in real analysis}
\author{Anton Freund}
\author{Nicholas Pischke}
\author{Patrick Uftring}

\address{Anton Freund, University of W\"urzburg, Institute of Mathematics, Emil-Fischer-Str.~40, 97074 W\"urz\-burg, Germany}
\email{anton.freund@uni-wuerzburg.de}
\address{Nicholas Pischke, University of Bath,  Department of Computer Science, Claverton Down, Bath, BA2 7AY, United Kingdom}
\email{nnp39@bath.ac.uk}
\address{Patrick Uftring, University of the Bundeswehr Munich, Department of Computer Science, Werner-Heisenberg-Weg 39, 85579 Neubiberg, Germany}
\email{patrick.uftring@unibw.de}

\thanks{The work of Freund has been funded by the Deutsche Forschungsgemeinschaft (DFG, German Research Foundation) -- Project number 460597863.}

\begin{document}

\begin{abstract}
In reverse mathematics, real numbers are traditionally represented by Cauchy sequences with a given rate of convergence. We work without rates and speak of slow Cauchy sequences. It turns out that almost all one-dimensional real analysis from the reverse mathematics book by Simpson can then be developed in theories that are conservative over~$\mathsf{RCA}_0$. Specifically, we obtain clusters of equivalences with the infinite pigeonhole principle and the strong cohesive principle. The second cluster includes results like the Bolzano-Weierstrass and Arzel\`a-Ascoli theorems, which are traditionally associated with the stronger axiom of arithmetical comprehension, but also the Heine-Borel theorem, which is normally separated from these principles. This suggests two things: In elementary analysis, one can avoid logical strength to an extent that the traditional picture seems to forbid. And the division of the so-called reverse mathematics zoo into analytical and combinatorial principles may be less rigid than previously assumed.
\end{abstract}


\maketitle

\tableofcontents

\section{Introduction}

Reverse mathematics is a program in mathematical logic that aims to determine the minimal set existence axioms that allow to prove theorems from various fields. For an introduction, the reader can, e.g., consult the classical book by S.~Simpson~\cite{simpson09} or the more recent one by D.~Dzhafarov and C.~Mummert~\cite{dzhafarov-mummert}. For a historical perspective, one may consider H.~Friedman's contribution~\cite{friedman-icm} to the International Congress of Mathematicians~1974.

Analysis over the real numbers and more general spaces has always been a main focus of reverse mathematics. To complement the many classical examples that can be found in the book by Simp\-son, we mention recent investigations into Ekeland's variational principle~\cite{ekeland} and Caristi's fixed point theorem~\cite{caristi}.

Over sufficiently strong axiom systems, different representations of the real numbers are equivalent (as in everyday mathematics). However, large parts of analysis can be developed in systems that are so weak that the representation matters. Specifically, one commonly works over a basic axiom system $\mathsf{RCA}_0$ (`recursive comprehension axiom'), which can be seen as a system of computable mathematics (though this view should be treated with some caution). Over this system, the real numbers are classically represented by Cauchy sequences with a rate (also called modulus) of convergence. The specific choice of rate is a matter of convenience. Simpson (see Definition~II.4.4 of~\cite{simpson09}) defines a real number as a sequence~$(q_n)=(q_n)_{n\in\mathbb N}$ of rationals with
\begin{equation*}
    |q_m-q_n|\leq 2^{-m}\quad\text{for}\quad m\leq n.
\end{equation*}
In this paper, we instead represent real numbers by sequences of rationals that are Cauchy but do not come with a rate, i.e., we simply demand
\begin{equation*}
    \forall\varepsilon>0\exists N\forall m,n\geq N:|q_m-q_n|<\varepsilon.
\end{equation*}
One can of course take~$\varepsilon$ to be rational (see the next section for official definitions). We sometimes speak of slow Cauchy sequences in order to emphasize that no rate is given. Conversely, we call a Cauchy sequence fast when it comes with a rate as in the classical approach.

At first, the omission of the rate may seem incredibly naive (and indeed our slow Cauchy sequences are called naive Cauchy sequences in some sources~\cite{kreitz-weihrauch}). On the one hand, what we call slow Cauchy sequences is probably the most common choice outside of logic (and Dedekind cuts are probably the most common alternative). One would expect that their status is fully clarified from a logical perspective. As far as we could determine, this is only partially true. We will discuss related work at the end of this introduction.

On the other hand, Cauchy sequences with a rate have long been the representation of choice across much of constructive mathematics~\cite{bishop} and computable analysis~\cite{pour-el-richards,weihrauch-book}.\footnote{It should be said, however, that the picture is not entirely clear. For example, Geuvers, Niqui, Spitters and Wiedijk~\cite{geuvers} write: ``Although this [i.e., the representation of real numbers by slow Cauchy sequences] is quite inefficient, its theoretical importance and its suitability for formalisation has made this representation the basis of the first full implementation of constructive real numbers in a proof assistant [\dots]."} In view of this, it is \emph{prima facie} plausible that the representation with rate is more suitable over weak axiom systems with limited access to non-computable sets. We argue that this, too, is only partially true. 

According to the classical picture with rates, a fair amount of analysis can be done in the axiom system~$\mathsf{WKL}_0$ (`weak K\H{o}nig's lemma'), which is $\Pi^1_1$-conservative over (i.e., proves the same $\Pi^1_1$-statements as) the base system~$\mathsf{RCA}_0$. One can argue that this makes it finitistically reducible in the sense of Hilbert's program. However, still classically with rates, the monotone convergence principle and other results that involve sequential compactness require the axiom system $\mathsf{ACA}_0$ (`arithmetical comprehension axiom'), which is considerably stronger. To some extent, this suggests that the fall of Hilbert's program through G\"odel's incompleteness theorems is already witnessed in elementary analysis.

Any monotone and bounded sequence of rationals is a slow Cauchy sequence. This observation is certainly well-known. So to try and build a theory on it is again naive. To our own surprise, however, that naive endeavor has turned out remarkably successful: We went through all the one-dimensional real analysis that is analyzed in Simpson's classical book on reverse mathematics~\cite{simpson09}.\footnote{To keep this initial investigation manageable, we have not looked at results that involve several real variables or more general spaces. Only future work will tell whether slow Cauchy sequences are still successful in these situations. Even if they are not, the one-dimensional real case seems substantial enough to justify our foundational conclusions.} For slow Cauchy sequences, the large majority of results -- now including theorems about sequential compactness like Bolzano-Weierstrass and even Arzel\`a-Ascoli -- is provable in systems that are conservative over~$\mathsf{RCA}_0$. 

Before we give a more detailed picture of our results, let us mention that the representation by slow Cauchy sequences is not the only modification. In its wake, we need to adapt other definitions. When it comes to continuous functions, our new representation (via the values on rational arguments; see Definition~\ref{def:continuous}) is arguably more intuitive than the classical one (via coherent systems of open balls; see Definition~II.6.1 in~\cite{simpson09}). For open sets, our definition is slightly more complicated but probably uncontroversial (compare Definition~II.5.6 in~\cite{simpson09} with our Definition~\ref{def:open}). The most noteworthy aspect is our representation of sequences of reals (and similarly of sequences of real-valued functions; see Definitions~\ref{def:seq} and~\ref{def:seq-fct}). These are given by double sequences~$(x_{in})_{i,n\in\mathbb N}$ of rationals such that the reals $x_i=(x_{in})_{n\in\mathbb N}$ are uniformly Cauchy, i.e., such that we have
\begin{equation*}
    \forall\varepsilon>0\exists N\forall i\forall m,n\geq N:|x_{im}-x_{in}|<\varepsilon.
\end{equation*}
Let us emphasize that this does not require us to provide a rate. In particular, when $x$ is any real in our sense, setting $x_i=x$ for all~$i\in\mathbb N$ yields a (constant) sequence $(x_i)_{i\in\mathbb N}$ of reals. So it is not the case that the uniformity condition reintroduces the classical representation through the back door. Nevertheless, one can challenge our representation of sequences on both philosophical and mathematical grounds. We respond to both challenges in turn.

On the philosophical side, one may argue that our uniformity condition on sequences goes against Simpson's objection to `extra data':
\begin{quote}
    ``The typical constructivist response to a nonconstructive mathematical theorem is to modify the theorem by adding hypotheses or `extra data'. In contrast, our approach in this book is to analyze the provability of mathematical theorems as they stand, passing to stronger subsystems of $Z_2$ if necessary." (from Remark~I.8.9 of~\cite{simpson09})
\end{quote}
In the framework of reverse mathematics, we have a canonical way to represent sequences of sets~$A_n\subseteq\mathbb N$, namely, as $(A_n)_{n\in\mathbb N}=\{(m,n):m\in A_n\}\subseteq\mathbb N$ with coded pairs~$(m,n)$. One could maintain that, to analyze mathematics as it stands, we are committed to representing sequences of reals in this canonical way and without the additional assumption of a common rate.

In defense, we first note that the classical representation of reals itself adds data in the form of a rate. This does not necessarily contradict Simpson's aims: Mathematics `as it stands' does not typically commit to a specific representation of the real numbers (except in the context of teaching). One can argue that this gives us the freedom to pick one. But then a similar point can be made for sequences: It seems that mathematicians do not typically care whether sequences of reals are represented in the usual set-theoretic way or by different means, as long as one can work with them. In our view, further justification comes from the program of strict reverse mathematics that has recently been promoted by Friedman~\cite{friedman-emergence,friedman-strict-real}. Here one avoids coding at the cost of multiple sorts, and Friedman explicitly mentions separate sorts for reals and sequences of reals. Now a formal interpretation from the strict into the classical setting may act independently on each sort. In particular, we are free to impose our uniformity condition when we interpret the sort of sequences. A definite framework for the strict reverse mathematics of real analysis does not yet exist. But once it exists, our results should show (by the indicated interpretation) that in strict reverse mathematics, too, much of elementary analysis must stay well below arithmetical comprehension.

Even if the reader should feel that we do not analyze mathematics `as it stands', we can argue that formal interpretations yield bounds on consistency strength. So our results reveal that a substantial part of analysis has lower consistency strength than the classical approach may suggest. They show that, despite G\"odel's theorems, Hilbert's program succeeds for much of elementary analysis, which includes notable sequential results like the Arzel\`a-Ascoli theorem.

By responding to the philosophical challenge, we have created a mathematical one: We need to demonstrate that our representation of sequences (and of other relevant objects like continuous functions and open sets) is indeed suitable for a formalization of real analysis in weak axiom systems. Here it is not sufficient to show that a few isolated theorems (like the monotone convergence principle) become weak: It seems obvious that this can be achieved by tweaking some definition. To make a non-trivial point, we need to integrate a substantial body of results. On the philosophical side, this fits with our focus on interpretations between theories.

In order to respond to the mathematical challenge, we now summarize the results that are proved in the present paper. When reals are represented by slow Cauchy sequences (and other representations are adapted as indicated above), the following are some of the results that can be proved in~$\mathsf{RCA}_0$:
\begin{itemize}[label=--]
    \item the reals form an Archimedean ordered field (Proposition~\ref{prop:archimedean}),
    \item arbitrary unions of open sets are open (Lemma~\ref{lem:unions}),
    \item preimages of open sets under continuous functions are open, and any open set can be realized as a preimage of $(0,\infty)$ (Proposition~\ref{pro:ContinuousOpen}),
    \item Urysohn's lemma (Lemma~\ref{lem:Urysohn}),
    \item any Cauchy sequence of reals converges (Lemma~\ref{lem:cauchy-converges}),
    \item any bounded sequence of reals has a supremum (Proposition~\ref{prop:sup}),
    \item the monotone convergence principle (Corollary~\ref{pro:monotoneConvergence}),
    \item $\mathbb R$ is nested-interval complete (Corollary~\ref{cor:nestedIntervals}).
\end{itemize}
Most results that are not provable in~$\mathsf{RCA}_0$ belong to one of two clusters. The first cluster of equivalences over~$\mathsf{RCA}_0$ includes
\begin{itemize}[label=--]
    \item the infinite pigeonhole principle~$\mathsf{IPP}$ (recalled before Lemma~\ref{lem:card-surjection}),
    \item any finite set of reals is bounded / has a maximum (Lemma~\ref{lem:fin-sets-bounded}),
    \item the intermediate value theorem (Theorem~\ref{thm:ivt}),
    \item finite intersections of open sets are open (Proposition~\ref{prop:intersections}),
    \item $\mathbb R$ is connected (Proposition~\ref{prop:connected}).
\end{itemize}
In the second cluster, we have
\begin{itemize}[label=--]
\item the strong cohesive principle (recalled before Proposition~\ref{prop:coh-cads}),
\item the Heine-Borel theorem (Theorem~\ref{thm:Heine-Borel}),
\item any continuous function $f\colon[0,1]\to\mathbb R$ is uniformly continuous / is bounded~/ assumes its extrema (Theorem~\ref{thm:StCADSContinuous}),
\item any continuous function $f\colon[0,1]\to\mathbb R$ is Riemann integrable (where the equivalence is shown over $\mathsf{RCA}_0+\mathsf{IPP}$; Theorem~\ref{thm:RiemannInteg}),
\item the Weierstrass approximation theorem (over $\mathsf{RCA}_0+\mathsf{IPP}$; Theorem~\ref{thm:weierstrass-approx}),
\item the Bolzano-Weierstrass theorem (Theorem~\ref{thm:BolzanoWeierstrass}),
\item the Arzel\`a-Ascoli theorem (Theorem~\ref{thm:arzela}).
\end{itemize}
There are some outliers and some results for which we have not determined the exact strength:
\begin{itemize}[label=--]
    \item The statement that any dense set of reals is infinite is, by Proposition~\ref{prop:dense-inf}, equivalent to the $\Sigma^0_2$-cardinality principle $\mathsf C\Sigma^0_2$ (which is weaker than $\mathsf{IPP}$; see the paragraph before Lemma~\ref{lem:card-surjection}).
    \item The statement that $\mathbb R$ is uncountable is $\Pi^1_1$-conservative over $\mathsf{RCA}_0$ and (hence strictly) weaker than $\mathsf{IPP}$ (Propositions~\ref{pro:cantor} and~\ref{prop:cantor_conservative}). We do not know whether it is provable in $\mathsf{RCA}_0$.
    \item The Baire category theorem can be derived if we assume either $\Sigma^0_2$-induction or the $\Pi^0_1$-genericity principle~$\Pi^0_1\mathsf G$ (Proposition~\ref{prop:Baire-category}). It thus implies neither of these principles and is $\mathsf r\Pi^1_2$-conservative over~$\mathsf{RCA}_0$ (but not necessarily over $\mathsf{RCA}_0+\mathsf{IPP}$; Corollary~\ref{prop:Baire-category}).
    \item The Tietze extension theorem can be derived by $\Sigma^0_2$-induction (see Propo\-sition~\ref{prop:tietze}), but we do not know its exact strength.
    \item The statement that any bounded sequence in~$\mathbb R$ has a limit superior implies $\Sigma^0_2$-induction, and we have no proof below~$\mathsf{ACA}_0$, though $\Sigma^0_2$-induction suffices to get arbitrarily good approximations (Proposition~\ref{prop:limsup-approx}).
\end{itemize}
Let us note that some theorems become slightly stronger than under the classical approach with rates. There, e.g., the intermediate value theorem is provable in~$\mathsf{RCA}_0$ (see Theorem~II.6.6 of~\cite{simpson09}). One could say that we pay an initial prize by making these theorems computably false. As a reward, the overall consistency strength is lowered considerably: As the strong cohesive principle is $\Pi^0_3$-conservative over~$\mathsf{RCA}_0$ (see~\cite{hirschfeldt-shore} and the explanations in Section~\ref{sect:cohesive}), the same holds for the large body of results in the list above. We emphasize that this list includes the theorems of Bolzano-Weierstrass and Arzel\`a-Ascoli, which are classically equivalent to the much stronger principle of arithmetical comprehension (see Theorems~III.2.2 and~III.2.9 of~\cite{simpson09}). In our view, this constitutes a convincing response to the mathematical challenge that was set out above.

Our approach is limited in at least two ways. First, sequential (also called uniform) versions of results are sometimes considerably stronger than in the classical setting. For example, the sequential version of the Heine-Borel theorem states that any family of open coverings admits a family of finite subcoverings (while the non-sequential version is about a single covering). Under the classical approach, both versions are equivalent to weak K\H{o}nig's lemma (see Theorems~IV.1.2 and~IV.1.6 of~\cite{simpson09}), while we need strong cohesiveness for the non-sequential but arithmetical comprehension for the sequential version (Proposition~\ref{prop:heine-borel-unif}). One can debate how relevant the sequential versions are. In our view, a good test is whether they are needed to prove non-sequential versions of other results. Simpson uses the sequential Heine-Borel theorem to prove that any continuous function on~$[0,1]$ is uniformly continuous (see the proof of Theorem~IV.2.2 in~\cite{simpson09}). We still manage to prove the latter result under strong cohesiveness, even though the sequential Heine-Borel theorem is not available to us. So at least in this case, the strength of the sequential result can be contained.

The second limitation brings us back to our treatment of sequences of reals, which we require to be uniformly Cauchy. This requirement is to our advantage when a result (like the Bolzano-Weierstrass theorem) provides us with a sequence. But it presents a challenge when it is us who need to provide a sequence in order to apply a theorem. To make this concrete, let us note that
\begin{itemize}[label=--]
    \item we can construct the sequence $(f(x_n))_{n\in\mathbb N}$ in $\mathsf{RCA}_0+\mathsf{StCOH}$ for a continuous function $f\colon\mathbb R\to\mathbb R$ and a bounded sequence $(x_n)_{n\in\mathbb N}$ of reals, but not necessarily for an unbounded sequence (Proposition~\ref{lem:fct-to-seq}),
    \item for any real~$x$, we can construct $(f_n(x))_{n\in\mathbb N}$ over~$\mathsf{RCA}_0$ when the functions~$f_n$ are pointwise equicontinuous (see Definition~\ref{def:seq-fct} and Corollary~\ref{cor:equicont-seq}),
    \item as an instance of the previous point, we can construct the sequence $(x^n)_{n\in\mathbb N}$ in $\mathsf{RCA}_0$ when $x$ is a real in~$[-1,1]$ but not in general (Example~\ref{ex:x^n}),
    \item we can construct the sequence $(f^n(x))$ of iterates over $\mathsf{RCA}_0+\mathsf{StCOH}$ when $f$ is a strong contraction (see Proposition~\ref{prop:Banach}, which derives the Banach fixed point theorem from our general convergence results), while the general case requires arithmetical comprehension (Proposition~\ref{prop:iterates-ACA}),
    \item over $\mathsf{RCA}_0$, we can prove a version of Caristi's theorem by mimicking the usual proof, which iterates a function to construct a sequence of reals (see Theorem~\ref{thm:caristi} and Remark~\ref{rmk:semi-cont}).
\end{itemize}
One may argue that sequences are most interesting when they converge. From this viewpoint, the issue with $(x^n)_{n\in\mathbb N}$ for $|x|>1$ is somewhat awkward but of little consequence. The results that we have listed above show that we can construct sequences of reals in several relevant cases. It would be desirable to do more case studies in future work: These could strengthen or weaken our claim that sequences of reals can be handled in weak axiom systems.

Based on our mathematical results, we have argued that a substantial part of analysis can be accommodated in axiom systems that are conservative over $\mathsf{RCA}_0$. To a reader who rejects this argument, we can still offer the following:
\begin{itemize}[label=--]
    \item In recent years, the focus of reverse mathematics has, arguably, shifted from analysis to combinatorics. This has led to new axioms in the so-called reverse mathematics zoo. These new axioms are incomparable with weak K\H{o}nig's lemma, which is classically equivalent to the open-cover compactness of~$[0,1]\subseteq\mathbb N$. With our new approach, open-cover compactness becomes equivalent to the strong cohesive principle, which thus unites the analytical and the combinatorial realm (see Section~\ref{sect:cohesive} for background on the combinatorial side).
    \item Related to the previous point, our approach brings out the combinatorial aspects of some analytical theorems. For example, we have mentioned that the monotone convergence principle is provable in $\mathsf{RCA}_0$ while the Bolzano-Weierstrass theorem requires strong cohesiveness. This conforms with the mathematical intuition that the proof of the latter is more involved. Classically, both results are equivalent to arithmetical comprehension.
    \item Under our approach, open-cover compactness and sequential compactness are both equivalent to strong cohesiveness. So the notion of compactness -- which is at the very heart of mathematical analysis -- becomes more robust.
    \item Our results yield additional guardrails for the emerging program of strict reverse mathematics (see above and~\cite{friedman-emergence}). Namely, if a body of results is weak under our approach, any strict formalization should also be weak. Conversely, if one wants to find logical strength in strict reverse mathematics, one needs to draw on resources that have strength in our framework (e.g., via an axiom that insists on the existence of certain sequences in~$\mathbb R$).
    \item We present new arguments that may be fruitful in different contexts. In particular, we point the reader to the proof of Proposition~\ref{prop:cantor_conservative}, which combines hyperimmunity with metastability (as used in proof mining).
\end{itemize}

Some of our findings resemble results that have been obtained in different settings. In particular, U.~Kohlenbach~\cite{kohlenbach-pra} has shown that parameter-free versions (i.e., single applications) of the monotone convergence principle as well as the Bolzano-Weierstrass and Arzel\`a-Ascoli theorems (formalized in higher-order arithmetic) are $\Pi^0_3$-conservative over primitive recursive arithmetic. In contrast, the parameter-free version of the limit superior principle entails and is $\Pi^0_4$-conservative over $\Sigma^0_2$-induction. This distinction conforms with experience from proof mining. It is interesting that we can make the same distiction, while the classical approach identifies all principles with arithmetical comprehension. 

In the case of Kohlenbach's results, it is not clear whether the similarity with our findings is a coincidence, since we allow parameters and Kohlenbach uses Cauchy sequences with rate, which seem to be orthogonal modifications. The connection is much clearer for results of A.~Kreuzer. The latter has proved~\cite{kreuzer-bolzano} that the strong cohesive principle is equivalent to a version of the Bolzano-Weierstrass theorem, where the input is a sequence of reals with rate and the output is only a slow Cauchy sequence. It had previously been shown by S.~Le~Roux and M.~Zieg\-ler~\cite{le-roux-ziegler} that the slow Cauchy sequence in the output may need to be non-computable. As an isolated result, the analysis of the Bolzano-Weierstrass theorem for slow Cauchy sequences is thus essentially due to Kreuzer (though the latter assumes a rate while we only require uniformity). The point of our paper is that we integrate the result into a coherent foundational picture, in which slow Cauchy sequences are used consistently in the input and output of all theorems, and where the definitions of open sets and continuous functions are adapted accordingly. Kreuzer has also proved~\cite{kreuzer-ascoli} that a version of the Arzel\`a-Ascoli theorem is equivalent to the conjunction of the strong cohesive principle and weak K\H{o}nig's lemma. This is of course similar to but not quite comparable with our result, since Kreuzer works with the traditional representation of continuous functions, which forces him to assume $\mathsf{WKL}$ as well.

In reverse mathematics, J.~Hirst has determined the minimal axioms that are needed to (uniformly) convert different representations of the real numbers into each other~\cite{hirst-reals}. He considered Cauchy sequences with rate, decimal expansions (i.e., sequences with rate that are additionally monotone) and Dedekind cuts with and without endpoints. Slow Cauchy sequences are not considered, possibly because it is obvious that conversions from these require arithmetical comprehension (see Remark~\ref{rem:FastCauchyACA}). Also, the paper by Hirst does not explore how elementary analysis plays out with the different representations.

A systematic picture of different representations has been obtained in computable analysis. Of particular relevance for us are results that go back to C.-K.~Ho~\cite{ho-slow-jump}, who proved that a real number is represented by a computable slow Cauchy sequence precisely if it is represented by a jump computable sequence with rate. An analogous result for the space of continuous functions has been proved by Ziegler~\cite{ziegler-type-2}, while V.~Brattka~\cite{brattka-galois} presented an even more general Galois connection. The latter could possibly explain our results in a more systematic way or even reduce some of them to previously known facts. Whether it really does is not straightforward to determine -- not least since we use new definitions of open sets and continuous functions -- but it is desirable to explore this in future work.

To avoid misunderstanding, we emphasize that our new approach is intended to complement but not to supersede the classical one. Cauchy sequences with rate have desirable properties from several perspectives. In computable analysis, each choice of representation induces a topology on the reals. Cauchy sequences without rate produce the trivial topology, while the usual Euclidian topology arises from representations with rate (see~\cite{kreitz-weihrauch}). In less technical terms, no finite part of a slow Cauchy sequence provides any information about the real number that is represented. In contrast, sequences with rate allow for effective approximations, which is clearly relevant both on a theoretical level and in terms of applications. For these reasons alone, the reverse mathematics of Cauchy sequences with rate is as relevant as ever. Also, we have mentioned a few examples of sequential results where the traditional approach fares better than the approach without rates. It is possible that more striking examples will be discovered in the future. Our aim is simply to show that a coherent alternative is possible, which opens up foundational possibilities that may have appeared to be blocked.

\addtocontents{toc}{\SkipTocEntry}
\subsection*{\bf Terminology}
As mentioned above, we require sequences of reals to satisfy a certain uniformity condition. To interpret all our claims correctly, the reader should note that we systematically distinguish between families and sequences: By a family, we simply mean a collection of sets (coded into a single set), which does not need to be uniform (consider, e.g., Proposition~\ref{pro:cantor}). In contrast, sequences of reals and continuous functions are always assumed to satisfy the uniformity conditions from Definitions~\ref{def:seq} and~\ref{def:seq-fct}, respectively.

\addtocontents{toc}{\SkipTocEntry}
\subsection*{\bf Acknowledgements}
We are very grateful for generous support from colleagues, who have provided invaluable advice, encouragement or information: in particular, we want to thank Vasco Brattka, Denis Hirschfeldt, Jeffry Hirst, Ulrich Kohlenbach, Ludovic Patey, Peter Schuster, Stephen Simpson and Martin A.~Ziegler.

\section{Fundamental properties of the reals}

In this section, we state our definition of the real numbers and discuss basic operations and relations on them. We also investigate cardinality questions in relation to the $\Sigma^0_2$-cardinality principle and the infinite pigeonhole principle.

The following definition is familiar from elementary analysis courses. In reverse mathematics, Simpson uses it over $\mathsf{ACA}_0$. The point is that we will use the same definition over the weaker base theory~$\mathsf{RCA}_0$, where the two are not equivalent (see the introduction and compare Definitions~I.4.2 and II.4.4 in~\cite{simpson09}).

\begin{definition}
    The class~$\mathbb R$ of reals consists of all Cauchy sequences of rationals, i.e., of all sequences~$(x_n)=(x_n)_{n\in\mathbb N}$ with $x_n\in\mathbb Q$ such that each rational~$\varepsilon>0$ admits an~$N\in\mathbb N$ with $|x_m-x_n|<\varepsilon$ for all~$m,n\geq N$.
\end{definition}

As mentioned in the introduction, we sometimes say that our reals are represented by slow Cauchy sequences. This is supposed to emphasize the difference with the classical representation by `fast' Cauchy sequences with rate. The following remark shows that arithmetical comprehension is needed to convert slow Cauchy sequences into fast ones. While one may be tempted to read this as a negative result, one can give it a positive twist: As slow Cauchy sequences can absorb so much logical strength, we may hope that they allow us to develop analysis over a much weaker theory. Over the course of this paper, we show that this hope materializes.

\begin{remark}\label{rem:FastCauchyACA}
Within $\mathsf{ACA}_0$, every Cauchy sequence $(q_n)$ in $\mathbb{Q}$ has a rate of convergence, i.e., there exists a function $r\colon\mathbb{Q}\to\mathbb{N}$ such that $\vert q_m-q_n\vert<\varepsilon$ holds for all $\varepsilon>0$ and all $m,n\geq r(\varepsilon)$. In particular, we can immediately use this $r$ to define a subsequence (which hence has the same limit) with Cauchy rate $2^{-n}$. Conversely, any injection $f\colon\mathbb N\to\mathbb N$ gives rise to rationals
\begin{equation*}
    q_n:=\sum_{i=0}^n 2^{-f(i)-1}
\end{equation*}
that form a Specker-type sequence (see~\cite{specker} as well as \cite{troelstra-van-dalen} for our specific construction). By the formula for the geometric series, it is immediate that the non-decreasing sequence $(q_n)$ is bounded. It must thus be Cauchy. Indeed, if this was false, we would have an $\varepsilon>0$ such that any $m\in\mathbb N$ admits an $n>m$ such that we have $q_n>q_m+\varepsilon$. But then our sequence would grow beyond any bound. However, any rate of convergence for $(q_n)$ allows us to compute the image of~$f$. Hence over~$\mathsf{RCA}_0$, arithmetical comprehension follows if any Cauchy sequence of rationals admits a rate and in particular if any slow Cauchy sequence admits a fast Cauchy sequence with the same limit (cf.~Lemma~III.1.3 of~\cite{simpson09}).
\end{remark}

We continue with basic definitions and results.

\begin{definition}\label{def:reals-ineq}
    For reals~$x=(x_n)$ and $y=(y_n)$, we write $x\leq y$ when every rational~$\varepsilon>0$ admits an~$N\in\mathbb N$ such that $x_n\leq y_n+\varepsilon$ holds for all~$n\geq N$. When we have both $x\leq y$ and $y\leq x$, we write~$x=y$.\footnote{One may write $=_{\mathbb R}$ to ensure that this is not confounded with equality as Cauchy sequences (where the latter amounts to~$x_n=y_n$ for all~$n\in\mathbb N$), though this is rarely necessary in practice.} We write $x<y$ when we have $x\leq y$ but not~$x=y$.
\end{definition}

Parts~(a) and~(b) of the following show that we have a total preorder. Due to parts~(c) and~(d), we get a total order if we quotient out equality. In the framework of reverse mathematics, the quotient is, of course, not represented as a set.

\begin{lemma}[$\mathsf{RCA}_0$]\label{lem:reals-order}
(a) The relation $\leq$ on~$\mathbb R$ is reflexive and transitive. 

(b) For any $x,y\in\mathbb R$, we have $x\leq y$ or~$y\leq x$.

(c) The equality on $\mathbb R$ is an equivalence relation.

(d) If we have $x=x'$ and $y=y'$, then $x\leq y$ is equivalent to $x'\leq y'$.
\end{lemma}
\begin{proof}
    The arguments from an elementary analysis course go through in~$\mathsf{RCA}_0$. We only prove a strong form of~(b), because we want to refer to it later. Let us write $x=(x_n)$ and~$y=(y_n)$. Assuming~$y\not\leq x$, we have an~$\varepsilon>0$ such that each~$N\in\mathbb N$ admits an~$n>N$ with~$y_n>x_n+\varepsilon$. Since our sequences are Cauchy, we now find an~$N\in\mathbb N$ such that all~$m,n\geq N$ validate
\begin{equation*}
    x_m-y_m\leq|x_m-x_n|+(x_n-y_n)+|y_n-y_m|\leq(x_n-y_n)+\frac\varepsilon2.
\end{equation*}
If we choose the right~$n>N$, we get $x_n-y_n<-\varepsilon$ from above. For any~$m\geq N$, we thus have $x_m-y_m<-\varepsilon/2$ and hence~$x_m+\varepsilon/2<y_m$. In particular, this means that we have~$x\leq y$.
\end{proof}

From the previous lemma, one readily infers that $<$ is a strict total order modulo equality. Also, it follows that $x<y$ holds precisely when~$y\leq x$ fails.

\begin{lemma}[$\mathsf{RCA}_0$]\label{lem:reals-ineq-Sigma02}
For reals $x=(x_n)$ and $y=(y_n)$, we have $x<y$ precisely if there is a rational~$\delta>0$ and an~$N\in\mathbb N$ such that all $n\geq N$ validate~$x_n+\delta<y_n$.
\end{lemma}
\begin{proof}
    The forward direction was established in the previous proof. In the other direction, one readily derives~$y\not\leq x$.
\end{proof}

In our case, being a real (i.e., a slow Cauchy sequence) is~$\Pi^0_3$. The previous lemma shows that the order relations $<$ and $\leq$ between our reals are $\Sigma^0_2$ and $\Pi^0_2$, respectively. This stands in contrast with the classical approach, where reals are represented by Cauchy sequences with some fixed rate (given by a computable function). Under this approach, being a real is $\Pi^0_1$, while $<$ and $\leq$ are $\Sigma^0_1$ and~$\Pi^0_1$, respectively. The reduction in quantifier complexity is certainly a motivation for the classical choice. However, our results show that it does not decrease the overall consistency strength --- if anything, the opposite is true. 

We now consider the real numbers as a field. To see that the following operations preserve Cauchy sequences, it suffices to recall the usual proofs that addition and multiplication are continuous. Analogous to the following definition and proposition, one obtains the absolute value function and its fundamental properties.

\begin{definition}
Addition and multiplication on~$\mathbb R$ act pointwise on the representing Cauchy sequences, i.e., we put $(x_n)+(y_n)=(x_n+y_n)$ and $(x_n)\cdot(y_n)=(x_n\cdot y_n)$.
\end{definition}

In the following result, the real numbers are considered modulo equality. That addition and multiplication are compatible with equality is verified in the proof. We will later use the infinite pigeonhole principle to prove the intermediate value theorem. It will then follow that~$\mathbb R$ is real closed (see Corollary~\ref{cor:real-closed}).

\begin{proposition}[$\mathsf{RCA}_0$]\label{prop:archimedean}
The reals form an Archimedean ordered field.
\end{proposition}
\begin{proof}
Even with respect to the pointwise equality of Cauchy sequences, we have a commutative ring, as the relevant properties are directly inherited from the rationals. We now show that $x=(x_n)$ has an inverse if $x\neq 0$ holds with respect to the equality of reals. By Lemma~\ref{lem:reals-order}, we have $x>0$ or~$x<0$. Lemma~\ref{lem:reals-ineq-Sigma02} ensures that $|x_n|$ is bounded away from zero for all sufficiently large~$n$. Thus we get a Cauchy sequence $(x_n')$ if we set $x_n'=1/x_n$ for~$x_n\neq 0$ and choose arbitrary $x_n'$ otherwise. Clearly, $(x_n')$ is the multiplicative inverse of~$x$.

It is straightforward to check that addition is monotone and, as a consequence, compatible with equality. To prove the monotonicity properties of multiplication, we show that $x,y\geq 0$ implies~$x\cdot y\geq 0$. As the sequences $x=(x_n)$ and~$y=(y_n)$ are Cauchy, we find a bound~$B$ such that $|x_n|,|y_n|\leq B$ holds for all~$n\in\mathbb N$. Given an arbitrary~$\varepsilon>0$, we find an~$N\in\mathbb N$ such that all~$n\geq N$ validate~\mbox{$x_n,y_n\geq-\varepsilon/B$}. The latter entails $x_n\cdot y_n\geq-\varepsilon$, as $x_n<0$ implies $|x_n|\leq\varepsilon/B$ and hence~\mbox{$|x_n\cdot y_n|\leq\varepsilon$}. So we indeed have $x\cdot y\geq 0$. As in any ordered field, we learn that $y\leq y'$ implies\linebreak $x\cdot y\leq x\cdot y'$ for~$x\geq 0$ and $x\cdot y\geq x\cdot y'$ for~$x<0$. In particular, multi\-plication is compatible with equality. The Archimedean property is inherited from the rationals, as Cauchy sequences are bounded.
\end{proof}

In the remainder of this section, we consider some questions of cardinality. The following principles beyond $\mathsf{RCA}_0$ will occur. The infinite pigeonhole principle (which was already mentioned above) asserts the following:
    \begin{equation*}\tag{$\ipp$}
        \parbox{.87\textwidth}{For any $N\in\mathbb N$ and any function~$c\colon\mathbb N\to\{0,\ldots,N\}$, there is an $n\leq N$ such that $\{i\in\mathbb N:c(i)=n\}$ is infinite.}
    \end{equation*}
As shown by Hirst (see Theorem~6.4 of~\cite{hirst-thesis}), the infinite pigeonhole principle is equivalent to the $\Sigma^0_2$-bounding principle~$\mathsf B\Sigma^0_2$ (also known as $\Sigma^0_2$-collection):
\begin{equation*}\tag{$\mathsf B\Sigma^0_2$}
\parbox{.87\textwidth}{For any $\Sigma^0_2$-relation~$R$ such that each $m\leq a$ admits an~$n$ with~$(m,n)\in R$, there is a~$b$ such that each $m\leq a$ admits an~$n\leq b$ with~$(m,n)\in R$.}
\end{equation*}
The $\Sigma^0_2$-cardinality principle, considered by Seetapun and Slaman~\cite{seetapun-slaman}, postulates:
    \begin{equation*}\tag{$\mathsf C\Sigma^0_2$}
        \parbox{.87\textwidth}{There is no $\Sigma^0_2$-injection $C\colon\mathbb N\to\{0,\ldots,N\}$ for $N\in\mathbb N$.}
    \end{equation*}
    Here $C$ is not given as a set. Instead, we appeal to a partial truth definition in order to quantify over $\Sigma^0_2$-formulas that define the graphs of injective functions. It is known that $\mathsf C\Sigma^0_2$ is unprovable in $\mathsf{RCA}_0$. Over the latter, $\mathsf C\Sigma^0_2$ follows from but does not imply~$\ipp$ (see Section~3.1 of~\cite{conidis-slaman}). As shown by H.~Friedman and independently Paris, $\mathsf B\Sigma^0_2$ and hence~$\ipp$ is $\Pi^0_3$-conservative over $\mathsf{RCA}_0$ (see e.g.\ Theorem~IV.1.59 of~\cite{hajek91}).

\begin{lemma}\label{lem:card-surjection}
    The following are equivalent over~$\mathsf{RCA}_0$:
    \begin{enumerate}[label=(\roman*)]
        \item The $\Sigma^0_2$-cardinality principle $\mathsf C\Sigma^0_2$.
        \item There is no $\Sigma^0_2$-surjection~$D\colon\{0,\ldots,N\}\to\mathbb N$ for $N\in\mathbb N$.
    \end{enumerate}
\end{lemma}
\begin{proof}
    First assume there is a $D$ as in~(ii). Consider a $\Pi^0_1$-relation $D(n)=_a m$ with
    \begin{equation*}
        D(n)=m\quad\Leftrightarrow\quad D(n)=_am\text{ for some }a\in\mathbb N.
    \end{equation*}
    We define a $\Sigma^0_2$-function $C_0\colon\mathbb N\to\{0,\ldots,N\}\times\mathbb N$ by
    \begin{equation*}
        C_0(m)=\min\{(n,a):D(n)=_a m\},
    \end{equation*}
    where we minimize over codes of pairs. Let $C(m)$ be the first component of~$C(m)$. This defines a $\Sigma^0_2$-injection $C\colon\mathbb N\to\{0,\ldots,N\}$ with left inverse~$D$.

    Conversely, assume that we have an injection $C\colon\mathbb N\to\{0,\ldots,K\}$ with $K\in\mathbb N$. Consider a $\Pi^0_1$-relation $C(n)=_a k$ with
    \begin{equation*}
        C(n)=k\quad\Leftrightarrow\quad C(n)=_a k\text{ for some }a\in\mathbb N.
    \end{equation*}
    We may assume that we have
    \begin{equation*}
        C(n)=_a k\quad\Rightarrow\quad C(n)=_bk\quad\text{for}\quad a\leq b.
    \end{equation*}
    As noted in the previous paragraph, $\mathsf C\Sigma^0_2$ is weaker than $\mathsf{IPP}$. So we may also consider a function $c\colon\mathbb N\to\{0,\ldots,M\}$ (given as a set) such that each $m\leq M$ admits a~$J\in\mathbb N$ with $c(j)\neq m$ for all~$j\geq J$. This yields an unbounded $\Sigma^0_2$-function $D_0\colon\{0,\ldots,M\}\to\mathbb N$ that is given by
    \begin{equation*}
        D_0(m)=\min\{i\in\mathbb N:c(j)\neq m\text{ for all }j\geq i\}.
    \end{equation*}
    We put $N=(K+1)\cdot(M+1)-1$ and note that each number up to $N$ has a unique representation of the form $(K+1)\cdot m+k$ for $m\leq M$ and $k\leq K$. With respect to this representation, we define a $\Sigma^0_2$-function $D\colon\{0,\ldots,N\}\to\mathbb N$ by
    \begin{equation*}
        D\big((K+1)\cdot m+k\big)=\min\big(\{n<D_0(m):C(n)=_{D_0(m)} k\}\cup\{D_0(m)\}\big).
    \end{equation*}
    To see that $D$ is surjective, consider an arbitrary~$n\in\mathbb N$. Let $k\leq K$ equal $C(n)$ and pick~$a\in\mathbb N$ with $C(n)=_a k$ as well as $m\leq M$ with $a,n<D_0(m)$. Given that $C$ is injective, we get $D((K+1)\cdot m+k)=n$.
\end{proof}

We now come to our first reversal:

\begin{proposition}\label{prop:dense-inf}
The following are equivalent  over~$\mathsf{RCA}_0$:
\begin{enumerate}[label=(\roman*)]
    \item No finite collection~$(x_n)_{n\leq N}$ of reals is dense in the order~$\mathbb R$, i.e., for any such collection there are reals $y<z$ such that $y\leq x_n\leq z$ fails for all~$n<N$.
    \item The $\Sigma^0_2$-cardinality principle $\mathsf C\Sigma^0_2$ holds.
\end{enumerate}
\end{proposition}
\begin{proof}
    To show that~(i) implies~(ii), we assume that $\mathsf C\Sigma^0_2$ fails. By the previous lemma, this yields a $\Sigma^0_2$-surjection~$D\colon\{0,\ldots,N\}\to\mathbb Q$ for some $N\in\mathbb N$. Consider a bounded formula~$\theta$ with
    \begin{equation*}
        D(n)=q\quad\Leftrightarrow\quad\exists a\in\mathbb N\forall b\in\mathbb N:\theta(a,b,n,q).
    \end{equation*}
    For $n<N$ and $i\in\mathbb N$, we set
    \begin{equation*}
        y_{ni}=\min\{(q,a)\,|\,\forall b<i:\theta(a,b,n,q)\},
    \end{equation*}
    where we again minimize over codes for pairs. To see that each sequence $(y_{ni})_{i\in\mathbb N}$ is eventually constant, consider the minimal pair~$(q,a)$ such that $\theta(a,b,n,q)$ holds for all~$b\in\mathbb N$. We find an~$I$ such that each $(q',a')<(q,a)$ admits a $b<I$ for which $\theta(a',b,n,q')$ fails. This yields $y_{ni}=(q,a)$ for all~$i\geq I$. We also learn~$D(n)=q$. So if we define $x_{ni}$ as the first component of~$y_{ni}$, the sequence $x_n=(x_{ni})_{i\in\mathbb N}$ stabilizes with eventual value~$D(n)$. We thus have a finite collection~$(x_n)_{n\leq N}\subseteq\mathbb R$ such that each~$q\in\mathbb Q$ admits an~$n\leq N$ with $x_n=q$ as reals. Given that $\mathbb Q$ is dense (say, via Lemma~\ref{lem:reals-ineq-Sigma02}), this means that~(i) fails.

    Conversely, assume that the reals contain a finite dense family~$(x_n)_{n\leq N}$. Here each~$x_n$ is given as a Cauchy sequence $(x_{ni})_{i\in\mathbb N}$. For every~$m\in\mathbb N$, density yields an $n\leq N$ with $m-1/4\leq x_n\leq m+1/4$ and thus $m-1/3<x_{nj}<m+1/3$ for large~$j$. We thus have a $\Sigma^0_2$-function $C'\colon\mathbb N\to\{0,\ldots,N\}\times\mathbb N$ with
    \begin{equation*}
        C'(m)=\min\left\{(n,i):m-\frac13<x_{nj}<m+\frac13\text{ for all }j\geq i\right\}.
    \end{equation*}
    Consider the $\Sigma^0_2$-function $C\colon\mathbb N\to\{0,\ldots,N\}$ such that $C(m)$ is the first component of~$C'(m)$. For $C(m)=n$ we have $m-1/3\leq x_n\leq m+1/3$, so that $C$ is injective.
\end{proof}

In the following result, the base theory is not optimal, as we will show below. We do not know whether $\mathsf{RCA}_0$ alone can prove the result.

A family~$(x_n)_{n\in\mathbb N}$ of reals is given as a double sequence~$(x_{ni})_{n,i\in\mathbb N}$ of rationals with the only requirement that $x_n=(x_{ni})_{i\in\mathbb N}$ is Cauchy for each~$n\in\mathbb N$ (cf.~the more restrictive notion of sequence in Definition~\ref{def:seq}). 

\begin{proposition}[$\mathsf{RCA}_0+\ipp$]\label{pro:cantor}
The real numbers are uncountable, i.e., for any family $(x_n)_{n\in\mathbb N}$ of reals there is a $y\in\mathbb R$ with $x_n\neq y$ for all~$n\in\mathbb N$.
\end{proposition}
\begin{proof}
We perform a nested interval construction, as in the proof of Theorem~II.4.9 from~\cite{simpson09}. In contrast to the latter, however, we allow that choices are revised finitely many times (when our Cauchy sequences are not yet sufficiently stable).

Our formal construction is based on the Cantor middle third set. To describe the latter, we consider the finite sets of rationals that are recursively defined by
\begin{equation*}
    B_0=\left\{0,\frac23\right\}\quad\text{and}\quad B_{n+1}=B_n\cup\left\{q+\frac2{3^{n+2}}:q\in B_n\right\}.
\end{equation*}
Much of the following relies on the fact that we have
\begin{equation*}
    |q-r|\geq\frac2{3^{n+1}}\quad\text{for }q\neq r\text{ from } B_n.
\end{equation*}
The $n$-th approximation to the middle third set is given as a disjoint union of closed intervals with left endpoint in~$B_n$, namely by
\begin{equation*}
    C_n=\bigcup_{q\in B_n}\left[q,q+\frac1{3^{n+1}}\right].
\end{equation*}
In order to write $B_n=B_n^l\cup B_n^r$ as a disjoint union of left and right branchings, we set $B_0^l=\{0\}$ and~$B_0^r=\{2/3\}$ as well as
\begin{equation*}
    B_{n+1}^l=B_n\quad\text{and}\quad B_{n+1}^r=\left\{q+\frac2{3^{n+2}}:q\in B_n\right\}.
\end{equation*}
We then decompose~$C_n=C_n^l\cup C_n^r$ into
\begin{equation*}
    C_n^l=\bigcup_{q\in B_n^l}\left[q,q+\frac1{3^{n+1}}\right]\quad\text{and}\quad C_n^r=\bigcup_{q\in B_n^r}\left[q,q+\frac1{3^{n+1}}\right].
\end{equation*}
Hulls for these sets can be given as unions of larger open intervals, namely by
\begin{equation*}
    H_n^l=\bigcup_{q\in B_n^l}\left(q-\frac1{2\cdot 3^{n+1}},q+\frac1{2\cdot3^n}\right)\text{ and }
    H_n^r=\bigcup_{q\in B_n^r}\left(q-\frac1{2\cdot 3^{n+1}},q+\frac1{2\cdot3^n}\right).
\end{equation*}
One readily verifies that we have $H_n^l\cap H_n^r=\emptyset$ as well as
    \begin{equation*}
        |p-r|\geq\frac1{2\cdot 3^{n+1}}\quad\text{ for }p\in C_n^l\text{ and }r\notin H_n^l.
    \end{equation*}
The same holds when we have $p\in C_n^r$ and $r\notin H_n^r$.

To find a real number~$y$ that differs from all the~$x_n=(x_{ni})_{i\in\mathbb N}$, we want to put $y$ into $C_n^l$ or~$C_n^r$, respectively, when a certain approximation~$x_{n,i}$ lies in~$H_n^r$ or $H_n^l$. Formally, we define
\begin{equation*}
    i(n,k)=\min\left\{i_0\leq k:|x_{ni}-x_{nj}|\leq\frac1{3^{n+2}}\text{ for }i_0\leq i,j\leq k\right\}.
\end{equation*}
Note that the minimum is taken over a non-empty set (containing at least~$i_0=k$). We now set
\begin{equation*}
    y_k=\sum_{n=0}^k s_{nk}\quad\text{with}\quad s_{nk}=\begin{cases}
    0 & \text{if $x_{n,i(n,k)}\in H_n^r$},\\
    2/3^{n+1} & \text{otherwise}.
    \end{cases}
\end{equation*}
Let us show that $y=(y_k)$ is Cauchy and hence a real. Each $n$ admits an~$M$ such that $|x_{ni}-x_{nj}|\leq 3^{-n-2}$ holds for all~$i,j\geq M$. This~$M$ bounds the values of the non-decreasing function $k\mapsto i(n,k)$. It follows that there is a~$K$ such that we have $i(n,k)=i(n,K)$ and hence $s_{nk}=s_{nK}$ for all~$k\geq K$. So for any~$N$, the pigeonhole principle (again in the form of $\Sigma^0_2$-boundedness) yields a $K\geq N$ with $s_{nk}=s_{nK}$ for all~$n<N$ and~$k\geq K$. When we have $k,l\geq K$, we thus get
\begin{equation*}
    \left|y_k-y_l\right|\leq\sum_{n=0}^{N-1}|s_{nk}-s_{nl}|+\sum_{n=N}^ks_{nk}+\sum_{n=N}^ls_{nl}<\frac2{3^N}.
\end{equation*}
Since $N$ was arbitrary, this confirms that~$(y_k)$ is Cauchy.

Finally, we show that~$y$ differs from each of the~$x_n$. For any~$k$, a straightforward induction on~$i$ yields~$\sum_{m=0}^i s_{mk}\in B_i$. Assuming~$k\geq n$, it follows that $\sum_{m=0}^n s_{mk}$ lies in~$B_n^l$ if we have~$x_{n,i(n,k)}\in H_n^r$ and that it lies in~$B_n^r$ otherwise. By the formula for the geometric series, we have $\sum_{m=n+1}^k s_{mk}<3^{-n-1}$. We thus get
\begin{equation*}
    y_k\in\begin{cases}
        C_n^l & \text{if }x_{n,i(n,k)}\in H_n^r,\\
        C_n^r & \text{otherwise}.
    \end{cases}
\end{equation*}
In the first case, we have $x_{n,i(n,k)}\notin H_n^l$ due to $H_n^l\cap H_n^r=\emptyset$ from above. So in each of the two cases, an observation from above yields
\begin{equation*}
    |y_k-x_{n,i(n,k)}|\geq\frac1{2\cdot 3^{n+1}}.
\end{equation*}
Considering the definition of~$i(n,k)$, we infer
\begin{equation*}
    |y_k-x_{nk}|\geq |y_k-x_{n,i(n,k)}|-|x_{n,i(n,k)}-x_{nk}|\geq\frac1{2\cdot 3^{n+1}}-\frac1{3^{n+2}}=\frac1{2\cdot 3^{n+1}}.
\end{equation*}
Since~$k\geq n$ was arbitrary, this shows~$y\neq x_n$.
\end{proof}

The following proposition will later be complemented by a conservativity result for the Baire category theorem (see Corollary~\ref{cor:BCT}).

\begin{proposition}\label{prop:cantor_conservative}
    The theory
    \begin{equation*}
        \mathsf{RCA}_0+\mathsf{WKL}+\text{`the real numbers are uncountable'}
    \end{equation*}
    is $\Pi^1_1$-conservative over~$\mathsf{RCA}_0$.
\end{proposition}

Before we prove the result, we record the following corollary, which follows because $\mathsf C\Sigma^0_2$ is a $\Pi^1_1$-statement and unprovable in~$\mathsf{RCA}_0$ (see above). In connection with Proposition~\ref{pro:cantor}, we explicitly note that the corollary remains valid with $\mathsf{IPP}$ at the place of the weaker principle $\mathsf C\Sigma^0_2$. 

\begin{corollary}
    We have
    \begin{equation*}
        \mathsf{RCA}_0+\mathsf{WKL}+\text{`the real numbers are uncountable'}\nvdash\mathsf C\Sigma^0_2.
    \end{equation*}
\end{corollary}

\begin{proof}[Proof of Proposition~\ref{prop:cantor_conservative}]
    It suffices to show that any countable model~$\mathcal M\vDash\mathsf{RCA}_0$ has an $\omega$-extension
    \begin{equation*}
        \mathcal N\vDash\mathsf{RCA}_0+\mathsf{WKL}+\text{`the real numbers are uncountable'}.
    \end{equation*}
    We may assume that $\mathcal M$ is topped, i.e., that its second-order part contains one set in which all others are $\Delta^0_1$-definable.
    
    By the hyperimmune-free basis theorem for non-$\omega$-models (see Theorem~2.16 of~\cite{freund-uftring-wkl}), we find an $\omega$-extension $\mathcal N\vDash\mathsf{RCA}_0+\mathsf{WKL}$ that is hyperimmune-free. This means that any $f\colon\mathbb N\to\mathbb N$ from (the second-order part of) the model~$\mathcal N$ is dominated by some $g\colon\mathbb N\to\mathbb N$ from~$\mathcal M$ (i.e., with $f(n)\leq g(n)$ for all~$n\in\mathbb N$, where $\mathbb N$ denotes the joint first-order part of~$\mathcal M$ and~$\mathcal N$). It remains to show that the reals are uncountable according to~$\mathcal N$.

    In $\mathcal N$, we consider an arbitrary family $(x_n)_{n\in\mathbb N}$ of real numbers. The collection of sets that are $\Delta^0_1$-definable from $(x_n)_{n\in\mathbb N}$ and parameters in~$\mathcal M$ form an $\omega$-submodel~$\mathcal M'\subseteq\mathbb N$. Note that $\mathcal M'$ is still topped and must thus violate~$\mathsf{WKL}$ (as there are computable trees without computable paths; see Theorem~VIII.2.15 in~\cite{simpson09}). Thus some set $Y\subseteq\mathbb N$ is contained in~$\mathcal N$ but not in~$\mathcal M'$. Note that the extension of~$\mathcal M'$ into~$\mathcal N$ is still hyperimmune-free. We may thus assume that $\mathcal M'$ equals $\mathcal M$, i.e., that the family $(x_n)_{n\in\mathbb N}$ lies in~$\mathcal M$.

    Let $y=(y_i)\in\mathcal N$ be the real given by
    \begin{equation*}
        y_i=\sum_{j<i}\frac{\delta_j}{3^j}\quad\text{with}\quad\delta_j=\begin{cases}
            1 & \text{if }j\in Y,\\
            0 & \text{otherwise}.
        \end{cases}
    \end{equation*}
    Towards a contradiction, assume that $y$ equals one of the~$x_n$. We will derive that $\mathcal M$ contains~$y$ and hence~$Y$, against the above. Recall that $x_n$ is given as a Cauchy sequence~$(x_{ni})_{i\in\mathbb N}$. Given $x_n=y$, an unbounded search yields an $f\colon\mathbb N\to\mathbb N$ in~$\mathcal N$ such that any $k\in\mathbb N$ validates $f(k)\geq k+2$ and $|x_{n,f(k)}-y_{f(k)}|<3^{-k-1}$. We get
    \begin{equation*}
        \left|x_{n,f(k)}-y\right|\leq\left|x_{n,f(k)}-y_{f(k)}\right|+\left|y_{f(k)}-y\right|<\frac1{3^{k+1}}+\sum_{j=k+2}^\infty\frac1{3^j}=\frac1{2\cdot3^k}.
    \end{equation*}
    To avoid misunderstanding, we emphasize that $f$ may not be a rate of convergence, i.e., that $|x_{ni}-y|$ can be large for some~$i\geq f(k)$.

    As $\mathcal N$ is hyperimmune-free over~$\mathcal M$, we find a $g\in\mathcal M$ that dominates~$f$. Given that $(x_{ni})$ is Cauchy, there is a $\Sigma^0_2$-definable $H\colon\mathbb N\to\mathbb N$ with
    \begin{equation*}
        \left|x_{n,H(k)}-x_{nl}\right|<\frac1{3^k}\quad\text{for}\quad l\geq H(k).
    \end{equation*}
    To obtain a computable substitute for~$H$, we employ a concept known as metastability, which is used in proof mining~\cite{kohlenbach-book,kohlenbach-icm}.\footnote{Our argument is inspired by a result due to Miller and Martin (see Corollary~1.21 of~\cite{miller-martin}) and its presentation as Theorem~5.2 in~\cite{lewis-pye}.} Specifically, an unbounded search yields an $h\colon\mathbb N\to\mathbb N$ in~$\mathcal M$ such that any $k\in\mathbb N$ validates $h(k)\geq k$ and
    \begin{equation*}
        \left|x_{n,h(k)}-x_{nl}\right|<\frac1{3^k}\quad\text{for}\quad h(k)\leq l\leq g(h(k)).
    \end{equation*}
    In view of $h(k)\leq f(h(k))\leq g(h(k))$, we get
    \begin{equation*}
        \left| x_{n,h(k)}-y\right|\leq\left| x_{n,h(k)}-x_{n,f(h(k))}\right|+\left| x_{n,f(h(k))}-y\right|<\frac1{3^k}+\frac1{2\cdot3^{h(k)}}\leq\frac1{2\cdot 3^{k-1}}.
    \end{equation*}
    One can conclude that we have
    \begin{equation*}
        \delta_i=\begin{cases}
            1 & \text{if }\left|y_i-x_{n,h(i+2)}\right|>2/3^{i+1},\\
            0 & \text{otherwise}.
        \end{cases}
    \end{equation*}
    Indeed, when we have $\delta_i=1$, we get
    \begin{equation*}
        \left|y_i-x_{n,h(i+2)}\right|\geq|y_i-y|-\left|y-x_{n,h(i+2)}\right|>\frac1{3^i}-\frac1{2\cdot 3^{i+1}}=\frac5{2\cdot 3^{i+1}},
    \end{equation*}
    while $\delta_i=0$ entails
    \begin{equation*}
        \left|y_i-x_{n,h(i+2)}\right|\leq\left|y-x_{n,h(i+2)}\right|+\left|y-y_i\right|\leq\frac1{2\cdot 3^{i+1}}+\sum_{j=i+1}^\infty\frac1{3^j}=\frac4{2\cdot 3^{i+1}}.
    \end{equation*}
    It follows that $\mathcal M$ contains the function $i\mapsto\delta_i$ and hence the set~$Y$.
\end{proof}

As one may have expected, $\mathsf{IPP}$ is needed for certain considerations that involve finite families of reals.

\begin{lemma}\label{lem:fin-sets-bounded}
    The following are equivalent over~$\mathsf{RCA}_0$:
    \begin{enumerate}[label=(\roman*)]
        \item The infinite pigeonhole principle $\mathsf{IPP}$.
        \item Any finite set of reals is bounded.
        \item Any non-empty finite set of reals has a maximum.
    \end{enumerate}
\end{lemma}

\begin{proof}
    It is clear that~(iii) implies~(ii). To close the circle of implications, we first show that~(i) implies~(iii). Let $(x_n)_{n < N}$ be an arbitrary family of reals with $N > 0$. We define a real $y$ by setting $y_i = \max\{x_{ni} \mid n < N\}$ for each $i \in \mathbb{N}$. To verify that $y$ is Cauchy, consider an arbitrary $\varepsilon > 0$. For each $n < N$, there is an $I_n \in \mathbb{N}$ with $|x_{ni} - x_{nj}| < \varepsilon$ for all $i,j\geq I_n$. Using $\mathsf{IPP}$ in the form of $\Sigma^0_2$-collection (see the paragraph above Lemma~\ref{lem:card-surjection}), we obtain a single~$I\in\mathbb N$ with $I\geq I_n$ for all~$n<N$. By a short case distinction, we see that this entails $|y_i - y_j| < \varepsilon$ for all $i,j\geq I$. Clearly, $y$ is an upper bound on each~$x_n$. Another application of the pigeonhole principle shows that for some~$n<N$, we have $y_i=x_{ni}$ for infinitely many~$i$, which yields~$y=x_n$.

    It remains to show that~(ii) implies~(i). Assume that the infinite pigeonhole principle does not hold. Then, there is a colouring $c\colon \mathbb{N} \to k$ for some $k \in \mathbb{N}$ such that each colour only appears finitely often. For $n<k$ and $i\in\mathbb N$, we put
    \begin{equation*}
        x_{ni}=|\{j<i:c(j)=n\}|.
    \end{equation*}
    Each sequence $x_n=(x_{ni})_{i\in\mathbb N}$ stabilizes (when colour~$n$ no longer occurs) and thus represents a real. Assume, for contradiction, that the family $(x_n)_{n < k}$ is bounded by a number $N \in \mathbb{N}$. By the finite pigeonhole principle (which is available in~$\mathsf{RCA}_0$), there is a colour $n < k$ that appears at least $(N+1)$-many times. But then $x_n$ is not bounded by $N$.
\end{proof}

We can make a similar observation about sums and products:

\begin{remark}\label{rmk:fin-sums}
    Using the pigeonhole principle, we can similarly define sums and products of finitely many reals in a pointwise manner. 
    Somewhat informally (because we have not made precise what it means that these sums and products exist), we note that the pigeonhole principle cannot be avoided: For $\sum_{n<N}|x_n|$ and $\prod_{n<N}1+|x_n|$ to exist, the family~$(x_n)_{n<N}$ must be bounded.
\end{remark}

\section{First results about continuous functions}\label{sect:functions1}

In this section, we introduce our representation of continuous functions. After checking some basic properties, we show that the intermediate value theorem is equivalent to the infinite pigeonhole principle. More results (including equivalences with the strong cohesive principle) will be proved in subsequent sections (once our representation of open sets has been introduced).

Arguably, the most straightforward representation of continuous functions by countable objects records the values on rational arguments. Under the classical approach -- where reals are Cauchy sequences with rate --, this representation is not suitable over~$\mathsf{RCA}_0$ (unless functions come with a modulus of continuity). So one typically uses a different representation, which is not, however, suitable in our setting (see Remark~\ref{rmk:cont-fct} below). In any case, it is attractive to revert to the straightforward representation by values on~$\mathbb Q$, which works well for us:

\begin{definition}\label{def:continuous}
By a partial continuous function on $\mathbb R$, we mean an arbitrary function~$f\colon\mathbb Q\to\mathbb Q^{\mathbb N}$ (with sequences $f(q)=(f(q)_i)_{i\in\mathbb N}$ of rationals as values). We say that a real $x$ lies in the domain $\dom(f)$ of~$f$ if each~$\varepsilon>0$ admits an~$N\in\mathbb N$ (possibly depending on~$x$) such that we have
\begin{equation*}
    \big|f(q)_i-f(r)_j\big|<\varepsilon\quad\text{for all rationals }q,r\in B_{1/N}(x)\text{ and all }i,j\geq N.
\end{equation*}
Here and in the following, $B_\delta(y)$ denotes the open ball
\begin{equation*}
    B_\delta(y)=\{z\in\mathbb R:|y-z|<\delta\}.
\end{equation*}
For $x=(x_i)_{i\in\mathbb N}\in\dom(f)$, we get a real value~$\bar f(x)$ (cf.~the next lemma) by setting
\begin{equation*}
    \bar f(x)=\left(\bar f(x)_i\right)_{i\in\mathbb N}\quad\text{with}\quad \bar f(x)_i=f(x_i)_i.
\end{equation*}
We often use~$f$ at the place of~$\bar f$ and write~$f\colon D\to\mathbb R$ for definable~$D\subseteq\mathbb R$ to assert that $D$ is contained in the domain of~$f$.
\end{definition}

Let us verify the following fundamental property.

\begin{lemma}[$\mathsf{RCA}_0$]
(a) If $x$ lies in~$\dom(f)$, then $\bar f(x)$ is a real number.

(b) For $x=_{\mathbb R}y$ in $\dom(f)$, we have $\bar f(x)=\bar f(y)$.
\end{lemma}
\begin{proof}
    To establish both (a) and~(b), it suffices to show that $x=y$ implies the following: For each~$\varepsilon>0$, there is an $M\in\mathbb N$ such that we have $|\bar f(x)_i-\bar f(y)_j|<\varepsilon$ for all~$i,j\geq M$. Pick an~$N$ that witnesses $x\in\dom(f)$ for our~$\varepsilon$. Then take an~$M\geq N$ such that $i,j\geq M$ implies $|x_i-x_j|\leq1/(2N)$ and $|y_i-y_j|\leq 1/(2N)$. When we have $i,j\geq M$, we get $x_i,y_j\in B_{1/N}(x)$ (recall $x=y$) and thus
    \begin{equation*}
        \left|\bar f(x)_i-\bar f(y)_j\right|=\left|f(x_i)_i-f(y_j)_j\right|<\varepsilon,
    \end{equation*}
    as desired.
\end{proof} 

The condition for $x\in\dom(f)$ from Definition~\ref{def:continuous} combines continuity with locally uniform convergence of Cauchy sequences, which was vital for the previous proof. Indeed, it seems that the uniformity condition is unavoidable if one wants to define $\bar f$ on irrational arguments without arithmetical comprehension. By a straight\-forward modification of the previous proof, we also get the following.

\begin{lemma}[$\mathsf{RCA}_0$]\label{lem:eps-delta}
For a continuous function $f\colon D\to\mathbb R$ and $x\in D$, any $\varepsilon>0$ admits a $\delta>0$ with $|f(x)-f(y)|<\varepsilon$ for $y\in B_\delta(x)\cap D$.
\end{lemma}

Let us stress that the uniformity condition from Definition~\ref{def:continuous} does not require that we provide a rate for the Cauchy sequences $(f(q)_i)_{i\in\mathbb N}$. This is illustrated, for example, by part~(b) of the following result, where $y$ does not come with a rate.

\begin{lemma}[$\mathsf{RCA}_0$]\label{lem:basicContFunctions}
    We have the following continuous functions:
    \begin{enumerate}[label=(\alph*)]
\item The identity~$\mathbb R\ni x\mapsto x\in\mathbb R$.
\item The constant function~$\mathbb R\ni x\mapsto y\in\mathbb R$ for any fixed~$y$.
\item The absolute value function on~$\mathbb R$.
\item The function $D\ni x\mapsto f_0(x)+f_1(x)$ for any continuous $f_i\colon D\to\mathbb R$, as well as the analogous functions for subtraction, multiplication and division (in the latter case without zeros in the denominator).
    \end{enumerate}
\end{lemma}
\begin{proof}
    (a) Define $f(q)=(f(q)_i)$ by $f(q)_i=q$. The condition from Definition~\ref{def:continuous} is satisfied when we have $N\geq 2/\varepsilon$. For $x=(x_i)$, we obtain $\bar f(x)_i=f(x_i)_i=x_i$ and hence~$\bar f(x)=x$. 

    (b) Set~$f(q)_i=y_i$ for~$y=(y_i)$. We have~$\dom(f)=\mathbb R^n$ since~$(y_i)$ is Cauchy. Here it is crucial that we only demand locally uniform convergence rather than a fixed Cauchy rate.

    For parts (c) and (d), the condition from Definition~\ref{def:continuous} reduces to the usual arguments for the continuity of the indicated functions.
\end{proof}

For typical operations on arbitrary finite or on infinite families~$(f_n)$ of continuous functions, we need the pigeonhole principle (cf.~Lemma~\ref{lem:fin-sets-bounded} and Remark~\ref{rmk:fin-sums}). Note that $(f_n)$ is simply a family of sets such that each $f_n$ represents a continuous function (without the uniformity in~$n$ that is required by Definition~\ref{def:seq-fct}). 

\begin{lemma}[$\mathsf{RCA}_0+\mathsf{IPP}$]\label{lem:inf-sum-fct}
    (a) For continuous $f_i\colon D\to\mathbb R$, we have continuous functions that send $x\in D$ to~$\sum_{i=0}^n f_i(x)$ and~$\prod_{i=0}^n f_i(x)$ and $\max_{i\leq n}f_i(x)$.

    (b) Consider rationals~$q_n\geq 0$ with $\sum_{n=0}^\infty q_n\leq B\in\mathbb R$ and a family of continuous functions $f_n\colon D\to\mathbb R$ with $|f_n(x)|\leq q_n$ for all~$n\in\mathbb N$ and all~$x\in D\subseteq\mathbb R$. We then have a continuous function~$f\colon D\to\mathbb R$ with
    \begin{equation*}
      f(x)=\sum_{n=0}^\infty f_n(x)\quad\text{and}\quad |f(x)|\leq B\quad\text{for all }x\in D.
    \end{equation*}
\end{lemma}
\begin{proof}
(a) Let us consider multiplication as the most involved case. For each~$q\in\mathbb Q$ and $j\in\mathbb N$, we put $f(q)_j=\prod_{i=0}^n f_i(q)_j$. Given $x\in D$, we invoke Lemma~\ref{lem:fin-sets-bounded} to find a~$C$ with $|f_i(x)|\leq C$ for all~$i\leq n$. For $\varepsilon > 0$, each $i \leq n$ admits an $N_i$ such that all rationals $q, r \in B_{1/N_i}(x)$ and all $j, k \geq N_i$ validate
\begin{equation}\label{eq:sum-fct-one}
    |f_i(q)_j|\leq C+1 \quad\text{and}\quad |f_i(q)_j-f_i(r)_k|<\frac{\varepsilon}{(n+1)\cdot(C+1)^n}=:\varepsilon'.
\end{equation}
Here the second inequality comes from the condition in Definition~\ref{def:continuous}. Since the premise $q, r \in B_{1/N_i}(x)$ is~$\Sigma^0_2$, finding a common bound $N \geq N_i$ requires, a priori, the $\Pi^0_2$-bounding principle, which is stronger than~$\mathsf{IPP}$. In order to avoid this, write $x=(x_m)$ and note that each~$i\leq n$ admits not only an~$N_i$ as above but, depending on the latter, also an~$M_i$ with
\begin{equation}\label{eq:sum-fct-two}
    |x_l-x_m|<\frac1{3N_i}\quad\text{for all}\quad l,m\geq M_i,
\end{equation}
which entails $B_{2/(3N_i)}(x_m)\subseteq B_{1/N_i}(x)$. So there are $M_i$ and~$N_i$ that validate (\ref{eq:sum-fct-two}) as well as (\ref{eq:sum-fct-one}) for all rationals $q, r \in B_{2/(3N_i)}(x_m)$ with $m\geq M_i$ and all $j, k \geq N_i$. But this statement is~$\Pi^0_1$, so that $\mathsf{IPP}$ yields~$M$ and $N$ that bound witnesses~$M_i$ and~$N_i$, respectively, for all~$i\leq n$. Given that we have $B_{1/(3N)}(x)\subseteq B_{2/(3N_i)}(x_M)$ for every~$i\leq n$ and corresponding~$N_i\leq N$, it follows that (\ref{eq:sum-fct-one}) holds for all rationals $q,r\in B_{1/(3N)}(x)$ and all~$j,k\geq N$.

By induction on~$m\leq n$, we now get
\begin{multline*}
    \left|\prod_{i=0}^m f_i(q)_j-\prod_{i=0}^m f_i(r)_k\right|\leq{}\\
    \begin{aligned}
    {}&\leq\left|f_m(q)_j\right|\cdot\left|\prod_{i=0}^{m-1} f_i(q)_j-\prod_{i=0}^{m-1} f_i(r)_k\right|+\left|\prod_{i=0}^{m-1} f_i(r)_k\right|\cdot\left|f_m(q)_j-f_m(r)_k\right|\\
    {}&<(C+1)\cdot m\cdot(C+1)^{m-1}\cdot\varepsilon'+(C+1)^m\cdot\varepsilon'=(m+1)\cdot(C+1)^m\cdot\varepsilon'.
    \end{aligned}
\end{multline*}
With $m=n$, and for arbitrary $q,r\in B_{1/(3N)}(x)$ and $j,k\geq N$, this yields
\begin{equation*}
    \left|f(q)_j-f(r)_k\right|<(n+1)\cdot(C+1)^n\cdot\varepsilon'=\varepsilon,
\end{equation*}
which means that~$f$ validates the condition from Definition~\ref{def:continuous}.

(b) Let us put
\begin{equation*}
f_n'(r)_i=\begin{cases}
q_n& \text{if }f_n(r)_i>q_n,\\
-q_n& \text{if }f_n(r)_i<-q_n,\\
f_n(r)_i& \text{otherwise}.
\end{cases}
\end{equation*}
It is straightforward to see that this represents continuous functions~$f_n'=f_n$. For notational convenience, assume that we had $|f_n(r)_i|\leq q_n$ to begin with. Now put
\begin{equation*}
    f(q)_i:=\sum_{n=0}^i f_n(q)_i.
\end{equation*}
To verify the condition from Definition~\ref{def:continuous}, we consider arbitrary $x\in D$ and~$\varepsilon>0$. First take~$M\in\mathbb N$ so large that we have $\sum_{n=M}^\infty q_n<\varepsilon/3$. As in the proof of~(a), we find an~$N\geq M$ such that all~$n<M$ validate
\begin{equation*}
    \left|f_n(q)_i-f_n(r)_j\right|<\frac\varepsilon{3M}\quad\text{for all $q,r\in B_{1/(3N)}(x)\cap\mathbb Q$ and all }i,j\geq N.
\end{equation*}
For $q,r$ and $i,j$ as indicated, we get
\begin{equation*}
\left|f(q)_i-f(r)_j\right|\leq\sum_{n=0}^{M-1}\left|f_n(q)_i-f_n(r)_j\right|+\sum_{n=M}^i\left|f_n(q)_i\right|+\sum_{n=M}^j\left|f_n(r)_j\right|<\varepsilon,
\end{equation*}
as required by Definition~\ref{def:continuous}. In order to see that $f$ is the desired limit, we consider the functions $F_N=\sum_{n<N}f_n$, as represented by $F_N(q)_i=\sum_{n<N}f_n(q)_i$. For $x=(x_i)\in D$ and $i\geq N$, we get
\begin{equation*}
\left|f(x)_i-F_N(x)_i\right|=\left|\sum_{n=0}^i f_n(x_i)_i-\sum_{n=0}^{N-1}f_n(x_i)_i\right|\leq\sum_{n=N}^i\left|f_n(x_i)_i\right|\leq\sum_{n=N}^\infty q_n,
\end{equation*}
which tends to zero as~$N$ grows.
\end{proof}

We also have the following fundamental closure property.

\begin{lemma}[$\mathsf{RCA}_0$]\label{lem:compositionCont}
If $f,g\colon\mathbb R\to\mathbb R$ are continuous, so is~$g\circ f$.
\end{lemma}
\begin{proof}
   Given representations~$f,g\colon\mathbb Q\to\mathbb R$, we define a representation~$h\colon\mathbb Q\to\mathbb R$ of the composition by $h(q)_i=g(f(q)_i)_i$. In order to show $\dom(h)=\mathbb R$, we consider arbitrary~$x\in\mathbb R$ and~$\varepsilon>0$. Our task is to find an~$N\in\mathbb N$ with $|h(q)_i-h(r)_j|<\varepsilon$ for all~$q,r\in B_{1/N}(x)$ and~$i,j\geq N$. First pick an~$M\in\mathbb N$ such that $|g(s)_i-g(t)_j|<\varepsilon$ holds for all~$s,t\in B_{1/M}(f(x))$ and~$i,j\geq M$. Then choose an~$N\geq M$ such that we have~$|f(q)_i-f(r)_j|<1/(2M)$ for~$q,r\in B_{1/N}(x)$ and~$i,j\geq N$. Given such~$q,r$ and~$i,j$, we learn that $f(q)_i$ and~$f(r)_j$ lie in~$B_{1/M}(f(x))$, which allows us to conclude. For $x=(x_i)\in\mathbb R$, we also have
   \begin{equation*}
       \bar g(\bar f(x))_i=g(\bar f(x)_i)_i=g(f(x_i)_i)_i=h(x_i)_i=\bar h(x)_i,
   \end{equation*}
   so that we indeed obtain~$\bar g(\bar f(x))=\bar h(x)$.
\end{proof}

Now that we have secured some fundamental properties, we briefly compare with the classical representation of continuous functions in reverse mathematics.

\begin{remark}\label{rmk:cont-fct}
    According to Definition~II.6.1 from Simpson's textbook~\cite{simpson09}, a partial continuous $f\colon\mathbb R\to\mathbb R$ is represented by a set of tuples $(n,a,r,b,s)\in\mathbb N\times\mathbb Q^4$ with $r,s>0$ that satisfy certain coherence conditions. The idea is that $f(x)$ lies in the closure of $B_s(b)$ whenever we have $x\in B_r(a)$ (and the component $n$ allows to turn enumerable into decidable sets). One then says that $x\in\mathbb R$ lies in the domain of~$f$ if each~$\varepsilon>0$ admits a tuple $(n,a,r,b,s)$ with $x\in B_r(a)$ and~$s<\varepsilon$. In the classical case of~\cite{simpson09}, where $x$ comes with a Cauchy rate, the condition~$x\in B_r(a)$ is~$\Sigma^0_1$, so that we can effectively search for a tuple as indicated. For appropriate~$\varepsilon$, we can then choose some $y_i\in B_\varepsilon(b)$ to construct the value~$y=(y_i)_{i\in\mathbb N}=f(x)$. The approach is not suitable when~$x$ comes without a rate (as in the present paper), because $x\in B_r(a)$ is then a $\Sigma^0_2$-relation (see the paragraph after Lemma~\ref{lem:reals-ineq-Sigma02}). It may be possible to adapt Simpson's representation to our setting in some way. But since we cannot use his representation as it stands, we may as well employ Definition~\ref{def:continuous} from above, which is arguably more straightforward in any case (but does not produce reals with rate). One interesting parallel is that both Simpson's representation and ours use a $\Pi^1_1$-statement to express that the domain of~$f$ comprises all of~$\mathbb R$ (or some other given set). This has the effect that, e.g., the boundedness of continuous $f\colon[0,1]\to\mathbb R$ is expressed by a $\Pi^1_2$-statement, i.e., by a set existence principle.
\end{remark}

In our setting, the intermediate value theorem cannot be proved in $\mathsf{RCA}_0$ (in contrast to the classical case of Theorem~II.6.6 from\cite{simpson09}), but it can be proved in a conservative extension.

\begin{theorem}\label{thm:ivt}
    The following are equivalent over $\mathsf{RCA}_0$:
    \begin{enumerate}[label=(\roman*)]
        \item The infinite pigeonhole principle.
        \item The intermediate value theorem, which states that any continuous function $f\colon[0, 1] \to \mathbb{R}$ with $f(0)\cdot f(1) < 0$ admits an $x \in (0, 1)$ with $f(x) = 0$.
        \item The intermediate value theorem restricted to strictly increasing functions.
        \item Any continuous function $f\colon\mathbb R\to\{0,1\}$ is constant.
    \end{enumerate}
\end{theorem}

Before we prove the theorem, we discuss a result that will be needed for the reversal (i.e., for the direction towards the pigeonhole principle).

\begin{lemma}[essentially~\cite{chong-lempp-yang}]\label{lem:ipp_restricted}
    Over $\mathsf{RCA}_0$, the infinite pigeonhole principle~$\mathsf{IPP}$ is equivalent to its restrictions to colourings $c\colon\mathbb N\to\{0,\ldots,N\}$ with the following property: For any colour~$n\leq N$, there is an~$I\in\mathbb N$ such that we have either $c(i) \leq n$ for all $i\geq I$ or $c(i) \geq n$ for all $i\geq I$.
\end{lemma}

We give a shorter version of the argument by Chong, Lempp and Yang~\cite{chong-lempp-yang}, which avoids the notion of bi-tame cut (see also the related work of Slaman~\cite{slaman-ind-bound}).

\begin{proof}
    We assume that $\mathsf{IPP}$ fails and derive that the restricted version fails as well. So let $c\colon \mathbb{N} \to\{0,\ldots,N\}$ be a colouring such that $\{i\in\mathbb N:c(i)=n\}$ is finite for each~$n\leq N$. Write $S=S_{N+1}$ for the set of permutations of $\{0,\ldots,N\}$. Elements $\sigma\neq\tau$ of~$S$ are compared lexicographically, i.e., by
    \begin{equation*}
        \sigma\prec\tau\quad\Leftrightarrow\quad\sigma(k)<\tau(k)\text{ for }k=\min\{l\leq N:\sigma(l)\neq\tau(l)\}.
    \end{equation*}
    Since this yields a linear order, we may use $(S,\prec)$ rather than $(\{0,\ldots,N!-1\},<)$ as the codomain of a counterexample to the restricted pigeonhole principle.
    
    In order to define a colouring $d\colon\mathbb N\to S$, we first declare that $d(0)$ is the identity permutation. Recur\-sively, we then determine $k_i\leq N$ and $d(i+1)\in S$ by
    \begin{equation*}
        d(i)(k_i)=c(i)\quad\text{and}\quad d(i+1)(l)=\begin{cases}
        d(i)(l) & \text{for }l<k_i,\\
        d(i)(l+1) & \text{for }k_i\leq l<N,\\
        c(i) & \text{for }l=N.
        \end{cases}
    \end{equation*}
    Intuitively, $d(i)$ orders colours by their last appearance in $c$ up to stage $i$. We will show that any $\sigma\in S$ admits an~$I$ such that we have either $d(i)\prec\sigma$ for all~$i\geq I$ or $d(i)\succ\sigma$ for all~$i\geq I$. This shows both that $d$ fulfills the restriction from the lemma and that no colour appears infinitely often.
    
    Let $\sigma\in S$ be arbitrary. For $i\in\mathbb N$, we determine $l_i\leq N$ by
    \begin{equation*}
        d(i)(l_i)=\sigma(N).
    \end{equation*}
    Due to the assumption that $c$ violates the pigeonhole principle, we find a $J\in\mathbb N$ such that $c(i)\neq\sigma(N)$ holds for all~$i\geq J$. This ensures~$l_i\neq k_i$ for~$i\geq J$, which implies that the map $J\geq i\mapsto l_i$ is non-increasing. Specifically, if we have $l_i<k_i$ or $k_i<l_i\leq N$, respectively, we get
    \begin{equation*}
        d(i+1)(l_i)=d(i)(l_i)=\sigma(N)\quad\text{or}\quad d(i+1)(l_i-1)=d(i)(l_i)=\sigma(N)
    \end{equation*}
    and hence $l_{i+1}=l_i$ or $l_{i+1}=l_i-1$. We now find an $I\geq J$ such that $l_I=l_i<k_i$ holds for all~$i\geq I$. For $l\leq l_I$ and $i\geq I$, we get $l<k_i$ and thus $d(i+1)(l)=d(i)(l)$, which inductively yields~$d(i)(l)=d(I)(l)$. In view of $l_I<k_I\leq N$, we also have
    \begin{equation*}
        d(I)(l_I)=\sigma(N)\neq\sigma(l_I).
    \end{equation*}
    It follows that we have $d(I)\neq\sigma$ as well as
    \begin{equation*}
        d(I)\prec\sigma\quad\Rightarrow\quad d(i)\prec\sigma\text{ for all }i\geq I.
    \end{equation*}
    The same holds with $\succ$ at the place of~$\prec$, which yields the claim.
\end{proof}

Let us now prove the equivalences with the intermediate value theorem.

\begin{proof}[Proof of Theorem~\ref{thm:ivt}]
    We first show that (i) implies~(ii), i.e., that the intermediate value theorem can be proved via the pigeonhole principle. Consider a continuous function~$f\colon[0,1]\to\mathbb R$, which is represented via its values $f(q)=(f(q)_n)_{n\in\mathbb N}\in\mathbb R$ on arguments~$q\in\mathbb Q$. We assume $f(0)<0<f(1)$, noting that the remaining case is symmetrical. Our task is to find an~$x\in(0,1)$ with $f(x)=0$. If we have $f(q)=0$ for some~$q\in\mathbb Q$, there is nothing to do, so we assume otherwise. For each~$n\in\mathbb N$, we use recursion on~$i\leq n$ to define $[a_0^n,b_0^n]=[0,1]$ and
    \begin{equation*}
        \left[a_{i+1}^n,b_{i+1}^n\right]=\begin{cases}
            \left[a_i^n,\delta\right] & \text{if }f(\delta)_n\geq 0,\\
            \left[\delta,b_i^n\right] & \text{otherwise},
        \end{cases}
        \qquad\text{where}\qquad\delta=\frac{a_i^n+b_i^n}2.
    \end{equation*}
    To define $x=(x_n)_{n\in\mathbb N}$, we now set~$x_n=a^n_n$. Let us note that this amounts to a dynamical version of the classical argument, where we allow for arbitrary errors for an indefinite amount of time -- until the relevant Cauchy sequences have stabilized.

    Let us show that $(x_n)$ is Cauchy. For an arbitrary~$k\in\mathbb N$, we use Lemma~\ref{lem:fin-sets-bounded} (and hence the pigeonhole principle) to obtain the real
    \begin{equation*}
        \varepsilon:=\min\left\{\left|f\left(i\cdot 2^{-k}\right)\right|:0\leq i\leq 2^k\right\}>0.
    \end{equation*}
    Again by the pigeonhole principle (in the form of $\mathsf B\Sigma^0_2$), we find an~$N=N_k\geq k$ such that we have
    \begin{equation*}
        \left|f\left(i\cdot 2^{-k}\right)-f\left(i\cdot 2^{-k}\right)_n\right|<\varepsilon\quad\text{for all }i\leq 2^{-k}\text{ and all }n\geq N.
    \end{equation*}
    In this situation, $f(i\cdot 2^{-k})_n$ has the same sign as $f(i\cdot 2^{-k})$, which is thus independent of the specific~$n\geq N$. This entails
    \begin{equation*}
        x_n=a_n^n\in\left[a_k^n,b_k^n\right]=\left[a_k^N,b_k^N\right]\quad\text{for all }n\geq N.
    \end{equation*}
    It follows that we have
    \begin{equation*}
        |x_m-x_n|\leq b_k^N-a_k^N=2^{-k}\quad\text{for all }m,n\geq N.
    \end{equation*}
    Since~$k$ was arbitrary, this confirms that~$(x_n)$ is Cauchy and hence a real.

    In order to complete the proof of the intermediate value theorem, we establish $|f(x)|\leq\varepsilon$ for an arbitrary~$\varepsilon>0$, so that we get~$f(x)=0$. Let $M$ witness the property from Definition~\ref{def:continuous}, which means that we have
    \begin{equation*}
    \big|f(q)_i-f(r)_j\big|<\varepsilon\quad\text{for all rationals }q,r\in B_{1/M}(x)\text{ and all }i,j\geq M.
    \end{equation*}
    Pick a $k\in\mathbb N$ with $2^{-k}<1/M$ and find~$N=N_k$ as in the previous paragraph. We may assume~$N\geq M$. By the above, we have
    \begin{equation*}
        x_n\in\left[a_k^N,b_k^N\right]\subseteq B_{1/M}(x)\quad\text{for all }n\geq N.
    \end{equation*}
    With $f(x)_n=f(x_n)_n$ as in Definition~\ref{def:continuous}, we get
    \begin{equation*}
        \max\left\{\big|f(x)_n-f\left(a_k^N\right)_N\big|,\big|f(x)_n-f\left(b_k^N\right)_N\big|\right\}<\varepsilon\quad\text{for all }n\geq N.
    \end{equation*}
    Possibly for a modified representation of~$f$, we may assume~$f(0)_N<0<f(1)_N$. By a straightforward induction on~$i\leq k$, we then obtain
    \begin{equation*}
        f\left(a_i^N\right)_N<0\leq f\left(b_i^N\right)_N.
    \end{equation*}
    It follows that we have $|f(x)_n|<\varepsilon$ for all~$n\geq N$. This yields $|f(x)|\leq\varepsilon$, as desired.

    The direction from~(ii) to~(iii) is simply a restriction. In order to show that~(ii) implies~(iv), we assume that the former holds while the latter fails. This gives a function $f\colon\mathbb R\to\{0,1\}$ that is continuous but not constant. Pick reals $y<z$ with $f(y)\neq f(z)$ and define $g\colon[0,1]\to\mathbb R$ by
    \begin{equation*}
        g(x)=f(x\cdot (z-y)+y)-\frac12.
    \end{equation*}
    Note that the continuous function~$g$ is available by Lemmas~\ref{lem:basicContFunctions} and~\ref{lem:compositionCont}. In view of
    \begin{equation*}
        g(0)\cdot g(1)=\left(f(y)-\frac12\right)\cdot\left(f(z)-\frac12\right)=-\frac14<0,
    \end{equation*}
    statement~(ii) yields an~$x\in(0,1)$ with $g(x)=0$. But then $1/2$ lies in the range of the function~$f$, against our assumption.

    Concerning the reversal, it remains to show that each of~(iii) and~(iv) implies~(i). We argue by contraposition and thus assume that~(i) fails. In light of Lemma~\ref{lem:ipp_restricted}, we find a colouring $c\colon\mathbb{N} \to\{0,\ldots,N\}$ such that the following holds for each $n\leq N$: There is a bound $I\in \mathbb{N}$ with the propery that either $c(i) < n$ holds for all $i\geq I$ or $c(i) > n$ holds for all $i\geq I$. To define $f\colon\mathbb Q\to\mathbb R$ with values~$f(q)=(f(q)_i)_{i\in\mathbb N}$, we now stipulate
    \begin{equation*}
        f(q)_i=
        \begin{cases}
            -1 & \text{if $(N+1)\cdot q \leq c(i)$},\\
            +1 & \text{if $(N+1)\cdot q > c(i)$}.
        \end{cases}
    \end{equation*}
    To see that~$f$ has domain~$\mathbb R$, we verify the condition from Definition~\ref{def:continuous}, which also entails that $(f(q)_i)_{i\in\mathbb N}$ is Cauchy for each~$q\in\mathbb Q$. We find $z\in\mathbb Z$ and $J\in\mathbb N$ with
    \begin{equation*}
        z-1<(N+1)\cdot q<z+1\quad\text{for every rational }q\in B_{1/J}(x).
    \end{equation*}
    Then consider an~$I\geq J$ such that we have either $c(i)<z$ for all~$i\geq I$ or $c(i)>z$ for all~$i\geq I$. In the first case, we get $f(q)_i=+1$ for all~$q\in B_{1/I}(x)$ and all~$i\geq I$. In the second case, the same holds with $-1$ at the place of~$+1$. Either way, we have
    \begin{equation*}
        |f(q)_i-f(r)_j|=0\quad\text{for all rationals }q,r\in B_{1/I}(x)\text{ and all }i,j\geq I.
    \end{equation*}
    So we indeed have a representation of a continuous function~$f\colon\mathbb R\to\{-1,+1\}$. Given that $0\leq c(i)<N+1$ holds for all~$i\in\mathbb N$, we have $f(0)=-1$ and $f(1)=+1$. So the function $g\colon\mathbb R\to\{0,1\}$ with $g(x)=(f(x)+1)/2$ violates~(iv). To get a violation of~(iii), we note that $f$ is non-decreasing, as $q<r$ entails $f(q)_i\leq f(r)_i$ for each~$i\in\mathbb N$. Thus the function $h\colon[0,1]\to\mathbb R$ with $h(x)=f(x)+x$ is strictly increasing with $h(0)<0<h(1)$. But there cannot be an $x\in(0,1)$ with $h(x)=0$, as this would yield~$f(x)\in(-1,0)$. So the intermediate value theorem for strictly increasing functions is also violated.
\end{proof}

Since Lemma~\ref{lem:basicContFunctions} ensures that polynomials are continuous, we get the following by the usual argument.

\begin{corollary}[$\mathsf{RCA}_0+\mathsf{IPP}$]\label{cor:real-closed}
The field~$\mathbb R$ is real closed, i.e., every positive real has a root and every polynomial of odd degree has a zero.
\end{corollary}

\section{Open sets}\label{sect:open}

In this section, we develop a representation of open sets that is suitable for our setting. As we represent real numbers by Cauchy sequences without rate, the relation $x<y$ between $x,y\in\mathbb R$ is defined by a $\Sigma^0_2$-formula (see the paragraph after Definition~\ref{lem:reals-ineq-Sigma02}). The relation $x\in U$ between $x\in\mathbb R$ and an open set~$U\subseteq\mathbb R$ should have the same complexity. This is achieved by the following variant of the classical representation (cf.~Definition~II.5.9 in~\cite{simpson09}).

\begin{definition}\label{def:open}
    By a code for an open subset of~$\mathbb R$, we mean a set $U\subseteq\mathbb N\times\mathbb Q\times\mathbb Q_{>0}$ (with $\mathbb Q_{>0}=\{q\in\mathbb Q:q>0\}$). For such a set and a real~$x$, we write~$x\in U$ if there are $N\in\mathbb N$ and $a,r$ such that we have $x\in B_r(a)$ and~$(n,a,r)\in U$ for all~$n\geq N$.
\end{definition}

In contrast to the following result, we will later see that the pigeonhole principle is needed to take finite intersections of open sets.

\begin{lemma}[$\mathsf{RCA}_0$]\label{lem:unions}
    For any family~$(U_i)_{i\in\mathbb N}$ of open sets, we have an open set
    \begin{equation*}
        U=\bigcup_{i\in\mathbb N}U_i.
    \end{equation*}
\end{lemma}
\begin{proof}
    For fixed~$(a,r)\in\mathbb Q\times\mathbb Q_{>0}$ and arbitrary~$n\in\mathbb N$, let $(i(n),j(n))\in\mathbb N^2$ be minimal (in terms of Cantor code) with $(k,a,r)\in U_{i(n)}$ for all~$k \in\{j(n),\ldots,n-1\}$. There are no such~$k$ for $j(n)\geq n$, so that we minimize over a non-empty set. The latter shrinks as~$n$ grows, so that $n\mapsto(i(n),j(n))$ is non-decreasing. We declare
    \begin{equation*}
        (n, a, r) \in U \quad:\Leftrightarrow\quad (n,a,r)\in U_{i(n)}.
    \end{equation*}
    If we have $x\in U$, there are $N,a,r$ with $x\in B_r(a)$ and $(n,a,r)\in U$ for all~$n\geq N$. So for any~$n\geq N$, we have~$(k,a,r)\in U_{i(n)}$ for~$k=n$ and hence for all~$k\in\{j(n),\ldots,n\}$. This yields $i(n)=i(N)$ (as well as $j(n)=j(N)$) and hence $(n,a,r)\in U_{i(N)}$ for all~$n\geq N$. But then we have $x\in U_{i(N)}$.

    Conversely, if we have $x\in U_i$, there are $K,a,r$ with $x\in B_r(a)$ and $(k,a,r)\in U_i$ for all~$k\geq K$. It follows that the $(i(n),j(n))$ are bounded by~$(i,K)$. We thus have an $N\in\mathbb N$ with $i(n)=i(N)$ and $j(n)=j(N)$ for all~$n\geq N$. By construction, we get $(n,a,r)\in U_{i(n+1)}$ whenever we have $n\geq j(n+1)$. So for $n\geq\max(N,j(N))$, we have $(n,a,r)\in U_{i(n)}$ and hence~$(n,a,r)\in U$. But this yields $x\in U$.
\end{proof}

Due to the previous lemma and the following elementary observation, we can sometimes reduce to the case of basic open sets.

\begin{lemma}[$\mathsf{RCA}_0$]
    For any open set~$U\subseteq\mathbb R$, there is a family $(U_i)_{i\in\mathbb N}$ of open sets such that we have $U=\bigcup_{i\in\mathbb N}U_i$ and each~$U_i$ is empty or an open interval with rational endpoints.
\end{lemma}
\begin{proof}
   When~$i$ codes the pair~$(a,r)$, we define $U_i$ as the set of all tuples~$(n,a,r)$ that lie in~$U$. Then $U_i$ is either the empty set or the interval $(a-r,a+r)$ (but we cannot decide which of the two it is). One readily verifies~$U=\bigcup_{i\in\mathbb N}U_i$.  
\end{proof}

The following correspondence between open sets and continuous functions provides support for our definitions of both notions.

\begin{proposition}[$\mathsf{RCA_0}$]\label{pro:ContinuousOpen}
    (a) For any continuous function~$f\colon\mathbb R\to\mathbb R$ and any open set $V\subseteq\mathbb R$, the preimage $f^{-1}(V)$ is again open.

    (b) For any open~$U$, there is a continuous $g\colon\mathbb R\to[0,1]$ with $g^{-1}((0,1])$.
\end{proposition}
\begin{proof}
    (a) By the proof of the previous lemma, we may write $V$ as the union over a family of open sets $V_{a,r}$ for $a,r\in\mathbb Q$ with $r>0$, where $V_{a,r}$ contains only tuples of the form~$(n,a,r)$. It will be enough to construct a family of open sets $U_{a,r}$ that are equal to the preimages $f^{-1}(V_{a,r})$. Indeed, we can then use Lemma~\ref{lem:unions} to define~$U$ as the union over the sets~$U_{a,r}$. One readily concludes $U=f^{-1}(V)$.
    
    Our~$f$ is represented by a family of Cauchy sequences $(f(q)_n)_{n\in\mathbb N}$ for~$q\in\mathbb Q$. We now declare
    \begin{equation*}
        (n,b,s)\in U_{a,r}\quad\Leftrightarrow\quad\parbox{.48\textwidth}{$(n,a,r)\in V_{a,r}$ and $f(q)_n\in B_{r-s}(a)$ for all rationals~$q\in B_s(b)$ with $q\leq_{\mathbb N}n$.}
    \end{equation*}
    If we have $x=(x_n)_{n\in\mathbb N}\in f^{-1}(V_{a,r})$, we get $f(x)\in B_r(a)$ and there is an $N\in\mathbb N$ such that $(n,a,r)\in V_{a,r}$ holds for all~$n\geq N$. Pick a rational~$t>0$ such that we even have $f(x)\in B_{r-3t}(a)$. In view of Definition~\ref{def:continuous}, we may increase~$N$ to get
    \begin{equation*}
        |f(q)_n-f(r)_m|<t\quad\text{for all rationals }q,r\in B_{1/N}(x)\text{ and all }m,n\geq N.
    \end{equation*}
    Pick an~$N'\geq N$ such that $n\geq N'$ implies $x_n\in B_{1/N}(x)$ and $|f(x)-f(x_n)_n|<t$. For $q\in B_{1/N}(x)$ and $n\geq N'$, we then have
    \begin{equation*}
        |f(q)_n-f(x)|\leq|f(q)_n-f(x_n)_n|+|f(x_n)_n-f(x)|<2t
    \end{equation*}
    and thus $f(q)_n\in B_{r-t}(a)$. Pick a positive rational $s\leq t$ with $x\in B_s(b)\subseteq B_{1/N}(x)$. We then have $(n,b,s)\in U_{a,r}$ for all~$n\geq N'$, which yields~$x\in U_{a,r}$.

    Conversely, if we have $x\in U_{a,r}$, there are $N,b,s$ such that we have $x\in B_s(b)$ and $(n,b,s)\in U_{a,r}$ for all~$n\geq N$. The latter yields $B_r(a)\subseteq V_{a,r}$. By repeating members of the Cauchy sequence $x=(x_n)_{n\in\mathbb N}$, we can ensure that $x_n\leq_{\mathbb N}n$ holds for all~$n\in\mathbb N$. We may also assume that $x_n\in B_s(b)$ holds for all~$n\geq N$ (possibly with an increased~$N$). Considering the definition of~$U_{a,r}$, we now see that $n\geq N$ implies $f(x_n)_n\in B_{r-s}(a)$. This yields
    \begin{equation*}
        f(x)=(f(x_n)_n)_{n\in\mathbb N}\in B_r(a)\subseteq V_{a,r}
    \end{equation*}
    and hence $x\in f^{-1}(V_{a,r})$.

    (b) Let $U[i]$ for $i\in\mathbb N$ consist of those tuples~$(n,a,r)$ with $n<i$ such that we have $(n',a,r)\in U$ whenever~$n\leq n'<i$ holds. Intuitively, these are the tuples that are assumed to contribute to $U$ at stage~$i$. It is important that we keep the tuples with~$n<i-1$, because these will allow us to give higher weight to contributions that are preserved since an early stage.

For $(a,r)\in\mathbb Q\times\mathbb Q_{>0}$, we define $g_{a,r}\colon\mathbb Q\to\mathbb Q$ by 
\begin{equation*}
    g_{a,r}(q)=\max\big(0,\min(1,r-|a-q|)\big)\in[0,1].
\end{equation*}
Let us note that we always have
\begin{equation*}
    \left|g_{a,r}(p)-g_{a,r}(q)\right|\leq|p-q|.
\end{equation*}
For each~$q\in\mathbb Q$, we now define a sequence~$g(q)=(g(q)_i)_{i\in\mathbb N}$ by
\begin{equation*}
    g(q)_i=\max\left(\left\{\frac{1}{i+1}\right\}\cup\left\{\frac{g_{a,r}(q)}{(n,a,r)+1}:(n,a,r)\in U[i]\right\}\right).
\end{equation*}
Note that we only need to consider tuples $(n,a,r)<_{\mathbb N} i$, which means that the maximum is well-defined and computable even when~$U[i]$ is infinite. In the following, we show that the condition from Definition~\ref{def:continuous} is satisfied for every~$x\in\mathbb R$. This condition implies both that the sequences~$g(q)$ are Cauchy and that they represent a continuous function~$g\colon\mathbb R\to\mathbb R$.

First consider a real~$x\in U$. We then find~$n,a,r$ with $x\in B_r(a)$ and $(n',a,r)\in U$ for all~$n'\geq n$, which entails that $(n,a,r)\in U[i]$ holds for all~$i>n$. Take a rational number $s\in(0,1]$ with $B_{2s}(x)\subseteq B_r(a)$. We then have
\begin{equation*}
    g_{a,r}(q)\geq s\quad\text{for all }q\in B_s(x).
\end{equation*}
Let us pick an integer~$I\geq ((n,a,r)+1)/s$ that also satisfies~$I>n$. Consider an integer $i\geq I$ and a rational $q\in B_s(x)$. If $(m,b,t)$ lies in~$U[i]$ but not in~$U[I]$, we must have $I\leq m\leq(m,b,t)$, where the second equality holds by a standard assumption on the encoding of tuples. We can conclude
\begin{equation*}
    \frac{g_{b,t}(q)}{(m,b,t)+1}\leq\frac1{I+1}\leq\frac{s}{(n,a,r)+1}\leq \frac{g_{a,r}(q)}{(n,a,r)+1}.
\end{equation*}
This shows that we have 
\begin{equation*}
    g(q)_i=g(q)_I\quad\text{ for all }q\in B_s(x)\text{ and all }i\geq I.
\end{equation*}
Given some~$\varepsilon>0$, we now take an integer~$N\geq\max(I,2/\varepsilon)$. For integers $i,j\geq N$ and rationals $p,q\in B_{1/N}(x)$, we get
\begin{equation*}
    \left|g(p)_i-g(q)_j\right|=\left|g(p)_I-g(q)_I\right|\leq|p-q|<\varepsilon,
\end{equation*}
as the condition from Definition~\ref{def:continuous} demands. Our considerations also yield
\begin{equation*}
    g(x)\geq\frac{s}{(n,a,r)+1}>0,
\end{equation*}
so that we have $x\in g^{-1}(\{y\in\mathbb R:y>0\})$ for $x\in U$.

Now consider a real~$x\notin U$. Given a rational~$\varepsilon>0$, pick a rational~$q_0\in B_{\varepsilon/3}(x)$. For any tuple~$(n,a,r)$ with $q_0\in B_{r-\varepsilon/3}(a)$, we get $x\in B_r(a)$. Due to~$x\notin U$, we must have $(n',a,r)\notin U$ for some~$n'\geq n$. It follows that we have $(n,a,r)\notin U[i]$ for all~$i>n'$. Pick an integer~$M\geq 1/\varepsilon$. By $\Sigma^0_1$-collection (which is a consequence of $\Sigma^0_1$-induction and hence available in~$\mathsf{RCA}_0$), we find an integer~$N\geq 3/\varepsilon$ with
\begin{equation*}
    |a-q_0|\geq r-\frac\varepsilon3\quad\text{for all }(n,a,r)\in\bigcup_{i\geq N}U[i]\text{ with }(n,a,r)<_{\mathbb N}M.
\end{equation*}
Now consider any $i\geq N$ and any rational $q\in B_{1/N}(x)$. Let $(n,a,r)$ be some tuple in~$U[i]$. If we have $(n,a,r)<_{\mathbb N}M$, then we get
\begin{equation*}
    r-\frac\varepsilon3\leq|a-q_0|\leq|a-q|+|q-x|+|x-q_0|<|a-q|+2\cdot\frac\varepsilon3.
\end{equation*}
This yields~$r-|a-q|<\varepsilon$ and hence~$g_{a,r}(q)<\varepsilon$. It follows that we have
\begin{equation*}
    0< g(q)_i<\varepsilon\quad\text{for all rationals }q\in B_{1/N}(x)\text{ and all }i\geq N,
\end{equation*}
which entails the condition from Definition~\ref{def:continuous}. Since $\varepsilon>0$ was arbitrary, we also see that $x\notin U$ entails~$g(x)=0$ and hence $x\notin g^{-1}(\{y\in\mathbb R:y>0\}$.
\end{proof}

As promised, we now show that finite intersections of open sets are harder to form than infinite unions.

\begin{proposition}\label{prop:intersections}
The following are equivalent over $\mathsf{RCA}_0$:
\begin{enumerate}[label=(\roman*)]
\item The infinite pigeonhole principle~$\mathsf{IPP}$.
\item For any $I\in\mathbb N$ and any family~$(U_i)_{i\leq I}$ of open sets, there is an open set~$U$ such that we have $U=\bigcap_{i\leq I}U_i$. 
\end{enumerate}
\end{proposition}
\begin{proof}
We first show that~(i) implies~(ii). The proof of Proposition~\ref{pro:ContinuousOpen}(b) is uniform, so that we get a family of continuous $g_i\colon\mathbb R\to[0,1]$ with $U_i=g_i^{-1}((0,1])$ for~$i\leq I$. Lemma~\ref{lem:inf-sum-fct} provides the continuous function $g\colon\mathbb R\to\mathbb R$ with $g(x)=\prod_{i\leq I}g_i(x)$. To conclude by Proposition~\ref{pro:ContinuousOpen}, it suffices to note that we have
\begin{equation*}
    \bigcap_{i\leq I}U_i=g^{-1}((0,1]).
\end{equation*}
The latter relies on the fact that a product $\prod_{i\leq I}x_i$ of reals is zero precisely when we have $x_i=0$ for some~$i\leq I$, which is readily derived from Lemma~\ref{lem:fin-sets-bounded}.

To show that (ii) implies (i), we argue by contraposition. So assume that we have an $N\in\mathbb N$ and a function~$c\colon\mathbb N\to\{0,\ldots,N\}$ such that $\{i\in\mathbb N:c(i)=n\}$ is finite for each $n\leq N$. Consider the open sets $U_n$ for $n\leq N$ that are coded by
\[
U_n:=\left\{\left(i,0,\frac1{k+1}\right): k=\max\{j\leq i:j=0\text{ or }c(j)=n\}\right\}.
\]
We claim that $\bigcap_{n\leq N}U_n$ is equal to~$\{0\}$ and hence not open. To see that~$0$ lies in each~$U_n$, we distinguish two cases. If $n$ is not in the range of~$c$, we have $(i,0,1)\in U_n$ for all~$i\in\mathbb N$, which yields~$0\in B_1(0)\subseteq U_n$. Now assume that $n$ is in the range of~$c$. As no colour occurs infinitely often, we may consider the largest~$K$ with $c(K)=n$. We then have $(i,0,1/(K+1))\in U_i$ for all~$i\geq K$, so that we get
\begin{equation*}
    0\in B_{1/(K+1)}(0)\subseteq U_n.
\end{equation*}
For the converse, we consider an arbitrary real~$x\neq 0$. Pick $J\in\mathbb N$ so large that we have~$x\notin B_{1/(J+1)}(0)$. Aiming at a contradiction, we assume~$x\in U_{c(J)}$. We must then have $k,I\in\mathbb N$ with $x\in B_{1/(k+1)}(0)$ and $(i,0,1/(k+1))\in U_{c(J)}$ for all~$i\geq I$. But for $i\geq\max(I,J)$, this forces~$k\geq J$ and hence
\begin{equation*}
x\in B_{1/(k+1)}(0)\subseteq B_{1/(J+1)}(0),
\end{equation*}
which contradicts the choice of~$J$.
\end{proof}

The intersection of two open sets can be formed over~$\mathsf{RCA}_0$ (without the pigeonhole principle). This can be shown as in the previous proof, with Lemma~\ref{lem:basicContFunctions} at the place of Lemma~\ref{lem:inf-sum-fct}. We obtain a particularly easy proof of the following elementary fact (though a direct proof is not too difficult either). 

\begin{corollary}[$\mathsf{RCA_0}$]
For~$x,y\in\mathbb R\cup\{\pm\infty\}$, the interval~$(x,y)$ is an open set.
\end{corollary}
\begin{proof}
    We consider the case where we have $-\infty=x<y<+\infty$. The general result follows by symmetry and since the open sets are closed under binary intersections. Writing $y=(y_n)$, we put $f_y(q)_n=y_n-q$ for $q<y_n$ and $f_y(q)_n=0$ other\-wise. One can check that this represents the continuous function~$f_y\colon\mathbb R\to\mathbb R$~with
    \begin{equation*}
        f_y(x)=\begin{cases}
            y-x & \text{for }x<y,\\
            0 & \text{otherwise}.
        \end{cases}
    \end{equation*}
    It is straightforward to see that $(0,\infty)$ is open (as the endpoint is rational). By the previous result, the same holds for $f_y^{-1}((0,\infty))=(-\infty,y)$.
\end{proof}

As another straightforward application, we obtain Urysohn's lemma. Of course, a set is closed precisely when its complement is open.

\begin{corollary}[$\mathsf{RCA}_0$]\label{lem:Urysohn}
For disjoint closed sets $C_0,C_1\subseteq\mathbb R$, there is a continuous function $g\colon\mathbb{R}\to [0,1]$ such that
\begin{equation*}
    x\in C_i\quad\Leftrightarrow\quad g(x)=i
\end{equation*}
holds for each~$i\in\{0,1\}$ and all $x\in\mathbb R$.
\end{corollary}
\begin{proof}
Proposition \ref{pro:ContinuousOpen} yields continuous functions $g_i\colon\mathbb{R}\to [0,1]$ such that $x\in C_i$ is equivalent to~$g_i(x)=0$. Due to Lemma~\ref{lem:basicContFunctions}, we can form $g=g_0/(g_0+g_1)$ as a continuous function. The desired property is readily verified.
\end{proof}

Next, we consider two topological properties of the reals, namely connectedness and paracompactness. On the other hand, the Heine-Borel theorem about the open-cover compactness of~$[0,1]$ -- which is arguably the most important result in this direction --, is deferred to the next section.

\begin{proposition}\label{prop:connected}
The following are equivalent over~$\mathsf{RCA}_0$:
\begin{enumerate}[label=(\roman*)]
\item The reals are connected, i.e., there are no open sets~$U_0,U_1\neq\emptyset$ such that we have $U_0\cup U_1=\mathbb R$ and $U_0\cap U_1=\emptyset$.
\item The infinite pigeonhole principle $\mathsf{IPP}$.
\end{enumerate}
\end{proposition}
\begin{proof}
It suffices to observe that the present statement~(i) is equivalent to statement~(iv) of Theorem \ref{thm:ivt}, which says that any continuous function~$g\colon\mathbb R\to\{0,1\}$ is constant. Concerning the forward direction, if $g\colon\mathbb R\to\{0,1\}$ is continuous, the sets~$U_i=g^{-1}(\{i\})$ are open by Proposition~\ref{pro:ContinuousOpen}, as $U_0$ and $U_1$ are also the preimages of $(-\infty,1/2)$ and $(1/2,\infty)$, respectively. By the present statement~(i), it follows that one of the $U_i$ is empty, so that~$g$ is indeed constant.

Conversely, suppose that $\mathbb{R}$ is a disjoint union~$U_0\cup U_1$ of open sets. Then each~$U_i$ is also closed. So by the previous corollary (Urysohn's lemma), there is a continuous function $g\colon\mathbb R\to[0,1]$ with $g(x)=i$ for~$x\in U_i$. By statement~(iv) of Theorem \ref{thm:ivt}, this~$g$ is constant. But then one of the $U_i$ is empty.
\end{proof}

For the following result, we have not established a reversal.

\begin{proposition}[$\mathsf{RCA}_0+\mathsf{IPP}$]
Given an open cover $(U_n)_{n\in\mathbb N}$ of~$\mathbb R$, we find open sets $V_n\subseteq U_n$ such that $(V_n)_{n\in\mathbb N}$ covers~$\mathbb R$ and is locally finite (which means that any $x\in\mathbb R$ lies in an open set $W$ with $V_n\cap W=\emptyset$ for all but finitely many $n\in\mathbb N$).
\end{proposition}
\begin{proof}
We follow the proof in~\cite{simpson09}. As noted before, the proof of Proposition~\ref{pro:ContinuousOpen}(b) is uniform, so that the given open cover yields a family of continuous $h_n\colon\mathbb R\to[0,1]$ with $U_n=h_n^{-1}\big((0,1]\big)$ for all $n\in\mathbb N$. By Lemma~\ref{lem:inf-sum-fct}, we form the continuous function
\begin{equation*}
    h\colon\mathbb R\to(0,\infty)\quad\text{with}\quad h(x)=\sum_{n=0}^\infty\frac{h_n(x)}{2^n}.
\end{equation*}
Note that the values of $h$ are indeed strictly positive, because $(U_n)$ is a cover. This allows us to consider the continuous functions
\begin{equation*}
    g_n\colon\mathbb R\to[0,1]\quad\text{with}\quad g_n(x)=\frac{h_n(x)}{2^n\cdot h(x)}.
\end{equation*}
Let us observe that these validate $U_n=g_n^{-1}((0,1])$. Finally, consider the continuous functions $f_n\colon\mathbb R\to[0,1/2]$ that are given by
\[
f_n(x)=\min\left(\frac{1}{2},\sum_{m\leq n}g_m(x)\right)-\min\left(\frac{1}{2},\sum_{m< n}g_m(x)\right).
\]
The proof of Proposition \ref{pro:ContinuousOpen}(a) is also uniform, so that we obtain a family of open sets $V_n=f_n^{-1}((0,1/2])$. Since $f_n(x)\leq g_n(x)$ holds for all $n\in\mathbb N$ and~$x\in\mathbb R$, we always have $V_n\subseteq U_n$.

To show that $(V_n)$ is a cover, we consider an arbitrary~$x\in\mathbb R$. Let us distinguish two cases. First assume that there is an~$N\in\mathbb N$ with $0<\sum_{n\leq N}g_m(x)<1/2$. Take such an~$N$ and pick an~$n\leq N$ with~$g_n(x)>0$. We then have $f_n(x)=g_n(x)$ and hence~$x\in V_n$. In the remaining case, we have
\begin{equation*}
    0<\sum_{n\leq N}g_n(x)\quad\Leftrightarrow\quad \frac12\leq\sum_{n\leq N}g_n(x).
\end{equation*}
Since $\mathsf{IPP}$ is equivalent to the $\Delta^0_2$-least number principle (see Theorem~I.2.5 of~\cite{hajek91}), we find an~$N$ that is minimal with $\sum_{n\leq N}g_n(x)>0$. For this $N$ we have $x\in V_N$.

By construction, we have $\sum_{n\in\mathbb N}g_n=1$. Given any~$x\in\mathbb R$, we thus find an $N\in\mathbb N$ and an open~$W\ni x$ with $\sum_{n<N}g_n(y)>1/2$ for all~$y\in W$. When we have $n\geq N$, we thus get $f_n(y)=0$ for $y\in W$, which yields~$V_n\cap W=\emptyset$.
\end{proof}

In the following, we consider the Tietze extension theorem. Once the latter is established, a stronger form of the following result follows from Lemma~\ref{pro:ContinuousOpen}.

\begin{lemma}[$\mathsf{RCA}_0$]\label{lem:ContinuousOpen-plus}
    For any $q\in\mathbb R$ and every continuous function $f\colon C\to\mathbb R$ on a closed set $C\subseteq\mathbb R$, the set $\{x\in C:f(x)<q\}\cup\mathbb R\backslash C$ is open.
\end{lemma}
\begin{proof}
We identify $\mathbb R\cup C$ with its code in the sense of Definition~\ref{def:open}. Let $U$ consist of all tuples~$(n,b,s)\in\mathbb N\times\mathbb Q\times\mathbb Q_{>0}$ such that we have $(n,b,s)\in\mathbb R\backslash C$ or such that $f(r)_n<q-s$ holds for all~$r\in B_s(b)\cap\mathbb Q$ with $r\leq_{\mathbb N}n$. Similarly to the proof of Lemma~\ref{pro:ContinuousOpen}, one verifies that $U$ represents the open set in question.
\end{proof}

As in the usual proof of the Tietze extension theorem, we use the following consequence of Urysohn's lemma.

\begin{lemma}\label{lem:Urysohn-plus}
    For each continuous $f\colon C\to[-1,1]$ on a closed set $C\subseteq\mathbb R$, there is a continuous $h\colon\mathbb R\to [-1/3,1/3]$ with $|f(x)-h(x)|\leq 2/3$ for all~$x\in C$.
\end{lemma}
\begin{proof}
    The sets $C_0=\{x\in C:f(x)\leq -1/3\}$ and $C_1=\{x\in C:f(x)\geq 1/3\}$ are closed by the previous lemma. For $g\colon\mathbb R\to[0,1]$ as in Corollary~\ref{lem:Urysohn}, we set
    \begin{equation*}
        h(x)=\frac23\cdot\left(g(x)-\frac12\right).
    \end{equation*}
    One readily checks the desired property.
\end{proof}

Finally, we derive the Tietze extension theorem. While we have no reversal to $\Sigma^0_2$-induction, the latter seems needed to transfer the argument from~\cite{simpson09}.

\begin{proposition}[$\mathsf{RCA}_0+\mathsf{I\Sigma}^0_2$]\label{prop:tietze}
For any continuous $f\colon C\to[-1,1]$ on a closed set~$C\subseteq\mathbb R$, there is a continuous $h\colon\mathbb R\to[-1,1]$ with $f(x)=h(x)$ for all~$x\in C$.    
\end{proposition}
\begin{proof}
Starting with $f_0=f$, we construct continuous~$f_n\colon C\to[-(2/3)^n,(2/3)^n]$ by recursion. In the step, use the previous lemma (and rescaling) to get a continuous function $h_n\colon\mathbb R\to[-2^n/3^{n+1},2^n/3^{n+1}]$ with
\begin{equation}\label{eq:tietze}
    \left|f_n(x)-h_n(x)\right|\leq\left(\frac23\right)^{n+1}\quad\text{for all }x\in C.
\end{equation}
We then set $f_{n+1}=f_n-h_n$.

In the theory $\mathsf{RCA}_0+\mathsf{I\Sigma}^0_2$, we can accommodate this construction as a strong effective recursion in the sense of~\cite{uftring-etr} (see also~\cite{westrick-etr,freund-etr}). To see this, we need to show that $f_{n+1}$ is uniformly $\Delta^0_1$-definable from~$f_n$. We successively get $\Delta^0_1$-definitions of
\begin{enumerate}[label=(\alph*)]
    \item the sets
    \begin{align*}
    U_{n,0}&=\left\{x\in C:f_n(x)>-\frac{2^n}{3^{n+1}}\right\}\cup\mathbb R\backslash C,\\
    U_{n,1}&=\left\{x\in C:f_n(x)<\frac{2^n}{3^{n+1}}\right\}\cup\mathbb R\backslash C,
    \end{align*}
    constructed as in the proof of Lemma~\ref{lem:ContinuousOpen-plus},
    \item the continuous functions $g_{n,i}\colon\mathbb R\to[0,1]$ with $U_{n,i}=g_{n,i}^{-1}((0,1])$ from the proof of Proposition \ref{pro:ContinuousOpen},
    \item the functions $g_n=g_{n,0}/(g_{n,0}+g_{n,1})$ and $h_n=(2/3)^{n+1}\cdot(g_n-1/2)$ and $f_{n+1}=f_n-h_n$, which arise from Lemmas~\ref{lem:Urysohn} and \ref{lem:Urysohn-plus} with $C_i=\mathbb R\backslash U_{n,i}$.
\end{enumerate}
Let us note that steps~(b) and~(c) always produce some continuous function~$f_{n+1}$. While its domain is not a priori guaranteed to be all of~$\mathbb R$, the rational $f_{n+1}(r)_i$ is defined for all~$r\in\mathbb Q$ and~$i\in\mathbb N$, which is important for the $\Delta^0_1$-definition in~(a).

For fixed $x\in C$, we show $x\in U_{n,0}\cup U_{n,1}$ by induction on~$n\in\mathbb N$ (as in the proof of Theorem~II.7.5 from~\cite{simpson09}). If this holds for all~$m<n$, then~$x$ lies in the domain of the~$g_m$ and hence of the~$h_m$, so that it also lies in the domain of~$f_n=f-\sum_{m<n}h_m$. As in the proof of Lemma~\ref{lem:ContinuousOpen-plus}, one can now check
\begin{equation*}
    x\in U_{n,0}\quad\Leftrightarrow\quad f_n(x)>-\frac{2^n}{3^{n+1}}\text{ or }x\in\mathbb R\backslash C.
\end{equation*}
An analogous equivalence holds for~$U_{n,1}$, which completes the induction step.

The previous paragraph establishes that all the $h_n$ have domain~$\mathbb R$. By Lemma~\ref{lem:inf-sum-fct}, we obtain a continuous function
\begin{equation*}
    h=\sum_{n=0}^\infty h_n\colon\mathbb R\to[-1,1].
\end{equation*}
Given any $q\in\mathbb Q$ and natural numbers~$i<n$, we recall $f_i=f-\sum_{j<i}h_j$ to get
\begin{equation*}
    \left|f(q)_n-h(q)_n\right|\leq\left|f_i(q)_n-h_i(q)_n\right|+\left(\frac23\right)^{i+1}.
\end{equation*}
Since~(\ref{eq:tietze}) holds by construction, all~$x\in C$ and $i\in\mathbb N$ validate
\begin{equation*}
    |f(x)-h(x)|\leq\left|f_i(x)-h_i(x)\right|+\left(\frac23\right)^{i+1}\leq2\cdot\left(\frac23\right)^{i+1}.
\end{equation*}
By letting~$i$ grow, we learn that~$f$ and~$h$ coincide on~$C$, as desired.
\end{proof}

To conclude this section, we study the Baire category theorem, by which we mean the statement that $\bigcap_{i\in\mathbb N}U_i$ is dense for any family of dense open sets $U_i\subseteq\mathbb R$. As usual, a set is dense if it intersects every non-empty open set.

\begin{proposition}\label{prop:Baire-category}
    The Baire category theorem can be derived in $\mathsf{RCA}_0$ extended by either of the following:
    \begin{enumerate}[label=(\roman*)]
    \item The principle of $\Sigma^0_2$-induction along $\mathbb N$.
    \item The $\Pi^0_1$-genericity principle $\Pi^0_1\mathsf G$ (\cite{HSS-ATM}; explained in the proof).
    \end{enumerate}
\end{proposition}
\begin{proof}
    We begin with the proof based on $\Sigma^0_2$-induction. Given rational numbers $a$ and $r>0$, we want to find a real $x\in B_r(a)$ that lies in each of the~$U_i$. To save indices, we also write $B(r,a)$ at the place of $B_r(a)$. For each $k\in\mathbb N$, we use recursion on $i\in\mathbb N$ to define rational numbers $a_i^k$ and $r_i^k>0$. In the base case, we set $a_0^k=a$ and~$r_0^k=r$. In the recursion step, let $t_i^k:=(N,b,s)\in\mathbb N\times\mathbb Q\times\mathbb Q_{>0}$ have minimal code such that
    \begin{itemize}[label=--]
        \item we have $(n,b,s)\in U_i$ for $N\leq n<k$,
        \item the intersection of $B(r_i^k/2,a_i^k)$ with $B(s,b)$ is non-empty.
    \end{itemize}
    Note that this is possible because the first condition is trivially satisfied when we have $N\geq k$. We now pick $a_{i+1}^k$ and $r_{i+1}^k$ with
    \begin{equation*}
        B\big(r_{i+1}^k,a_{i+1}^k\big)\subseteq B\big(r_i^k/2,a_i^k\big)\cap B\big(s,b\big).       
    \end{equation*}
    We assume that $a_{i+1}^k$ and $r_{i+1}^k$ depend only on $r_i^k,a_i^k$ and $s,b$ (but not on~$k$).
    By $\Sigma^0_2$-induction, we show that each $i\in\mathbb N$ admits a~$K\in\mathbb N$ with $a_i^k=a_i^K$ and $r_i^k=r_i^K$ for all~$k\geq K$. If this holds at~$i$, the map $k\mapsto t_i^k$ is non-descreasing for $k\geq K$. To conclude that the map stabilizes -- which yields the induction step~--, it suffices to note that it is bounded. Indeed, since $U_i$ is dense, there are $N,b,s$ such that we have $B(r_i^K/2,a_i^K)\cap B(s,b)\neq\emptyset$ and $(n,b,s)\in U_i$ for all~$n\geq N$, so that $k\geq K$ entails $t_i^k\leq(N,b,s)$.

    Let us deduce that the sequence of rationals $(a_k^k)$ is Cauchy and hence a real. We have $r_{i+1}^k\leq r_i^k/2$ and thus $r_i^k\leq r/2^i$. Given $\varepsilon>0$, pick an $i\in\mathbb N$ with $r/2^i\leq\varepsilon/2$. The above yields a $K\geq i$ so that $k\geq K$ entails $a_i^k=a_i^K$ and hence
    \begin{equation*}
        a^k_k\in B\big(r_k^k,a_k^k\big)\subseteq B\big(r^k_{k-1},a^k_{k-1}\big)\subseteq\ldots\subseteq B\big(r_i^k,a_i^k\big)\subseteq B\big(\varepsilon/2,a_i^K\big).
    \end{equation*}
    So we get $|a_k^k-a_l^l|<\varepsilon$ for $k,l\geq K$. Since $k>0$ entails
    \begin{equation*}
        a^k_k\in B(r_1^k,a_1^k)\subseteq B(r_0^k/2,a_0^k)=B(r/2,a),
    \end{equation*}
    we have $x:=(a^k_k)\in B_r(a)$. It remains to show that we have $x\in U_i$ for each~$i\in\mathbb N$. Again by the above, there is a $K\geq i+2$ with $t_i^k=t_i^K$ for all $k\geq K$. If we write $t_i^K=(N,b,s)$, we thus have $(n,b,s)\in U_i$ for all~$n\geq N$, which yields
    \begin{equation*}
        B\big(r_{i+1}^k,a_{i+1}^k\big)\subseteq B\big(s,b\big)\subseteq U_i\quad\text{for all }k\geq K.
    \end{equation*}
    Now $k\geq K$ also entails
    \begin{equation*}
        a_k^k\in B(r_{i+2}^k,a_{i+2}^k)\subseteq B(r_{i+1}^k/2,a_{i+1}^k).
    \end{equation*}
    Possibly after increasing~$K$, we may assume that the last ball equals $B(r_{i+1}^K/2,a_{i+1}^K)$ for all~$k\geq K$. So the real $(a^k_k)$ lies in $B(r_{i+1}^K,a_{i+1}^K)$ and hence in~$U_i$.

    We now move on to $\Pi^0_1$-genericity. Consider the tree $2^{<\omega}$ that consists of all finite sequences~$\langle\sigma_0,\ldots,\sigma_{|\sigma|-1}\rangle$ with $\sigma_i\in\{0,1\}$ (where we write $|\sigma|$ for the length of~$\sigma$). Given such a sequence and~$i\leq|\sigma|$, we put~$\sigma[i]=\langle\sigma_0,\ldots,\sigma_{i-1}\rangle$. We write $\sigma\sqsubseteq\tau$ and say that $\tau$ extends $\sigma$ if we have $\sigma=\tau[i]$ for some $i\leq|\tau|$. The $\Pi^0_1$-genericity principle $\Pi^0_1\mathsf G$ (studied in~\cite{HSS-ATM}) concerns uniformly $\Pi^0_1$-subcollections of $2^{<\omega}$. Any such collection is determined by a set $X\subseteq\mathbb N$, from which it is obtained as
    \begin{equation*}
        \mathcal D_i=\left\{\sigma\in 2^{<\omega}:(\sigma,i,n)\in X\text{ for all }n\in\mathbb N\right\}.
    \end{equation*}
    Assuming that each $\mathcal D_i$ is dense (which means that any $\sigma\in 2^{<\omega}$ admits a $\tau\in\mathcal D_i$ with $\sigma\sqsubseteq\tau$), the principle $\Pi^0_1\mathsf G$ asserts the existence of a set $G$ that is generic for the collection, i.e., such that each $i\in\mathbb N$ admits a $k\in\mathbb N$ with $G[k]\in D_i$ (where we have $G[k]=\langle\chi_G(0),\ldots,\chi_G(k-1)\rangle\in 2^{<\omega}$ with $\chi_G(j)=1$ precisely for $j\in G$).

    Again, we want to find a real~$x$ that lies in some given ball $B_r(a)$ as well as in each~$U_i$. Purely for convenience, we assume $[0,1]\subseteq B_r(a)$. For each~$\sigma\in 2^{<\omega}$, we have an interval
\begin{equation*}
    [\sigma]:=\left[p,p+\frac1{2^{|\sigma|}}\right]\subseteq[0,1]\quad\text{with}\quad p=\sum_{i<|\sigma|}\frac{\sigma_i}{2^{i+1}}.
\end{equation*}
    Let $X\subseteq\mathbb N$ consist of all tuples $(\sigma,i,n)$ with $\sigma\in 2^{<\omega}$ such that $|\sigma|$ codes a tuple $(N,b,s)$ with $[\sigma]\subseteq B_s(b)$ and $(n,b,s)\in U_i$ in case $n\geq N$. For the $\Pi^0_1$-collection determined as above, this yields
    \begin{equation*}
        \sigma\in\mathcal D_i\,\,\Leftrightarrow\,\,\text{$|\sigma|$ codes $(N,b,s)$ with $[\sigma]\subseteq B_s(b)$ and $(n,b,s)\in U_i$ for all $n\geq N$.}
    \end{equation*}
    To show that $\mathcal D_i$ is dense, consider an arbitrary~$\sigma\in 2^{<\omega}$. Given that $U_i$ is dense (as a subset of $\mathbb R$ rather than $2^{<\omega}$), there is a real $x\in[\sigma]\cap U_i$. We thus have a tuple $(N,b,s)$ with $x\in B_s(b)$ and hence $[\sigma]\cap B_s(b)\neq\emptyset$ such that $(n,b,s)\in U_i$ holds for all $n\geq N$. Pick a $\tau\in 2^{<\omega}$ with $\sigma\sqsubseteq\tau$ (thus $[\tau]\subseteq[\sigma]$) and $[\tau]\subseteq B_s(b)$. Then find an $N'\geq N$ and $\tau'\in 2^{<\omega}$ with $\tau\sqsubseteq\tau'$ such that $|\tau'|$ codes $(N',b,s)$. The sequence $\tau'$ lies in~$\mathcal D_i$, as needed to see that the latter is dense.

    By $\Pi^0_1\mathsf G$, pick a set $G\subseteq\mathbb N$ that is generic for the given collection of~$\mathcal D_i$. Then
    \begin{equation*}
        x=(x_k)\in[0,1]\subseteq B_r(a)\quad\text{with}\quad x_k=\sum_{j<k}\frac{\chi_G(j)}{2^{j+1}}
    \end{equation*}
    defines a real that lies in each of the $U_i$. To see this, take a $k\in\mathbb N$ with $G[k]\in\mathcal D_i$. Writing $(N,b,s)$ for the tuple that is coded by $k=|G[k]|$, we get
    \begin{equation*}
        x\in\big[G[k]\big]\subseteq B_s(b)\subseteq U_i,
    \end{equation*}
    which is as desired.
    \end{proof}

    Let us recall that a formula is $\mathsf r\Pi^1_2$ if it has the form
    \begin{equation*}
        \forall X\subseteq\mathbb N\big(\varphi(X)\to\exists Y\subseteq\mathbb N\,\theta(X,Y)\big)
    \end{equation*}
    with arithmetical~$\varphi$ and $\theta\in\Sigma^0_3$.

    \begin{corollary}\label{cor:BCT}
        (a) The Baire category theorem is $\mathsf r\Pi^1_2$-conservative over $\mathsf{RCA}_0$. In particular, it does not entail $\Sigma^0_2$-induction (and not even the pigeonhole principle).

        (b) The Baire category theorem does not imply $\Pi^0_1\mathsf G$ over~$\mathsf{RCA}_0$.
    \end{corollary}
    \begin{proof}
        (a) It is known that $\Pi^0_1\mathsf G$ (but not $\mathsf{IPP}$) is $\mathsf r\Pi^1_2$-conservative over $\mathsf{RCA}_0$ (see the paragraph after the proof of Theorem~4.3 in~\cite{HSS-ATM}; note however that, by the cited theorem, $\Pi^0_1\mathsf G$ and $\mathsf{IPP}$ together entail $\Sigma^0_2$-induction).

        (b) Since the Baire category theorem follows from $\Sigma^0_2$-induction, it is satisfied in the $\omega$-model of computable sets. But this model does not validate $\Pi^0_1\mathsf G$. Indeed, the latter implies the atomic model theorem and hence the omitting types theorem, which yields hyperimmune and in particular non-computable sets (as shown in Sections~4 and~5 of~\cite{HSS-ATM}).
    \end{proof}

    Since the Baire category theorem entails that the reals are uncountable (consider the dense open sets $U_i=\mathbb R\backslash\{x_i\}$ for a countable family of reals~$x_i$), the previous corollary essentially implies Proposition~\ref{prop:cantor_conservative} (modulo the inclusion of $\mathsf{WKL}$). We have decided to keep our earlier proof of Proposition~\ref{prop:cantor_conservative} in this paper, because that proof involves an unusual combination of hyperimmunity and metastability in the sense of proof mining, which may be fruitful for other applications.

\section{Two faces of the strong cohesive principle}\label{sect:cohesive}

In this section, we present the (strong) cohesive principle and some of its consequences and equivalent formulations, mostly in the form of a literature review. We emphasize that the strong cohesive principle has two `faces': On the one hand, it is equivalent to combinatorial facts related to the ascending/descending sequence principle. On the other hand, there is an equivalence with a $\Delta^0_2$-version of weak K\H{o}nig's lemma. One can argue that the latter has a more analytical flavour. Indeed, we will see that it is equivalent to the Heine-Borel theorem for our representation of open sets of reals.

To motivate the cohesive principle, we first recall Ramsey's theorem for pairs and two colours (where $[X]^n$ denotes the collection of~$n$-element subsets of~$X$):
\begin{equation*}\tag{$\mathsf{RT}^2_2$}
        \parbox[t]{.8\textwidth}{Any $c\colon[\mathbb N]^2\to\{0,1\}$ is constant on~$[H]^2$ for some infinite $H\subseteq\mathbb N$.}
\end{equation*}
We recall the following observation (made in~\cite{hirst-thesis}):

\begin{lemma}[$\mathsf{RCA}_0$]\label{lem:rt-ipp}
The pigeonhole principle $\mathsf{IPP}$ is a consequence of~$\mathsf{RT}^2_2$.
\end{lemma}
\begin{proof}
    Given $c\colon\mathbb N\to\{0,\ldots,n\}$, define $c'\colon[\mathbb N]^2\to\{0,1\}$ by
    \begin{equation*}
        c'(i,j)=1\quad\Leftrightarrow\quad c(i)=c(j).
    \end{equation*}
    Here we write $c'(i,j)$ at the place of $c'(\{i,j\})$ with $i<j$. By $\mathsf{RT}^2_2$, take an infinite $H\subseteq\mathbb N$ that is homogeneous, i.e., such that $c'$ is constant on~$[H]^2$. The constant value cannot be zero, since this would make~$c\colon H\to\{0,\ldots,n\}$ injective. So $c$ is constant on~$H$, and the value of~$c$ on~$H$ occurs infinitely often.
\end{proof}

By a breakthrough result of Patey and Yokoyama~\cite{patey-yokoyama} and its recent strengthening by these authors and Le~Hou\'erou~\cite{houerou-patey-yokoyama-Pi04}, $\mathsf{RT}^2_2$ is $\Pi^0_4$-conservative over $\mathsf{RCA}_0+\mathsf{IPP}$ and hence $\Pi^0_3$-conservative over~$\mathsf{RCA}_0$. In particular, as already shown by Seetapun (see the paper~\cite{seetapun-slaman} with Slaman), $\mathsf{RT}^2_2$ is strictly weaker than Ramsey's theorem for colourings of sets with three rather than two elements, which is equivalent to arithmetical comprehension (see Theorem~III.7.6 of~\cite{simpson09}). 

It had previously been shown by Cholack, Jockusch and Slaman~\cite{CJS-Ramsey} that $\mathsf{RT}^2_2$ is $\Pi^1_1$-conservative over~$\mathsf{RCA}_0+\mathsf I\Sigma^0_2$. The proof singles out colourings $c\colon[\mathbb N]^2\to\{0,1\}$ that are stable, i.e., for which each~$i$ admits a~$J>i$ with $c(i,j)=c(i,J)$ for all~\mbox{$j\geq J$}. Of course, $\mathsf{RT}^2_2$ follows from its restriction to stable colourings together with the statement that every colouring restricts to a stable one. This latter statement is called the cohesive part of $\mathsf{RT}^2_2$ (e.g.~in~\cite{hirschfeldt-shore}). It is also known that $\mathsf{RT}^2_2$ implies its cohesive part (as discussed below).

We will be particularly interested in a consequence of~$\mathsf{RT}^2_2$ that is known as the ascending/descending sequence principle:
\begin{equation*}\tag{$\mathsf{ADS}$}
        \parbox[t]{.8\textwidth}{In any infinite linear order, there is a strictly monotone sequence.}
\end{equation*}
Let us record the following easy observation.

\begin{lemma}[$\mathsf{RCA}_0$]\label{lem:rt-ads}
 The principle $\mathsf{ADS}$ is a consequence of~$\mathsf{RT}^2_2$.
\end{lemma}
\begin{proof}
    Given a linear order~$L=(\mathbb N,<_L)$, define $c\colon[\mathbb N]^2\to\{0,1\}$ by
    \begin{equation*}
        c(i,j)=1\quad\Leftrightarrow\quad i<_L j.
    \end{equation*}
    By $\mathsf{RT}^2_2$, take an infinite $H\subseteq\mathbb N$ such that $c$ is constant on~$[H]^2$. Let $f\colon\mathbb N\to H$ be the strictly increasing enumeration with respect to the usual order on~$\mathbb N$. Then $f$ is strictly monotone with respect to the order~$<_L$ on~$H$.
\end{proof}

Parallel to the case of~$\mathsf{RT}^2_2$, it makes sense to say that an order $(L,<_L)$ with underlying set $L\subseteq\mathbb N$ is stable if each~$i\in L$ admits a~$J\in L$ such that $i<_L j$ is equivalent to~$i<_L J$ for all~$j\geq J$ (the latter in the usual order on~$\mathbb N$). This can also be expressed as follows.

\begin{definition}
    A linear order is stable if each element has at most finitely many predecessors or at most finitely many successors.
\end{definition}

Recall that a linear order is discrete if every element that is non-minimal or non-maximal has an immediate predecessor or successor, respectively. The following will allow us to connect with the setting of Hirschfeldt and Shore~\cite{hirschfeldt-shore}.

\begin{lemma}[$\mathsf{RCA}_0$]\label{lem:stable-discrete}
Any infinite linear order that is stable has an infinite suborder that is discrete.
\end{lemma}
\begin{proof}
Consider a linear order~$L$ that is not discrete itself. By symmetry, we may assume that there is a non-minimal~$i\in L$ without an immediate predecessor. To get a strictly increasing sequence $f\colon\mathbb N\to L$, search for values $f(n)<_Lf(n+1)<_Li$ by recursion. Within~$\mathsf{RCA}_0$, we may not be able to form the image of~$f$, but we find an infinite subset of the image (by the usual proof that any infinite enumerable set has an infinite computable subset). The order on this subset is discrete.
\end{proof}

By an order of type $\omega$ or $\omega^*$, respectively, Hirschfeldt and Shore~\cite{hirschfeldt-shore} mean an infinite and discrete linear order in which all elements have only finitely many predecessors or all elements have only finitely many successors. The previous lemma shows that a strictly monotone sequence yields a suborder of type~$\omega$ or~$\omega^*$. This shows that our formulation of~$\mathsf{ADS}$ above coincides with the formulation in~\cite{hirschfeldt-shore}. Hirschfeldt and Shore also define an order of type $\omega+\omega^*$ as an infinite linear order that is discrete and stable but not of type~$\omega$ or $\omega^*$ (though it could be, e.g., isomorphic to~$\omega+1$). By the previous lemma and its proof, we get the following.

\begin{corollary}[$\mathsf{RCA}_0$]
    Any infinite linear order that is stable has a suborder that has type~$\omega$ or $\omega^*$ or $\omega+\omega^*$.
\end{corollary}

This means that the following coincides with the cohesive ascending/descending sequence principle as formulated by Hirschfeldt and Shore~\cite{hirschfeldt-shore}.
\begin{equation*}\tag{$\mathsf{CADS}$}
        \parbox[t]{.8\textwidth}{Any infinite linear order has an infinite stable suborder.}
\end{equation*}
Correspondingly, the stable ascending/descending sequence principle says that any infinite linear order that is stable contains a strictly monotone sequence (where only stable orders of type~$\omega+\omega^*$ are of interest). This stable part will play no role in the following, but we will use one special case that is computably true:

\begin{lemma}[$\mathsf{RCA}_0$]\label{lem:sads}
    Consider a stable linear order~$L$. If $I,J\subseteq L$ are infinite with $i<_Lj$ for all~$i\in I$ and~$j\in J$, then~$I$ contains a strictly increasing sequence.
\end{lemma}
Of course, symmetry also yields a strictly descreasing sequence in~$J$.
\begin{proof}
    Given that~$L$ is stable, any point in~$I$ can have only finitely many predecessor, so that it must have some successor in $I$ (in fact infinitely many). Thus a recursive search yields the desired sequence.
\end{proof}

Correspondingly, we obtain the following reformulation of~$\mathsf{CADS}$, which will be particularly useful for our applications.

\begin{proposition}\label{prop:cads-I<J}
    The following are equivalent over~$\mathsf{RCA}_0$:
    \begin{enumerate}[label=(\roman*)]
        \item The cohesive ascending/descending sequence principle~$\mathsf{CADS}$ holds.
        \item Any infinite linear order~$L$ has an infinite suborder~$S$ such that there are no infinite $I,J\subseteq S$ with $i<_Lj$ for all~$i\in I$ and~$j\in J$.
    \end{enumerate}
\end{proposition}
\begin{proof}
    For the forward direction, use~$\mathsf{CADS}$ to assume that~$L$ is stable. If $L$ itself does not have the desired property, the previous lemma yields a strictly monotone sequence. By passing to a subsequence, we may assume that the range~$S$ of this sequence exists (as in the proof of Lemma~\ref{lem:stable-discrete}). Clearly, $S$ is as required.

    For the converse, it suffices to note that any $S$ as in~(ii) is stable. To see this, consider some~$i\in S$ and apply the indicated property to the set $I=\{j\in S:j\leq_L i\}$ of predecessors and the set $J=\{j\in S:i<_L j\}$ of successors.
\end{proof}

We will see that the following yields a strong version of the cohesive ascending/\allowbreak{}descending sequence principle:
\begin{equation*}\tag{$\mathsf{StCADS}$}
        \parbox[t]{.85\textwidth}{Any infinite linear order has a suborder~$S$ such that each partition of~$S$ into finitely many intervals contains precisely one infinite interval.}
\end{equation*}
Our formulation of this principle is inspired by but not quite the same as the formulation by Hirschfeldt and Shore~\cite{hirschfeldt-shore}. The following shows that the two formulations are equivalent.

\begin{lemma}[$\mathsf{RCA}_0$]\label{lem:CADS-IPP}
    The conjunction of $\mathsf{CADS}$ and $\mathsf{IPP}$ is equivalent to~$\mathsf{StCADS}$.
\end{lemma}
\begin{proof}
    In the forward direction, use $\mathsf{CADS}$ to find a suborder as in statement~(ii) of Proposition~\ref{prop:cads-I<J}. For any partition into finitely many intervals, it immediately follows that at most one interval is infinite. The pigeonhole principle ensures that there is an infinite interval.

    For the converse, we first derive~$\mathsf{CADS}$. Given a linear order, consider a suborder as provided by~$\mathsf{StCADS}$. To see that this suborder is stable, consider the partition into predecessors and successors of any given element (as in the proof of Proposition~\ref{prop:cads-I<J}).

    Finally, we derive~$\mathsf{IPP}$. Given a function $f\colon\mathbb N\to\{0,\ldots,n\}$, consider
    \begin{equation*}
        L=\{(c,i):i\in\mathbb N\text{ and }f(i)=c\}
    \end{equation*}
    with the lexicographic order. Let $S\subseteq L$ be a suborder as provided by~$\mathsf{StCADS}$. The sets $I_c=\{(c',i)\in S:c'=c\}$ form a partition of~$L$ into intervals. We obtain one interval~$I_c$ that is infinite. This means that there are infinitely many~$i\in\mathbb N$ with $f(i)=c$, as needed for~$\mathsf{IPP}$. 
\end{proof}

We have discussed splittings of $\mathsf{RT}^2_2$ and $\mathsf{ADS}$ into stable and cohesive parts. It turns out that the latter can be subsumed under the following general cohesive principle (see Statement~7.7 of~\cite{CJS-Ramsey}). We write $A\subseteq^*B$ to indicate that $A\backslash B$ is finite.
\begin{equation*}\tag{$\mathsf{COH}$}
        \parbox{.8\textwidth}{Any sequence of sets~$R_i\subseteq\mathbb N$ admits an infinite set $C\subseteq\mathbb N$ such that we have $C\subseteq^* R_i$ or $C\subseteq^*\mathbb N\backslash R_i$ for each~$i\in\mathbb N$.}
\end{equation*}
One can strengthen~$\mathsf{COH}$ by demanding that the finitely many exceptions that occur in $C\subseteq^* R_i$ and $C\subseteq^*\mathbb N\backslash R_i$ are uniformly bounded for all~$i\leq n$ below each~$n\in\mathbb N$. Hirschfeldt and Shore~\cite{hirschfeldt-shore} show that this strong cohesive principle admits the following characterization (which we simply take as our definition):
\begin{equation*}\tag{$\mathsf{StCOH}$}
        \parbox[t]{.85\textwidth}{We have $\mathsf{COH}$ and $\mathsf{IPP}$.}
\end{equation*}
The following proof is similar to the one by of Hirschfeldt and Shore but quite a bit shorter with our formulation of~$\mathsf{StCADS}$ (and we want to refer to the proof later).

\begin{proposition}[$\mathsf{RCA}_0$; \cite{hirschfeldt-shore}]\label{prop:coh-cads}
    The principles $\mathsf{StCOH}$ and $\mathsf{StCADS}$ are equivalent. 
\end{proposition}
\begin{proof}
    In the forward direction, we prove the stronger result that $\mathsf{COH}$ implies~$\mathsf{CADS}$ (cf.~Lemma~\ref{lem:CADS-IPP}). Consider a linear order~$L=(\mathbb N,<_L)$. Let~$C$ be given by $\mathsf{COH}$ with respect to $R_i=\{j>i:i<_L j\}$. To see that $(C,<_L)$ is stable, consider any~$i\in C$. We may assume that~$C\backslash R_i$ is finite (as the argument for the complement is similar). Then $i$ has only finitely many predecessors in~$C$.

    Conversely, assume we are given sets $R_i$ for~$i\in\mathbb N$. Consider the sequences
    \begin{equation*}
        \sigma^j[n]=\langle\sigma^j_0,\ldots,\sigma^j_n\rangle\quad\text{with}\quad\sigma^j_k=\begin{cases}
            1 & \text{if }j\in R_k,\\
            0 & \text{otherwise}.
        \end{cases}
    \end{equation*}
    To define a linear order~$L$ with underlying set~$\mathbb N$, we declare that $j<_Lk$ holds precisely when~$\sigma^j[j]$ precedes~$\sigma^k[k]$ in the lexicographic order. Let~$S\subseteq L$ be a suborder as in statement~(ii) of Proposition~\ref{prop:cads-I<J}. To see that~$S$ validates~$\mathsf{COH}$, consider an arbitrary~$i\in\mathbb N$. Each 0/1-sequence~$\sigma$ of length $i+1$ determines a set
    \begin{equation*}
        I_\sigma=\left\{j\in S:j\geq i\text{ and }\sigma^j[i]=\sigma\right\}.
    \end{equation*}
    When $\sigma$ precedes~$\tau$ in the lexicographic order, we have $j<_L k$ for all~$j\in I_\sigma$ and~$k\in I_\tau$. So at most one set~$I_\sigma$ can be infinite. By the pigeonhole principle, there are~$\sigma$ and~$J\in\mathbb N$ such that we have $j\in I_\sigma$ for all $j\in S$ with~$j\geq J$. Write $\sigma_i$ for the last entry of~$\sigma$. If we have $\sigma_i=1$, we get $\sigma^j_i=1$ and hence~$j\in R_i$ for all~$j\in S$ with $j\geq J$, which yields~$S\subseteq^*R_i$. If we have $\sigma_i=0$, we get $S\subseteq^*\mathbb N\backslash R_i$.
\end{proof}

At first glance, it may not even be clear that the cohesive principle is true. As more readers may be familiar with Ramsey's theorem, we include the following.

\begin{corollary}[$\mathsf{RCA}_0$]
    The principle~$\mathsf{COH}$ follows from~$\mathsf{RT}^2_2$.
\end{corollary}
\begin{proof}
    For the ascending/descending sequence principle, it is straightforward to see that the cohesive part~$\mathsf{CADS}$ follows from the principle~$\mathsf{ADS}$ itself (as orders of type~$\omega$ and $\omega^*$ are stable, cf.~the paragraph after Lemma~\ref{lem:stable-discrete}). The result follows by the previous proposition together with Lemmas~\ref{lem:rt-ipp} and~\ref{lem:rt-ads}.
\end{proof}

Let us also state important non-implications between the principles that we have discussed (while it would go beyond the scope of this paper to recall the proofs): The implications from $\mathsf{RT}^2_2$ to~$\mathsf{ADS}$ and from the latter to~$\mathsf{StCOH}$ are strict~\cite{hirschfeldt-shore}. Also, $\mathsf{COH}$ cannot be proved in~$\mathsf{RCA}_0$ or even by weak K\H{o}nig's lemma~\cite{CJS-Ramsey}. The latter is indeed independent of~$\mathsf{RT}^2_2$ over~$\mathsf{RCA}_0$ (see~\cite{liu-RT-WKL}). It has already been mentioned that $\mathsf{RT}^2_2$ and hence also $\mathsf{StCOH}$ (even in conjunction with weak K\H{o}nig's lemma) is $\Pi^0_4$-conservative over~$\mathsf{RCA}_0+\mathsf{IPP}$. In particular, these principles are far weaker than arithmetical comprehension. Stronger conservativity results are known for principles below~$\mathsf{RT}^2_2$ (see in particular~\cite{hirschfeldt-shore} for the case of~$\mathsf{COH}$).

In the first part of this section, we have motivated the cohesive principle in terms of Ramsey's theorem, which is arguably part of combinatorics (though monotone sequences as in~$\mathsf{ADS}$ are of course also central for analysis). Before we come to a different and arguably more analytical side of the cohesive principle -- which in our setting is embodied by the Heine-Borel theorem --, we consider combinatorial results on sequences of rationals.

As a preparation, we study convergence in the tree~$2^{<\omega}$ (see the proof of Proposition~\ref{prop:Baire-category} for relevant notation). A reversal for the following two results can be obtained via Lemma~\ref{lem:kreuzer-bolzano} below.

\begin{lemma}[$\mathsf{RCA}_0+\mathsf{StCOH}$]\label{lem:cantor_compact_cauchy}
    For any sequence in $2^{<\omega}$ with infinite range, there is a sub\-sequence $(\sigma^n)_{n\in\mathbb N}$ such that each~$i\in\mathbb N$ admits an~$N\in\mathbb N$ with $\sigma^n[i]=\sigma^N[i]$ (which we take to include $|\sigma^n|,|\sigma^N|\geq i$) for all~$n\geq N$.
\end{lemma}
\begin{proof}
Since the given sequence has infinite range, we find a subsequence~$(\rho^n)$ with $|\rho^n|<|\rho^{n+1}|$ for all~$n\in\mathbb N$. This ensures in particular that $L=\{\rho^n:n\in\mathbb N\}$ exists as a set. We consider~$L$ with the lexicographic order. Let $S\subseteq L$ be a sub\-order as in statement~(ii) of Proposition~\ref{lem:CADS-IPP} (see also Proposition~\ref{prop:coh-cads}). We find a subsequence $(\sigma^n)$ of~$(\rho^n)$ that enumerates~$S$. For given~$i\in\mathbb N$, each~$\sigma\in 2^{<\omega}$ of length~$i$ determines a set
\begin{equation*}
    I_\sigma=\left\{\sigma^n:n\geq i\text{ and }\sigma^n[i]=\sigma\right\}\subseteq S.
\end{equation*}
As in the proof of Proposition~\ref{lem:cantor_compact_cauchy}, there is a~$\sigma$ and an~$N\in\mathbb N$ such that we have $\sigma^n\in I_\sigma$ for all~$n\geq N$, which thus validate $\sigma^n[i]=\sigma=\sigma^N[i]$.
\end{proof}

We will later need the following result about simultaneous convergence in~$\mathbb Q$.

\begin{proposition}[$\mathsf{RCA}_0+\mathsf{StCOH}$]\label{prop:compact_family_rationals}
    Any double sequence of rationals~$q^k_n\in[0,1]$ admits a strictly increasing $n\colon\mathbb N\to\mathbb N$ such that $\big(q^k_{n(i)}\big)_{i\in\mathbb N}$ is Cauchy for each~$k\in\mathbb N$.
\end{proposition}
\begin{proof}
As in the proof of Proposition~\ref{prop:Baire-category}, each $\sigma\in 2^{<\omega}$ yields an interval
\begin{equation*}
    [\sigma]:=\left[p,p+\frac1{2^{|\sigma|}}\right]\subseteq[0,1]\quad\text{with}\quad p=\sum_{i<|\sigma|}\frac{\sigma_i}{2^{i+1}}.
\end{equation*}
Let $\pi\colon\mathbb N^2\to\mathbb N$ be the Cantor pairing (or any computable injection). Given $\sigma\in 2^{<\omega}$ and $j\in\mathbb N$, put
\begin{equation*}
    \pi_k(\sigma)=\left\langle\sigma_{\pi(k,0)},\ldots,\sigma_{\pi(k,l-1)}\right\rangle\quad\text{with}\quad l=\min\{l'\in\mathbb N:\pi(k,l')\geq|\sigma|\}.
\end{equation*}
For every~$n\in\mathbb N$, choose $\rho^n\in 2^{<\omega}$ with $|\rho^n|=n$ and $q^k_n\in[\pi_k(\rho^n)]$ for all~$k\leq n$. To see that this is possible, note that $l_{kn}=|\pi_k(\rho^n)|$ depends only on~$k$ and~$n$. For each~$k\leq n$, we find a $\rho_{kn}\in 2^{<\omega}$ with $|\rho_{kn}|=l_{kn}$ and $q^k_n\in[\rho_{kn}]$. Given that~$\pi$ is injective, there is a~$\rho^n$ with $\pi_k(\rho^n)=\rho_{kn}$ for all~$k\leq n$.

Consider a subsequence $\sigma^i=\rho^{n(i)}$ as provided by the previous lemma. To show that $(q^k_{n(i)})_{i\in\mathbb N}$ is Cauchy, we consider an arbitrary~$\varepsilon=2^{-L}>0$. Put
\begin{equation*}
l=\max\{\pi(k,l'):l'\leq L\}+1.
\end{equation*}
Now take an $I\geq k$ such that all $i\geq I$ validate $\sigma^i[l]=\sigma^I[l]$ and hence
\begin{equation*}
q^k_{n(i)}\in\left[\pi_k\left(\rho^{n(i)}\right)\right]=\left[\pi_k\left(\sigma^i\right)\right]\subseteq\left[\pi_k\left(\sigma^i[l]\right)\right]=\left[\pi_k\left(\sigma^I[l]\right)\right].
\end{equation*}
Due to $|\pi_k(\sigma^I[l])|>L$, we get $|q^k_{n(i)}-q^k_{n(j)}|<2^{-L}=\varepsilon$ for all~$i,j\geq I$.
\end{proof}

The following reversal is due to Kreuzer. We give a different and simple proof.

\begin{lemma}[\cite{kreuzer-bolzano}]\label{lem:kreuzer-bolzano}
The following are equivalent over~$\mathsf{RCA}_0$:
\begin{enumerate}[label=(\roman*)]
    \item The strong cohesive principle~$\mathsf{StCOH}$.
    \item Any bounded sequence of rationals has a subsequence that is Cauchy.
\end{enumerate}
\end{lemma}
\begin{proof}
    That~(i) implies~(ii) is a special case of the previous result. For the converse, it is staightforward to see that~(ii) implies~$\mathsf{IPP}$. So it remains to establish~$\mathsf{CADS}$. Given an infinite linear order~$L$, we fix an enumeration~$L=\{\rho^n:n\in\mathbb N\}$ and construct an embedding~$f\colon L\to\mathbb Q$. By~(ii), we get a strictly increasing~\mbox{$n\colon\mathbb N\to\mathbb N$} such that the sequence~$(f(\rho^{n(i)}))_{i\in\mathbb N}$ is Cauchy. If the enumeration of~$L$ is increasing (with respect to the usual order on~$\mathbb N$), then $S=\{\rho^{n(i)}:i\in\mathbb N\}$ exists as a set. In $\{f(\rho):\rho\in S\}$ and hence in~$S$, at most one point (namely the limit of the Cauchy sequence) can have infinitely many predecessors and infinitely many successors. If we omit this point, we have an infinite stable suborder of~$L$.
\end{proof}

The coherent part of the ascending/descending sequence principle can also be characterized as in part~(b) of the following result, which yields an interesting comparison with the full principle in part~(a).

\begin{proposition}[$\mathsf{RCA}_0$]
(a) The principle~$\mathsf{ADS}$ holds precisely if every sequence in~$\mathbb Q$ has a monotone subsequence.

(b) The principle~$\mathsf{StCADS}$ holds precisely if every sequence in~$\mathbb Q$ has a subsequence that is almost increasing or almost decreasing, which means that each~$\varepsilon>0$ admits an~$N\in\mathbb N$ such that all $n>m\geq N$ validate $q_n\leq q_m+\varepsilon$ or $q_n\geq q_m-\varepsilon$, respectively.
\end{proposition}
\begin{proof}
    (a) Consider a sequence~$(q_n)\subseteq\mathbb Q$. As before, passing to a subsequence allows us to assume that~$L=\{q_n:n\in\mathbb N\}$ exists as a set. If $L$ is finite, we can conclude by the pigeonhole principle, which is a consequence of~$\mathsf{ADS}$ (see Proposition~4.5 of~\cite{hirschfeldt-shore}). Otherwise, $\mathsf{ADS}$ yields a strictly monotone sequence in~$L$ (with respect to the usual order on~$\mathbb Q$). A subsequence of the latter is a subsequence of~$(q_n)$.

    Conversely, let $L$ be any infinite linear order. Pick an enumeration~$(\rho^n)$ of~$L$. Then construct an embedding~$f\colon L\to\mathbb Q$. The statement in~(a) yields a strictly increasing~$n\colon\mathbb N\to\mathbb N$ such that~$(f(\rho^{n(i)}))$ and hence~$(\rho^{n(i)})$ is (strictly) monotone.

    (b) Consider a sequence~$(r_n)\subseteq\mathbb Q$. If this sequence is unbounded, we recursively find a subsequence that is even strictly monotone. In the remaining case, the previous lemma yields a subsequence~$(q_n)$ that is Cauchy. This sequence is almost increasing as well as almost decreasing.

    For the converse, we also reduce to the previous lemma. Consider a sequence of rationals~$q_n\in[-B,B]$. Passing to a subsequence, we assume that~$(q_n)$ is almost decreasing. To show that it is even Cauchy, consider some~$\varepsilon>0$. Take~$M\in\mathbb N$ such that we have $q_n<q_m+\varepsilon/2$ for all~$n>m\geq M$. Then use~$\Sigma^0_1$-induction to find an~$N\geq M$ for which $q_N$ is close to minimal, by which we mean that $q_n>q_N-\varepsilon/2$ holds for all~$n\geq M$. For any $n\geq N$, we get $q_n\in(q_N-\varepsilon/2,q_N+\varepsilon/2)$.
\end{proof}

The previous proof reveals that almost monotone sequences are enough to show convergence. This may provide some intuition why we only need cohesiveness rather than the ascending/descending sequence principle for our applications in analysis.

In the following, we discuss a side of cohesiveness that is arguably more analytical or topological. Specifically, we prove an equivalence with the Heine-Borel theorem. The result may not be completely surprising in view of two known facts. First, the cohesive principle is equivalent to a $\Delta^0_2$-version of weak K\H{o}nig's lemma (which states that every infinite binary $\Delta^0_2$-tree has a $\Delta^0_2$-path). On the level of computability, this goes back to work of Jockusch and Stephan~\cite{jockusch-stephan-coh}, which was refined by Brattka, Gherardi and Marcone~\cite{brattko-gherardi-marcone}. As a result of reverse mathematics, it has been established by Belanger~\cite{belanger-cohesive}. Secondly, the `regular' version of weak K\H{o}nig's lemma -- for trees that are (computable and hence) given as sets -- is equivalent to the Heine-Borel theorem under the traditional approach (with reals as fast Cauchy sequences and a corresponding encoding of open sets; see Theorem~IV.1.2 of~\cite{simpson09}). Since our approach with slow Cauchy sequences often `adds a quantifier' (so that, e.g., strict inequalities between reals are $\Sigma^0_2$ rather than~$\Sigma^0_1$), it should be possible to obtain the following theorem as a lift of the traditional result. However, we prefer to give a more direct proof. The argument for the reversal seems particularly `mathematical' in the sense that it could similarly appear in an analysis textbook, where it would be used to derive the Bolzano-Weierstrass theorem.

\begin{theorem}\label{thm:Heine-Borel}
    The following are equivalent over~$\mathsf{RCA}_0$:
    \begin{enumerate}[label=(\roman*)]
    \item The Heine-Borel theorem holds, i.e., any covering $[0,1]\subseteq\bigcup_{n\in\mathbb N}U_i$ by open sets~$U_i$ (cf.~Definition~\ref{def:open}) admits a finite subcovering $[0,1]\subseteq\bigcup_{n\leq N}U_i$.
    \item We have the strong cohesive principle~$\mathsf{StCOH}$.
    \end{enumerate}
\end{theorem}
\begin{proof}
    We first show that~(i) implies~(ii). In view of Lemma~\ref{lem:kreuzer-bolzano}, it suffices to~prove that any sequence of rationals~$q_i\in[0,1]$ has a subsequence that is Cauchy. Given a natural number~$N$ as well as rationals~$a$ and $r>0$, we set
    \begin{equation*}
        U_{N,a,r}=\{(n,a,r)\in\mathbb N\times\mathbb Q\times\mathbb Q_{>0}:q_i\notin B_r(a)\text{ for }N\leq i\leq n\}.
    \end{equation*}
    No finite collection of (open sets represented by) the $U_{N,a,r}$ can cover~$[0,1]$. To see this, note that $q_i\in U_{N,a,r}$ requires~$q_i\in B_r(a)$ as well as $(n,a,r)\in U_{N,a,r}$ for all sufficiently large~$n$ (here in fact for all~$n$), which forces~$i<N$. So by~(i), we get a real number~$x=(x_j)\in[0,1]$ that is not contained in any of the sets~$U_{N,a,r}$. This means that all~$N,a,r$ validate
    \begin{equation*}
        x\in B_r(a)\quad\Rightarrow\quad q_i\in B_r(a)\text{ for some }i\geq N.
    \end{equation*}
    We recursively search for two strictly increasing sequences of indices~$i(n)$ and~$j(n)$ with $|q_{i(n)}-x_{j(n)}|<1/n$. In order to see that these always exist, first pick $j(n)$ such that we have $|x-x_{j(n)}|\leq1/3n$ and $j(n)>j(n-1)$ in case~$n>0$. Now take a rational~$a$ with $x\in B_{1/3n}(a)$. By the above, we find an~$i(n)$ with $q_{i(n)}\in B_{1/3n}(a)$ and $i(n)>i(n-1)$ in case~$n>0$. We indeed get
    \begin{equation*}
        \left|q_{i(n)}-x_{j(n)}\right|\leq\left|q_{i(n)}-a\right|+\left|a-x\right|+\left|x-x_{j(n)}\right|<\frac1n.
    \end{equation*}
    It follows that the~$q_{i(n)}$ form a Cauchy sequence (with limit~$x$). Indeed, for any~$\varepsilon>0$ we find an~$N\geq 3/\varepsilon$ such that all $m,n\geq N$ validate $|x_{j(m)}-x_{j(n)}|\leq\varepsilon/3$ for all~$n\geq N$. When we have $m,n\geq N$, we obtain
    \begin{equation*}
        \left|q_{i(m)}-q_{i(m)}\right|\leq\left|q_{i(m)}-x_{j(m)}\right|+\left|x_{j(m)}-x_{j(n)}\right|+\left|x_{j(n)}-q_{i(n)}\right|<\varepsilon.
    \end{equation*}

    To establish that~(ii) implies~(i), we consider a family of open sets~$U_i$. Say that a pair $(a,r)\in\mathbb Q\times\mathbb Q_{>0}$ is $[n,N]$-active if there is some~$i<n$ such that $(n',a,r)\in U_i$ holds for $n\leq n'\leq N$. We write $\mathcal B_{n,N}$ for the set of balls~$B_r(a)$ such that~$(a,r)$ has code below~$n$ and is $[n,N]$-active. Assuming that $(U_i)$ has no finite subcover, each~$n$ admits an~$N=N(n)\geq n$ such that $[0,1]$ is not covered by~$\mathcal B_{n,N}$. To see this, assume the claim is false for some~$n$. Since~$\mathcal B_{n,N}$ shrinks as the second index grows, it must stabilize at some~$N$. For any~$B_r(a)\in\mathcal B_{n,N}$ this means that the pair $(a,r)$ is $[n,N']$-active for all~$N'\geq N$. But then we have~$B_r(a)\subseteq U_i$ for some~$i<n$. Indeed, if this was false, we would find~$n_i'\geq n$ with $(n_i',a,r)\notin U_i$. By bounded collection, we would get an~$N'$ above all~$n_i'$ with~$i<n$. So $(a,r)$ would not be~$[n,N']$-active, against the assumption that~$\mathcal B_{n,N}$ has stabilized. Now since~$\mathcal B_{n,N}$ is a cover of~$[0,1]$, the same holds for~$(U_i)_{i<n}$, which contradicts our assumption.

Let us choose rationals~$x_n\in[0,1]$ that are not contained in any ball in~$\mathcal B_{n,N(n)}$. By~$\mathsf{StCOH}$ we learn that some subsequence $x=(x_{n(i)})\in[0,1]$ is Cauchy and hence a real number (see again Lemma~\ref{lem:kreuzer-bolzano}). \mbox{We show that $x$ is contained in no~$U_j$.} Assuming the contrary, the $x_{n(i)}$ eventually lie in a ball~$B_r(a)$ such that $(n',a,r)\in U_j$ holds for all~$n'$ above some~$n>\max\{j,(a,r)\}$. For any~$i$ with $n(i)\geq n$, it follows that we have $B_r(a)\in\mathcal B_{n(i),N(n(i))}$, which contradicts the choice of~$x_{n(i)}$.
\end{proof}

To conclude this section, we show that a sequential version of the Heine-Borel theorem is strong. In this respect, our setting is different from the classical approach with fast Cauchy sequences, where the sequential version is still as weak as $\mathsf{WKL}_0$ (see Theorem~IV.1.6 of~\cite{simpson09}). Similar observations can be made in other cases, e.g., for the intermediate value theorem. The relevance of the sequential results has been discussed in the introduction.

\begin{proposition}[$\mathsf{RCA}_0$]\label{prop:heine-borel-unif}
    The following are equivalent over $\mathsf{RCA}_0$:
    \begin{enumerate}[label=(\roman*)]
        \item Given any open sets $U_{ni}$ with $[0,1]\subseteq\bigcup_{i\in\mathbb N}U_{ni}$ for all~$n\in\mathbb N$, there is a function~$h\colon\mathbb N\to\mathbb N$ such that all~$n$ validate $[0,1]\subseteq\bigcup_{i\leq h(n)}U_{ni}$.
        \item We have arithmetical comprehension (i.e., the main axiom of~$\mathsf{ACA}_0$).
    \end{enumerate}
\end{proposition}
\begin{proof}
    In order to see that~(i) implies~(ii), we show that the range~$\rng(f)$ of any function~$f\colon\mathbb N\to\mathbb N$ exists as a set (cf.~Lemma~III.1.3 of~\cite{simpson09}). Let $U_{n0}$ consist of all triples~$(k,1/2,1)$ such that there is no~$m<k$ with $f(m)=n$. If we also write~$U_{n0}$ for the open set that this represents, then we have
    \begin{equation*}
        U_{n0}=\begin{cases}
            \left(-1/2,3/2\right) & \text{if $n\notin\rng(f)$},\\
            \emptyset & \text{otherwise}.
        \end{cases}
    \end{equation*}
    Also, it is straightforward to represent open sets with
    \begin{equation*}
        U_{n,i+1}=\begin{cases}
            \left(-1/2,3/2\right) & \text{if $f(i)=n$},\\
            \emptyset & \text{otherwise}.
        \end{cases}
    \end{equation*}
    We clearly have $[0,1]\subseteq\bigcup_{i\in\mathbb N}U_{ni}$ for all~$n\in\mathbb N$. Now if we have $[0,1]\subseteq\bigcup_{i\leq h(n)}U_{ni}$, then we obtain
    \begin{equation*}
        \rng(f)=\{n\in\mathbb N:f(i)=n\text{ for some }i<h(n)\},
    \end{equation*}
    where the right side can be formed by~$\Delta^0_1$-comprehension.

    To establish that~(ii) implies~(i), it suffices to turn~$[0,1]\subseteq\bigcup_{i\leq I}U_{ni}$ into an arithmetical property. Using the Cantor pairing function $\pi$, we define
    \begin{equation*}
        U_{nk}'=\big\{(m,a,r)\in U_i:(a,r)\text{ has code at most }j\}\quad\text{for}\quad k=\pi(i,j).
    \end{equation*}
    Crucially, the represented open sets are finite unions of balls with rational endpoints. We thus have an arithmetical definition of
    \begin{equation*}
        h(n)=\min\left\{K\in\mathbb N:[0,1]\subseteq\bigcup_{k\leq K}U_{nk}'\right\},
    \end{equation*}
    which is always defined when we have $[0,1]\subseteq\bigcup_{i\in\mathbb N}U_{ni}$ and hence $[0,1]\subseteq\bigcup_{k\in\mathbb N}U_{nk}'$, due to the previous theorem. As we have $U_{nk}'\subseteq U_{ni}$ for $k=\pi(i,j)\geq i$, we can conclude that we have $[0,1]\subseteq\bigcup_{i\leq h(n)}U_{ni}$ for all~$n\in\mathbb N$.
\end{proof}

\section{Further properties of continuous functions}

In Section~\ref{sect:functions1} we have begun the analysis of continuous functions in our setting. Here we continue with results about continuous functions that rely on open sets or on the cohesive principle, which were discussed in Sections~\ref{sect:open} and~\ref{sect:cohesive}, respectively.

\begin{theorem}\label{thm:StCADSContinuous}
The following are equivalent over $\mathsf{RCA}_0$:
\begin{enumerate}[label=(\roman*)]
\item The strong cohesive principle $\mathsf{StCOH}$ holds.
\item Any continuous function~$f\colon[0,1]\to\mathbb R$ is uniformly continuous.
\item Any continuous function $f\colon[0,1]\to\mathbb R$ has a maximum and a minimum. 
\item Any continuous function $f\colon [0,1] \to \mathbb{R}$ is bounded.
\end{enumerate}
\end{theorem}
\begin{proof}
    We first show that~(i) implies~(ii). According to Definition~\ref{def:continuous}, a continuous function $f\colon[0,1]\to\mathbb R$ is given as a family of rationals $f(q)_i$ for $q\in\mathbb Q$ and $i\in\mathbb N$, where all $x\in[0,1]$ and $\varepsilon>0$ admit an~$N\in\mathbb N$ with
\begin{equation*}
    \big|f(p)_i-f(q)_j\big|<\varepsilon\quad\text{for all rationals }p,q\in B_{1/N}(x)\text{ and all }i,j\geq N.
\end{equation*}
    In the presence of cohesiveness and for the compact domain~$[0,1]$, we show that this implies the following uniform condition (which we display for future reference):
    \begin{equation}\label{eq:unif-cont}
    \parbox{.8\textwidth}{Every~$\varepsilon>0$ admits an~$N\in\mathbb N$ such that $|f(p)_i-f(q)_j|<\varepsilon$ holds for all rationals~$p,q\in[0,1]$ with $|p-q|<1/N$ and for all~$i,j\geq N$.}
    \end{equation}
    Write $\mathbb Q[n]$ for the set of rationals with code below~$n\in\mathbb N$. For fixed~$\varepsilon>0$, let~$U_m$ be the set of triples~$(n,a,1/m)$ with $a\in\mathbb Q$ and $m,n\in\mathbb N$ that validate
\begin{equation*}
|f(p)_i-f(q)_j|<\varepsilon\quad\text{for all }p,q\in B_{2/m}(a)\cap\mathbb Q[n]\text{ and }m\leq i,j\leq n.
\end{equation*}
To see that any real number~$x\in[0,1]$ lies in one of the represented open sets~$U_m$, take a~$k\in\mathbb N$ with $|f(p)_i-f(q)_j|<\varepsilon$ for all~$p,q\in B_{1/k}(x)\cap\mathbb Q$ and all~$i,j\geq k$. For $m=3k$ we pick a rational $a\in B_{1/m}(x)$, which yields $B_{2/m}(a)\subseteq B_{1/k}(x)$. So we have $(n,a,1/m)\in U_m$ for all~$n\in\mathbb N$, which yields $x\in B_{1/m}(a)\subseteq U_m$. In view of Theorem~\ref{thm:Heine-Borel}, we can use the Heine-Borel theorem to find an~$N\in\mathbb N$ such that $[0,1]$ is already covered by the open sets~$U_m$ with~$m\leq N$. We show that this~$N$ validates~(\ref{eq:unif-cont}) for our fixed~$\varepsilon$. Consider rationals~$p,q\in[0,1]$ with $|p-q|<1/N$. Take an~$m\leq N$ with $p\in U_m$, which means that $p$ lies in some ball~$B_{1/m}(a)$ with~$(n,a,1/m)\in U_m$ for all sufficiently large~$n$ (hence in fact for all~$n$). In view of $q\in B_{2/m}(a)$, the definition of~$U_m$ yields~$|f(p)_i-f(q)_j|<\varepsilon$ for all~$i,j\geq N$.

To complete the direction from~(i) to~(ii), we show that uniform continuity follows from (\ref{eq:unif-cont}). Assume that the latter holds for some~$\varepsilon>0$ and~$N\in\mathbb N$. Let us consider arbitrary reals~$x,y\in [0,1]$ with $|x-y|<1/(3N)$. We may assume that these are represented as Cauchy sequences $x=(x_n)$ and $y=(y_n)$ with rational approx\-imations $x_n,y_n\in[0,1]$. Take an~$M\geq N$ such that $n\geq M$ implies $|x_n-x|<1/(3N)$ and $|y_n-y|<1/(3N)$, so that we get $|x_n-y_n|<1/N$ and thus
\begin{equation*}
\left|f(x_n)_n-f(y_n)_n\right|<\varepsilon.
\end{equation*}
Since $f(x)$ is defined as the Cauchy sequence~$(f(x_n)_n)_{n\in\mathbb N}$ (see Definition~\ref{def:continuous}), this yields $|f(x)-f(y)|\leq\varepsilon$, as needed for uniform continuity.

Next, we show that~(i) implies~(iii). By symmetry, it suffices to prove that any continuous function~$f\colon[0,1]\to\mathbb R$ has a maximum. For $n\in\mathbb N$, pick $x_n=2^{-n}\cdot i$ with $0\leq i\leq 2^n$ such that $f(x_n)_n$ is as large as possible. Using cohesiveness in the form of Lemma~\ref{lem:kreuzer-bolzano}, we get a subsequence $x=(x_{n(i)})_{i\in\mathbb N}$ that is Cauchy and hence a real. To see that $f(x)$ is the maximum value, we prove $f(x)\geq f(y)-\varepsilon$ for arbitrary~$y\in[0,1]$ and $\varepsilon>0$. As before, we may assume that~$y$ is given as a Cauchy sequence $(y_n)$ with $y_n\in[0,1]$ for all~$n\in\mathbb N$. By the proof that~(i) implies~(ii), we find an~$n$ such that $|f(p)_i-f(q)_j|<\varepsilon/4$ holds for any~$p,q\in[0,1]$ with~$|p-q|<2^{-n}$ and any~$i,j\geq n$.  Now pick an~$i\geq n$ such that we have $|f(x)-f(x_{n(i)})_i|<\varepsilon/4$ as well as $|f(y)-f(y_{n(i)})_{n(i)}|<\varepsilon/4$. For some rational $q=2^{-n(i)}\cdot j$ with $0\leq j\leq 2^{n(i)}$, we have $|y_{n(i)}-q|<2^{-n(i)}\leq 2^{-n}$. By the choice of~$x_{n(i)}$, we get
\begin{equation*}
f(x)>f(x_{n(i)})_i-\frac\varepsilon4>f(x_{n(i)})_{n(i)}-\frac\varepsilon2\geq f(q)_{n(i)}-\frac\varepsilon2>f(y_{n(i)})_{n(i)}-\frac{3\varepsilon}4>f(y)-\varepsilon,
\end{equation*}
as desired.

Given that~(iv) is an immediate consequence of~(iii), it remains to show that each of~(ii) and~(iv) implies~(i). We argue by contraposition. In view of of Lemma~\ref{lem:kreuzer-bolzano}, a failure of~(i) gives us a sequence of rationals~$q_n\in[0,1]$ of which no subsequence is Cauchy. To falsify~(ii) and~(iv), we construct a continuous function~$f\colon[0,1]\to\mathbb R$ that is neither uniformly continuous nor bounded. The idea is to define~$f$ as a maximum of ever higher and narrower hats over the~$q_n$. Specifically, for $r\in\mathbb Q$ and $n\in\mathbb N$ we consider the approximations
\begin{equation*}
    f(r)_i=\max_{n\leq i} g_n(r)\quad\text{with}\quad g_n(r)=\max\left(0,2^n\cdot\big(1-2^n\cdot|q_n-r|\right)\big).
\end{equation*}
To see that these represent a continuous function~$f\colon[0,1]\to\mathbb R$, we consider an arbitrary real~$x\in[0,1]$. There are~$N\in\mathbb N$ and $a,r\in\mathbb Q$ with $r>0$ such that we have $x\in B_r(a)$ but $q_n\notin B_r(a)$ for all~$n\geq N$. Indeed, if this was false, the proof of Theorem~\ref{thm:Heine-Borel} would yield a subsequence of~$(q_n)$ that is Cauchy. Equivalently, we have a $\delta>0$ such that $q_n\notin B_\delta(x)$ holds for all~$n\geq N$. We may assume~$2^{-N}\leq\delta/2$. For $r\in B_{2^{-N}}(x)$ and $n\geq N$, we get $|q_n-r|>2^{-n}$ and thus~$g_n(r)=0$. So we have
\begin{equation*}
f(r)_i=\max_{n\leq N}g_n(r)\quad\text{for }r\in B_{2^{-N}}(x)\text{ and }i\geq N.
\end{equation*}
Still assuming~$i\geq N$ as well as $p,q\in B_{2^{-N}}(x)$, we can conclude
\begin{equation*}
    \left|f(p)_i-f(r)_j\right|\leq\max_{n\leq N}\left|g_n(p)-g_n(r)\right|\leq 2^{2N}\cdot|p-r|.
\end{equation*}
Given some~$\varepsilon>0$, we now pick an~$M>\max\{2^N,2^{2N+1}/\varepsilon\}$ to get
\begin{equation*}
    \left|f(p)_i-f(r)_j\right|\leq 2^{2N}\cdot\frac2M<\varepsilon\quad\text{for all }p,r\in B_{1/M}(x)\cap\mathbb Q\text{ and }i,j\geq M.
\end{equation*}
According to Definition~\ref{def:continuous}, this is the condition that representations of continuous functions need to satisfy.

The continuous function $f\colon[0,1]\to\mathbb R$ with the above representation is clearly unbounded: For any~$n\in\mathbb N$ and all $i\geq n$ we have $f(q_n)_i\geq g_n(q_n)=2^n$, so that the real number $f(q_n)=(f(q_n)_i)_{i\in\mathbb N}$ is also at least~$2^n$. It remains to show that~$f$ is not uniformly continuous. Towards a contradiction, assume there is a $\delta>0$ such that we have $|f(x)-f(y)|<1/2$ for all~$x,y\in[0,1]$ with $|x-y|<\delta$. There must be an~$n\in\mathbb N$ with $2^{-n-1}<\delta$ such that $q_N\notin B_{2^{-n}}(q_n)$ holds for all~$N>n$. Otherwise, we could recursively search for a strictly increasing sequence of indices $n(i)$ that validate $|q_{n(i+1)}-q_{n(i)}|<2^{-n(i)}\leq 2^{-i}$, which would yield a subsequence $(q_{n(i)})_{i\in\mathbb N}$ that is Cauchy, against our assumption. Now for~$n$ as indicated, we set $x=q_n$ and choose $y\in[0,1]$ with $|x-y|=2^{-n-1}<\delta$. For $i\geq n$ we have $f(q_n)_i\geq g_n(q_n)=2^n$, which yields~$f(x)\geq 2^n$. For $N>n$, we have
\begin{equation*}
    2^{-n}\leq|q_N-q_n|\leq|q_N-y|+|y-q_n|=|q_N-y|+2^{-n-1},
\end{equation*}
which yields $|q_N-y|\geq 2^{-n-1}\geq 2^{-N}$ and hence~$g_N(y)=0$. Let us also observe that $m<n$ entails $g_m(y)\leq 2^m\leq 2^{n-1}$ and that we have
\begin{equation*}
    g_n(y)=2^n\cdot(1-2^n\cdot|q_n-y|)=2^{n-1}.
\end{equation*}
So we have $f(y)_i\leq 2^{n-1}$ for any~$i\in\mathbb N$, which shows~$f(y)\leq 2^{n-1}$. Together, it follows that we have $|f(x)-f(y)|\geq 2^n-2^{n-1}=2^{n-1}\geq1/2$ despite $|x-y|<\delta$, which is the desired contradiction with uniform continuity.
\end{proof}

For later use, we record the following.

\begin{remark}\label{rem:UCtoBounded}
Already in $\mathsf{RCA}_0+\mathsf{IPP}$, we can show that a uniformly continuous function $f\colon[0,1]\to\mathbb R$ is bounded. Concretely, the assumption yields an $N\in\mathbb{N}$ such that $\vert f(x)-f(y)\vert<1$ holds for all $x,y\in [0,1]$ with $\vert x-y\vert\leq 1/N$. In view of Lemma~\ref{lem:fin-sets-bounded}, we use $\mathsf{IPP}$ to form $K:=\max\{\vert f(a_i)\vert : 0\leq i\leq N\}$ with $a_i:=i/N$. Given any $x\in [0,1]$, we find an $i\leq N$ with $|x-a_i|<1/N$ and hence
\[
\vert f(x)\vert\leq\vert f(x)-f(a_i)\vert +\vert f(a_i)\vert<1+K,
\]
so that $f$ is indeed bounded.
\end{remark}

Using the strong cohesive principle, we can also develop the Riemann integral for continuous functions. Under a partition~$P$ of $[0,1]$, we understand a finite collection of pairwise disjoint intervals with union~$[0,1]$. We put $\Delta P=\max\{\Delta I:I\in P\}$, where $\Delta I$ refers to the usual length of the interval~$I$.

\begin{definition}
A continuous function $f\colon[0,1]\to\mathbb R$ is Riemann integrable if there is a real number~$\int_{[0,1]} f(x)\,dx$ as follows: Any $\varepsilon>0$ admits a $\delta>0$ such that
\begin{equation*}
    \left|\int_{[0,1]} f(x)\,dx-\sum_{I\in P} f(t_I)\cdot\Delta I\right|<\varepsilon
\end{equation*}
holds for any partition~$P$ of $[0,1]$ with $\Delta P<\delta$ and any choice of $t_I\in I$ for~$I\in P$.
\end{definition}

We note that already the definition relies on the pigeonhole principle, which is needed to form finite sums of reals (see Lemma~\ref{lem:fin-sets-bounded} and the remark that follows it). For this reason, we prove the following equivalence over a stronger base theory. Let us recall that the cohesive principle $\mathsf{COH}$ and its strong variant $\mathsf{StCOH}$ are equivalent in the presence of $\mathsf{IPP}$.

\begin{theorem}[$\mathsf{RCA}_0+\mathsf{IPP}$]\label{thm:RiemannInteg}
The following are equivalent:
\begin{enumerate}[label=(\roman*)]
    \item The cohesive principle $\mathsf{COH}$ holds.
    \item Every continuous function~$f\colon[0,1]\to\mathbb R$ is Riemann integrable.
\end{enumerate}
\end{theorem}
\begin{proof}
We first show that (i) implies~(ii). To define a real~$x=(x_n)$ with the property of $\int_{[0,1]} f(x)\,dx$, we consider the partitions into intervals of length $2^{-n}$, i.e., we set
\begin{equation*}
    x_n=\sum_{i=0}^{2^n-1} f\left(i\cdot 2^{-n}\right)_n\cdot2^{-n}.
\end{equation*}
To see that this yields a Cauchy sequence, consider some $\varepsilon>0$. With $\mathsf{StCOH}$ available, (\ref{eq:unif-cont}) lets us pick an~$N\in\mathbb N$ with $|f(p)_m-f(q)_n|<\varepsilon$ for all $m,n\geq N$ and all rationals $p,q\in[0,1]$ with $|p-q|<2^{-N}$. Given $n>m\geq N$, we get
\begin{equation*}
    \left|f\left(i\cdot 2^{-m}\right)_m-f\left((i\cdot 2^{n-m}+j)\cdot 2^{-n}\right)_n\right|<\varepsilon\quad\text{for}\quad 0\leq j<2^{n-m},
\end{equation*}
which yields
\begin{multline*}
    \left|f(i\cdot 2^{-m})_m\cdot2^{-m}-\sum_{j=0}^{2^{n-m}-1}f\left((i\cdot 2^{n-m}+j)\cdot 2^{-n}\right)_n\cdot 2^{-n}\right|\\
    {}\leq2^{-n}\cdot\sum_{j=0}^{2^{n-m}-1}\left|f\left(i\cdot 2^{-m}\right)_m-f\left((i\cdot 2^{n-m}+j)\cdot 2^{-n}\right)_n\right|<\varepsilon\cdot 2^{-m}
\end{multline*}
and then $|x_m-x_n|<\varepsilon$, as needed.

Let us now show that~$x$ has the desired property. We say that a partition~$P$ with points $t_I\in I$ for~$I\in P$ is rational if the~$t_I$ and the endpoints of the intervals~$I$ are rational numbers. Given $\varepsilon>0$, choose~$N\in\mathbb N$ as above. We claim that
\begin{equation*}
    \left|x_n-\sum_{I\in P}f(t_I)_n\cdot\Delta I\right|<\varepsilon
\end{equation*}
holds for any rational partition~$P$ with $\Delta P\leq 2^{-N-1}$ and for all~$n>N$. To see this, take a common refinement of~$P$ and the partition with endpoints~$i/(n+1)$. Then argue as in the proof that $(x_n)$ is Cauchy. By fixing~$P$ while~$n$ goes to infinity, we get $|x-\sum_{I\in P} f(t_I)\cdot\Delta I|\leq\varepsilon$. Finally, in the presence of $\mathsf{IPP}$, all partitions admit arbitrarily good approximations by rational ones.

We now assume~(ii) and derive~(i). Assuming the latter fails, Theorem \ref{thm:StCADSContinuous} gives us a continuous~$f\colon[0,1]\to[0,\infty)$ that is unbounded. If $f$ is Riemann integrable, there are $B$ and $N>0$ such that we have
\begin{equation*}
    \sum_{I\in P}f(t_I)\cdot\Delta I<B
\end{equation*}
for every partition~$P$ with $\Delta P\leq 1/N$ and any choice of $t_I\in I$. Let~$P$ be the partition into intervals with endpoints $i/N$ and pick $t_I\in I$ for $I\in P$ such that we have $f(t_I)\geq B\cdot N$ for some~$I$. We get
\begin{equation*}
    \sum_{I\in P}f(t_I)\cdot\Delta I\geq B\cdot N\cdot\frac1N=B,
\end{equation*}
which yields a contradiction.
\end{proof}

The picture is similar for the Weierstrass approximation theorem. Using the pigeonhole principle, we get the following result for uniformly continuous functions.

\begin{lemma}[$\mathsf{RCA}_0+\mathsf{IPP}$]\label{lem:WeierstrassUniform}
If $f\colon[0,1]\to\mathbb{R}$ is uniformly continuous, any $\varepsilon>0$ admits a polynomial $p$ such that $\vert f(x)-p(x)\vert <\varepsilon$ holds for all $x\in [0,1]$.
\end{lemma}
\begin{proof}
We recall the well-known proof via Bernstein polynomials, so that the reader can see that it goes through in our system. In the presence of~$\mathsf{IPP}$ (cf.~Remark~\ref{rmk:fin-sums}), we can define 
\begin{equation*}
    B_n(f)(x):=\sum_{k=0}^nf\left(\frac{k}{n}\right)\cdot b_{n,k}(x)\quad\text{with}\quad b_{n,k}(x):=\binom{n}{k}\cdot x^k\cdot(1-x)^{n-k}.
\end{equation*}
Let $\varepsilon>0$ be arbitrary. We show that there is an $n\in\mathbb{N}$ with $\vert B_n(f)(x)-f(x)\vert <\varepsilon$ for all~$x\in[0,1]$. As the binomial formula yields $\sum_{k\leq n}b_{n,k}(x)=1$, we get
\begin{equation*}
B_n(f)(x)-f(x)=\sum_{k=0}^n\left(f\left(\frac{k}{n}\right)-f(x)\right)\cdot b_{n,k}(x).
\end{equation*}
Since $f$ is uniformly continuous, it is bounded (see Remark \ref{rem:UCtoBounded}), say by $b$. Further, there is a $\delta>0$ such that $\vert x-y\vert <\delta$ implies $\vert f(x)-f(y)\vert<\varepsilon/2$ for~$x,y\in [0,1]$. This entails that $\vert B_n(f)(x)-f(x)\vert$ is bounded by
\begin{multline*}
\sum_{\left\vert\frac kn-x\right\vert<\delta}\left\vert f\left(\tfrac{k}{n}\right)-f(x)\right\vert\cdot b_{n,k}(x)+\sum_{\left\vert\frac kn-x\right\vert\geq \delta}\left\vert f\left(\tfrac{k}{n}\right)-f(x)\right\vert\cdot b_{n,k}(x)\\[1.5ex]
{}\leq \frac{\varepsilon}{2}\cdot\sum_{\left\vert\frac{k}{n}-x\right\vert<\delta} b_{n,k}(x)+2b\cdot \sum_{\left\vert\frac{k}{n}-x\right\vert\geq \delta} b_{n,k}(x).
\end{multline*}
By elementary computations (see, e.g., the proof of Theorem~1.1.1 in~\cite{lorentz-bernstein}), one has
\begin{equation*}
    \sum_{\left\vert\frac{k}{n}-x\right\vert\geq \delta} b_{n,k}(x)\leq\frac1{\delta^2}\cdot\sum_{k=0}^n\left(\frac kn-x\right)^2\cdot b_{n,k}(x)=\frac{x\cdot(1-x)}{\delta^2\cdot n}\leq\frac1{4n\cdot\delta^2}.
\end{equation*}
Together with $\sum_{k\leq n}b_{n,k}(x)=1$ as above, we get
\[
\vert B_n(f)(x)-f(x)\vert\leq \frac{\varepsilon}{2}+\frac{b}{2n\cdot\delta^2},
\]
so that choosing $n>b/(\delta^2\cdot\varepsilon)$ suffices.
\end{proof}

If we only assume that the function is continuous, we get the following reversal.

\begin{theorem}\label{thm:weierstrass-approx}
The following are equivalent over $\mathsf{RCA}_0+\ipp$:
\begin{enumerate}[label=(\roman*)]
\item The cohesive principle $\mathsf{COH}$.
\item The Weierstrass approximation theorem: For continuous $f\colon[0,1]\to\mathbb{R}$ and any $\varepsilon>0$, there is a polynomial $p$ with $\vert f(x)-p(x)\vert <\varepsilon$ for all $x\in [0,1]$.
\end{enumerate}
\end{theorem}
\begin{proof}
Since $\mathsf{StCOH}$ is the conjunction of $\mathsf{COH}$ and $\mathsf{IPP}$, the direction from~(i) to~(ii) holds by Theorem \ref{thm:StCADSContinuous} and Lemma \ref{lem:WeierstrassUniform}. For the converse, Lemma~\ref{lem:fin-sets-bounded} ensures that all polynomials are bounded in the presence of~$\mathsf{IPP}$. So if~(ii) holds, any continuous function $f\colon[0,1]\to\mathbb{R}$ is bounded. Again by Theorem \ref{thm:StCADSContinuous}, this yields~(i).
\end{proof}

As in the case of the Riemann integral, we have analysed the Weierstrass approximation theorem relative to the pigeonhole principle, because the latter is required to handle arbitrary polynomials (see Remark~\ref{rmk:fin-sums}). The following remark indicates how the base theory can be lowered to $\mathsf{RCA}_0$ (in which case $\mathsf{COH}$ should be replaced by $\mathsf{StCOH}$) when we restrict to polynomials with rational coefficients.

\begin{remark}
    If the pigeonhole principle fails, Theorem~\ref{thm:ivt} yields a continuous function $f\colon[0,1]\to\{-1,1\}$ with $f(0)=-1$ and $f(1)=1$. Assume that we have a rational polynomial $p$ with $|f(x)-p(x)|<1$ for all~$x\in[0,1]$. Given that we have $p(0)<0<p(1)$, the intermediate value theorem from the classical setting (see Theorem~II.6.6 of~\cite{simpson09}) gives us an $x\in[0,1]$ with $p(x)=0$. But then we have $|f(x)|<1$, against the assumption. So if every continuous $f\colon[0,1]\to\mathbb R$ can be approximated by a rational polynomial, we get $\mathsf{IPP}$ and then $\mathsf{StCOH}$ over~$\mathsf{RCA}_0$.
\end{remark}

\section{Sequences}

In this section, we show how results about sequences -- such as the Bolzano-Weierstrass and Arzel\`a-Ascoli theorems -- can be accommodated in our setting. The uniform convergence condition in the following definition may well be the most controversial aspect of our approach. A philosophical justification (which connects with H.~Friedman's strict reverse mathematics) has been given in the introduction, where we set out the following crucial test: Does our definition of sequence allow us to develop large parts of analysis in a weak theory? In the present section, we argue that the overall answer is positive, though there are some limitations.

\begin{definition}\label{def:seq}
A sequence of reals~$x_n=(x_{ni})$ is a double sequence of rationals that are uniformly Cauchy, i.e., where each $\varepsilon>0$ admits an $N\in\mathbb N$ such that we have $|x_{ni}-x_{nj}|<\varepsilon$ for all~$n\in\mathbb N$ and all $i,j\geq N$.
\end{definition}

In addition to the justification above, we note that the uniformity condition for sequences is similar to the one for continuous functions (see Definition~\ref{def:continuous}), which some readers may find more appealing. As in the case of Remark~\ref{rem:FastCauchyACA}, the following can be read as a positive result.

\begin{remark}
Under arithmetical comprehension, as every Cauchy sequence of rationals can be made fast (see Remark \ref{rem:FastCauchyACA}), every family of reals can be turned into a sequence in the sense of the previous definition. For the converse, consider a function $f\colon\mathbb{N}\to\mathbb{N}$ and let $(x_n)=(x_{nk})$ be the family of reals with
\begin{equation*}
    x_{nk}=\begin{cases}
        1 & \text{if }f(i)=n\text{ for some }i<k,\\
        0 & \text{otherwise}.
    \end{cases}
\end{equation*}
Assume that the same reals can be represented by sequences $x'_n=(x'_{nk})$ that are uniformly Cauchy. This means, in particular, that we find a $K\in\mathbb N$ such that $|x_n'-x'_{nk}|<1/2$ holds for all~$k\geq K$ and every~$n\in\mathbb N$. But then $n$ lies in the image of~$f$ precisely if we have $x'_{nK}>1/2$. So over $\mathsf{RCA}_0$, arithmetical comprehension follows if every family of reals can be converted into a sequence as in Definition~\ref{def:seq}.
\end{remark}

The following technical lemma will be needed several times.

\begin{lemma}[$\mathsf{RCA}_0$]\label{lem:Cauchy-Q-R}
    Consider a sequence of reals $x_n=(x_{ni})$ and a non-decreasing function $f\colon\mathbb N\to\mathbb N$ with unbounded image. If $x=(x_{n,f(n)})$ is Cauchy (and hence a real), then $(x_n)$ converges to~$x$.
\end{lemma}
\begin{proof}
    Given $\varepsilon>0$, consider an $N\in\mathbb N$ with 
    \begin{align*}
        |x_{ni}-x_{nj}|<\frac\varepsilon2\quad & \text{for every }n\in\mathbb N\text{ and all }i,j\geq N,\\
        |x_{i,f(i)}-x_{j,f(j)}|<\frac\varepsilon2\quad & \text{for all }i,j\geq N.
    \end{align*}
    Then find an $I\geq N$ with $f(n)\geq N$ for $n\geq I$. When we have $n,i\geq I$, we get
    \begin{equation*}
        |x_{ni}-x_{i,f(i)}|\leq|x_{ni}-x_{n,f(n)}|+|x_{n,f(n)}-x_{i,f(i)}|<\varepsilon,
    \end{equation*}
    which yields $|x_n-x|\leq\varepsilon$.
\end{proof}

We now show that $\mathbb R$ is complete over a weak base theory.

\begin{proposition}[$\mathsf{RCA}_0$]\label{lem:cauchy-converges}
If a sequence of reals is Cauchy, then it converges.
\end{proposition}
\begin{proof}
Writing $(x_n)=(x_{ni})$ for the sequence in question, we show that $x := (x_{nn})$ is a real. Let us consider an arbitrary $\varepsilon > 0$. Given that $(x_n)$ is Cauchy and the rational sequences $(x_{ni})_{i\in\mathbb N}$ are uniformly so, there is an $N\in\mathbb N$ with
\begin{align*}
    |x_m-x_n|<\frac\varepsilon3\quad&\text{for }m,n\geq N,\\
    |x_{ni}-x_{nj}|<\frac\varepsilon3\quad&\text{for }i,j\geq N\text{ and any }n\in\mathbb N.
\end{align*}
When we have $m,n\geq N$, we find a $k\geq N$ with $|x_{mk}-x_{nk}|<\varepsilon/3$ (cf.~Lemma~\ref{lem:reals-ineq-Sigma02}), so that we get
\begin{equation*}
    |x_{mm}-x_{nn}|\leq|x_{mm}-x_{mk}|+|x_{mk}-x_{nk}|+|x_{nk}-x_{nn}|<\varepsilon.
\end{equation*}
The previous lemma shows that $(x_n)$ converges to~$x$.
\end{proof}

We also get completeness in the following form.

\begin{proposition}[$\mathsf{RCA}_0$]\label{prop:sup}
    Any bounded sequence of reals has a supremum.
\end{proposition}
\begin{proof}
    Writing $(x_n)=(x_{ni})$ for the sequence in question, put $s_i=\max_{n\leq i}x_{ni}$ and choose $n(i)\leq i$ with $s_i=x_{n(i),i}$. To see that the $s_i$ are bounded, assume $x_n\leq B$ for all $n\in\mathbb N$ and take an $N\in\mathbb N$ such that $i,j\geq N$ entails $|x_{ni}-x_{nj}|<1$ and hence $|x_{ni}-x_n|\leq 1$, which yields $s_i\leq B+1$ for $i\geq N$. Also, the $s_i$ are almost increasing in the following sense: Given $\varepsilon>0$, take an $I\in\mathbb N$ with $|x_{ni}-x_{nj}|<\varepsilon$ for all~$i,j\geq I$. When we have $j>i\geq I$, we get
    \begin{equation*}
        s_i=x_{n(i),i}\leq x_{n(i),j}+\varepsilon\leq s_j+\varepsilon.
    \end{equation*}
   Thus $(s_i)$ is Cauchy (as in Remark~\ref{rem:FastCauchyACA}). By Lemma~\ref{lem:Cauchy-Q-R}, the sequence $(x_{n(i)})$ converges to $x=(s_i)$, which shows $x\leq\sup_{n\in\mathbb N}x_n$. Conversely, since $i\geq n$ entails $x_{ni}\leq s_i$, we have $x_n\leq x$ for all~$n\in\mathbb N$. To avoid misunderstanding, we note that $(x_{n(i)})$ need not be a subsequence of $(x_n)$, as $i\mapsto n(i)$ could, e.g., be constant.
\end{proof}

In the case of a monotone sequence, it is immediate that the supremum is the limit. So our base theory proves the monotone convergence theorem:

\begin{corollary}[$\mathsf{RCA}_0$]\label{pro:monotoneConvergence}
Any bounded monotone sequence of reals converges.
\end{corollary}

When a result is provable in $\mathsf{RCA}_0$, one can investigate its strength over a weaker base theory. We do not systematically do this in the present paper, but we record the following known result:

\begin{remark}
    Over a theory $\mathsf{EA}$ of second order elementary arithmetic, Kohlenbach has shown that $\Sigma^0_1$-induction is equivalent to the statement that every bounded monotone sequence of reals that are represented by fast Cauchy sequences is itself Cauchy (see Proposition~5.2 and Corollary~5.3 of~\cite{kohlenbach-pra}).
\end{remark}

From monotone convergence, we obtain nested-interval completeness:

\begin{corollary}[$\mathsf{RCA}_0$]\label{cor:nestedIntervals}
Let $(x_n)$ and $(y_n)$ be sequences of real numbers with
\begin{equation*}
x_n\leq x_{n+1}\leq y_{n+1}\leq y_n\quad\text{for all }n\in\mathbb N.    
\end{equation*}
There exists a real number $z$ such that we have $x_n\leq z\leq y_n$ for all $n\in\mathbb N$. If we have $\lim_{n\to\infty}\vert x_n-y_n\vert = 0$, this $z$ is unique.
\end{corollary}
\begin{proof}
The sequences $(x_n)$ and $(y_n)$ are bounded and monotone. So they converge to some limits $x$ and $y$, by Proposition~\ref{pro:monotoneConvergence}. If we had $x>y$, we would get $x_n>y_n$ for large~$n$. So we have $x_n\leq x\leq y\leq y_n$ for all~$n\in\mathbb N$. We may thus pick $z=x$.
\end{proof}

We move on to stronger theorems for sequences of real numbers. As mentioned before, the following result is essentially due to Kreuzer~\cite{kreuzer-bolzano}. The latter works with sequences of reals that are given by fast Cauchy sequences, while our Definition~\ref{def:seq} requires uniformity but not a rate. This makes no difference for the proof but allows us to treat the Bolzano-Weierstrass theorem within a coherent approach to analysis. More specifically, we avoid the asymmetry in the result by Kreuzer, whose Cauchy sequences are fast in the input but slow in the output.

\begin{theorem}\label{thm:BolzanoWeierstrass}
The following are equivalent over $\mathsf{RCA}_0$:
\begin{enumerate}[label=(\roman*)]
\item The strong cohesive principle $\mathsf{StCOH}$.
\item The Bolzano-Weierstrass theorem: Any bounded sequence of reals has a convergent subsequence.
\end{enumerate}
\end{theorem}
\begin{proof}
That~(ii) implies~(i) follows from Lemma~\ref{lem:kreuzer-bolzano} (as a convergent sequence is Cauchy). For the converse, consider a sequence of reals~$x_n=(x_{ni})\in[-B,B]$. Put
\begin{equation*}
    x_{ni}'=\begin{cases}
        -B & \text{if }x_{ni}<-B,\\
        x_{ni} & \text{if }-B\leq x_{ni}\leq B,\\
        B & \text{if }B<x_{ni}.
    \end{cases}
\end{equation*}
One checks that the sequences $(x_{ni}')_{i\in\mathbb N}$ remain uniformly Cauchy and represent the same reals~$x_n$. So we may as well assume that the rationals $x_{nn}$ are bounded. Assuming~(i), we again use Lemma~\ref{lem:kreuzer-bolzano} to get a sequence $n(0)<n(1)<\ldots$ such that $(x_{n(i),n(i)})_{i\in\mathbb N}$ is Cauchy. By Lemma~\ref{lem:Cauchy-Q-R}, it follows that $(x_{n(i)})$ converges.
\end{proof}

In connection with sequential completeness, we also record the following core property of continuous functions.

\begin{proposition}[$\mathsf{RCA}_0+\mathsf{StCOH}$]\label{lem:fct-to-seq}
    Consider a continuous function~$f\colon\mathbb R\to\mathbb R$.
    
    (a) If $(x_n)$ is a bounded sequence of reals, so is~$(f(x_n))$.
    
    (b) If $(x_n)$ is a convergent sequence of reals, so is $(f(x_n))$ and we have
    \begin{equation*}
        \lim_{n\to\infty}f(x_n)=f\left(\lim_{n\to\infty}x_n\right).
    \end{equation*}
\end{proposition}
\begin{proof}
    (a) From Theorem~\ref{thm:StCADSContinuous} we know that the values~$f(x_n)$ are bounded. The crucial claim is that $(f(x_n))$ is a sequence in the sense of Definition~\ref{def:seq}. To see this, assume that we have $|x_n|\leq B$ for all~$n\in\mathbb N$. Writing $x_n=(x_{ni})$, we may assume that the rationals~$x_{ni}$ also have absolute value at most~$B$, as in the proof of Theorem~\ref{thm:BolzanoWeierstrass}. In order to establish uniform convergence for the Cauchy sequences $f(x_n)=(f(x_{ni})_i)_{i\in\mathbb N}$ (cf.~Definition~\ref{def:continuous}), we consider an arbitrary~$\varepsilon>0$. By the proof of uniform continuity (see statement~(\ref{eq:unif-cont}) in the proof of Theorem~\ref{thm:StCADSContinuous}), there is an $M\in\mathbb N$ with
    \begin{equation*}
        |f(p)_i-f(q)_j|<\varepsilon\quad\text{for all }p,q\in[-B,B]\cap\mathbb Q\text{ with }|p-q|<\frac1M\text{ and all }i,j\geq M.
    \end{equation*}
    Since $(x_n)$ is a sequence in the sense of Definition~\ref{def:seq}, we find an $N\geq M$ such that $|x_{ni}-x_{nj}|<1/M$ holds for all $i,j\geq N$ and all $n\in\mathbb N$. So for $i,j\geq N$, we get
    \begin{equation*}
        |f(x_{ni})_i-f(x_{nj})_j|<\varepsilon
    \end{equation*}
    independently of~$n$, as required.
    
    (b) First note that the convergent sequence $(x_n)$ is also bounded. In the presence of $\mathsf{StCOH}$ (which entails~$\mathsf{IPP}$), one can see this via Lemma~\ref{lem:fin-sets-bounded}, though the uniform convergence from Definition~\ref{def:seq} also allows for a proof over~$\mathsf{RCA}_0$. By Lemma~\ref{lem:Cauchy-Q-R}, the limit of~$(x_n)$ is given by $x=(x_{ii})$. Hence $f(x)$ is represented by the sequence of rationals~$f(x_{ii})_i$. Again by Lemma~\ref{lem:Cauchy-Q-R} (with $f(x_{ni})_i$ at the place of $x_{ni}$), it follows that the $f(x_n)$ converge with limit~$f(x)$.
\end{proof}

Part~(a) of the previous proof is the first place where we verify the uniform convergence condition of Definition~\ref{def:seq} for a sequence -- here $(f(x_n))$ -- that we have constructed (whereas previous proofs rely on the assumption that some given sequence converges uniformly). Note that uniformity for $(f(x_n))$ was only established under the assumption that $(x_n)$ is bounded. While this hints at a limitation of our approach, we emphasize that the sequence $(f(x_n))$ can be constructed when it matters most -- namely, in the case where it converges.

According to Proposition~\ref{prop:sup}, the supremum of a sequence can be constructed in~$\mathsf{RCA}_0$. In contrast, we do not know if the limit superior can be constructed below~$\mathsf{ACA}_0$. What we now is that $\mathsf{RCA}_0$ does not suffice here. This is a consequence of the following proposition, which is essentially due to Kohlenbach (see Theorem~1.4 of~\cite{kohlenbach-pra} as well as~\cite{kohlenbach-low-growth}; we include a proof that is adapted to our setting). That the limit superior is more complex than the supremum also conforms with experience from proof mining.

\begin{proposition}[$\mathsf{RCA}_0$]\label{prop:limsup-approx}
    The following are equivalent:
    \begin{enumerate}[label=(\roman*)]
        \item For any bounded sequence~$(x_n)$ of real numbers and any~$\varepsilon>0$, there is an $\varepsilon$-approximation~$q\in\mathbb Q$ to the limit superior, which means that
        \begin{itemize}[label=--]
            \item any $m$ admits an~$n\geq m$ with $x_n\geq q-\varepsilon$,
            \item there is an~$m$ such that all~$n\geq m$ validate~$x_n\leq q+\varepsilon$,
        \end{itemize}
        \item the principle of $\Sigma^0_2$-induction holds.
    \end{enumerate}
\end{proposition}
\begin{proof}
We first show that~(ii) implies~(i). Given that the sequence is bounded, we find a~$b$ with $|x_n|\leq b$ for all~$n$. Invoking~$\Sigma^0_2$-induction, we may consider the minimal~$l\in\mathbb N$ with $x_{nn}\geq b-(l+1)\cdot\varepsilon/2$ for infinitely many~$n$. For any sufficiently large~$n$, we have $|x_n-x_{nn}|\leq\varepsilon/2$. Given any~$m$, we thus find an~$n\geq m$ with
\begin{equation*}
    x_n\geq x_{nn}-\varepsilon/2\geq q-\varepsilon\quad\text{for}\quad q:=b-l\cdot\varepsilon/2.
\end{equation*}
On the other hand, the minimality of~$l$ entails that we have $x_{nn}\leq q$ for any sufficiently large~$n$. For suitable~$m$, this means that all~$n\geq m$ validate
\begin{equation*}
    x_n\leq x_{nn}+\varepsilon\leq q+\varepsilon,
\end{equation*}
as required.

We now show that~(ii) implies~(i). Given a $\Sigma^0_2$-formula~$\varphi(k)=\exists m\forall n\,\theta(m,n,k)$, let~$P_k$ be the collection of pairs~$(m,n)$ such that $\theta(m,n',k)$ holds for all~$n'<n$. Let $e_k\colon\mathbb N\to P_k$ list the elements of~$P_k$ in lexicographic order, which means that $(m,n)$ is before $(m+1,n')$ for any second components, so that not all pairs may be reached. We put $s_k(i)=0$ if $e_k(i)$ and $e_k(i+1)$ have equal first component and $s_k(i)=1$ if not. This yields
    \begin{equation*}
        \limsup_{i\to\infty} s_k(i)=\begin{cases}
            0 & \text{if we have }\varphi(k),\\
            1 & \text{otherwise}.
        \end{cases}
    \end{equation*}
    Now define a sequence~$(x_i)$ of reals (in fact of rationals) by setting
    \begin{equation*}
        x_i=\sum_{k=0}^i s_k(i)\cdot 2^{-k}.
    \end{equation*}
    Assuming that we have $\varphi(0)$ and that $\varphi(k)$ always entails~$\varphi(k+1)$, we want to establish~$\varphi(K)$ for arbitrary~$K$. Let~$q$ be a $2^{-K-1}$-approximation to the limit superior. By induction on~$k\leq K$, we prove $q\leq 2^{-k}$. For~$k=0$, we note that~$\varphi(0)$ and $\varphi(1)$ together imply~$s_0(i)=s_1(i)=0$ and hence~$x_i<2^{-1}$ for all large enough~$i$, which indeed yields $q\leq 2^{-1}+2^{-K-1}\leq 2^0$. Inductively, we now assume~$q\leq 2^{-k}$ with $k<K$. Take an~$m$ such that $x_n\leq q+2^{-K-1}<2^{-k+1}$ hold for all~$n\geq m$. For any~$l\leq k-1$, it follows that $n\geq\max(l,m)$ entails
    \begin{equation*}
        s_l(n)\cdot 2^{-k+1}\leq s_l(n)\cdot 2^{-l}\leq x_n<2^{-k+1}
    \end{equation*}
    and thus~$s_l(n)=0$. This means that we have $\varphi(l)$ for~$l\leq k-1$ and hence (using that $\varphi$ is inductive) even for all~$l\leq k+2$. Increasing~$m$ if necessary, we get $s_l(n)=0$ for all~$l\leq k+2$ and all~$n\geq m$. Considering some suitable large~$n$, we can derive
    \begin{equation*}
        q\leq x_n+2^{-K-1}\leq 2^{-K-1}+\sum_{l=k+3}^n 2^{-l}<2^{-K-1}+2^{-k-2}\leq 2^{-(k+1)},
    \end{equation*}
    as needed for the induction step.
\end{proof}

We now approach the Arzel\`a-Ascoli theorem.

\begin{definition}\label{def:seq-fct}
    Given continuous functions $f_n\colon D\to\mathbb R$ for $D\subseteq\mathbb R$ (with representations $f_n(q)=(f_n(q)_i)_{i\in\mathbb N}$ as in Definition~\ref{def:continuous}), we say that $(f_n)_{n\in\mathbb N}$ is a sequence of functions if all $x\in D$ and $\varepsilon>0$ admit an $N\in\mathbb N$ such that we have
    \begin{equation*}
        |f_n(q)_i-f_n(q)_j|<\varepsilon\quad\text{for any rational }q\in B_{1/N}(x)\text{ and all }n\in\mathbb N\text{ and }i,j\geq N.
    \end{equation*}
\end{definition}

The following example will help to motivate our definition of sequence.

\begin{example}
    The functions $f_n\colon\mathbb R\to\mathbb R$ with $f_n(x)=x^n$ trivially form a sequence, because $f_n(q)_i=q^n$ does not depend on~$i$.
\end{example}

For pointwise equicontinuous functions, the following result shows that we get a uniform version of the condition from Definition~\ref{eq:unif-cont}. We have not required this stronger condition as part of Definition~\ref{def:seq-fct}, because we did not want to force all sequences of functions to be equicontinuous (see the example above).

\begin{lemma}[$\mathsf{RCA}_0$]
Consider a sequence of continuous functions $f_n\colon D\to\mathbb R$ that are pointwise equicontinuous, i.e., where all $x\in D$ and $\varepsilon>0$ admit a $\delta>0$ with $|f_n(x)-f_n(y)|<\varepsilon$ for all $y\in B_\delta(x)$ and~$n\in\mathbb N$. Then all $x\in D$ and~$\varepsilon>0$ admit an $N\in\mathbb N$ with
\begin{equation*}
|f_n(q)_i-f_n(r)_j|<\varepsilon\quad\text{for any rationals }q,r\in B_{1/N}(x)\text{ and all }i,j\geq N.
\end{equation*}
\end{lemma}
\begin{proof}
    By the triangle inequality, we have
    \begin{equation*}
        |f_n(q)_i-f_n(r)_j|\leq|f_n(q)_i-f_n(q)|+|f_n(q)-f_n(r)|+|f_n(r)-f_n(r)_j|.
    \end{equation*}
    The middle summand on the right is covered by equicontinuity, while the other summands are covered by the condition from the previous definition.
\end{proof}

We can conclude that pointwise evaluation preserves the notion of sequence.

\begin{corollary}[$\mathsf{RCA}_0$]\label{cor:equicont-seq}
    In the situation of the previous lemma and for any $x\in D$, the reals $f_n(x)$ with $n\in\mathbb N$ form a sequence in the sense of Definition~\ref{def:seq}.
\end{corollary}
\begin{proof}
    For $x=(x_i)$, Definition~\ref{def:continuous} yields $f_n(x)=(f_n(x_i)_i)_{i\in\mathbb N}$. Given $\varepsilon>0$, take~$N$ as in the previous lemma. Then pick $N'\geq N$ such that we have $x_i\in B_{1/N}(x)$ for all $i\geq N'$. When we have $i,j\geq N'$, we get $|f_n(x_i)_i-f_n(x_j)_j|<\varepsilon$ for all $i,j\geq N'$ and every~$n\in\mathbb N$, as Definition~\ref{def:seq} demands.
\end{proof}

Applied to the previous example, this has the following implication.

\begin{example}\label{ex:x^n}
    The functions $f_n(x)=x^n$ are pointwise equicontinuous on $(-1,1)$. By the previous corollary, it follows that $(x^n)_{n\in\mathbb N}$ with $|x|<1$ (and trivially also with $|x|=1$) is a sequence in the sense of Definition~\ref{def:seq}. For $x=(x_i)$ with $|x|>1$, on the other hand, the convergence of $x^n=(x_i^n)$ becomes slower as~$n$ increases. To restore uniformity as in Definition~\ref{def:seq}, we need a rate of convergence for~$x$, which relies on arithmetical comprehension (see Remark~\ref{rem:FastCauchyACA}). This may be seen as a limitation of our approach. At the same time, one is ultimately most interested in sequences that converge, so that the exclusion of $(x^n)$ for $|x|>1$ may be tolerable.
\end{example}

The following lifts Proposition~\ref{lem:cauchy-converges} from single numbers to functions. We recall that uniform convergence means convergence with respect to the supremum norm.

\begin{proposition}[$\mathsf{RCA}_0$]\label{prop:cauchy-fct}
    Consider a sequence of continuous $f_n\colon D\to\mathbb R$ that is uniformly Cauchy, i.e., such that any $\varepsilon>0$ admits an $N\in\mathbb N$ with
    \begin{equation*}
        |f_m(x)-f_n(x)|<\varepsilon\quad\text{for any }x\in D\text{ and all }m,n\geq N. 
    \end{equation*}
    Then there is a continuous~$f\colon D\to\mathbb R$ such that $(f_n)$ converges uniformly to~$f$.
\end{proposition}
\begin{proof}
        We put $f(q)_i=f_i(q)_i$. In order to see that these rationals represent a continuous function $f\colon D\to\mathbb R$, consider any $x\in D$ and $\varepsilon>0$. By Definition~\ref{def:seq-fct}, take $N\in\mathbb N$ with $|f_n(q)_i-f_n(q)|\leq\varepsilon/5$ for any rational $q\in B_{1/N}(x)$ and all $n\in\mathbb N$ and $i\geq N$. Increasing~$N$ if necessary, we may also assume that the condition from the present proposition holds with $\varepsilon/5$ at the place of~$\varepsilon$. Finally, find an~$M\geq N$ such that $|f_N(q)-f_N(r)|\leq\varepsilon/5$ holds for $q,r\in B_{1/M}(x)$ (cf.~Lemma~\ref{lem:eps-delta}). When we have $q,r\in B_{1/M}(x)\cap\mathbb Q$ and $i,j\geq M$, we thus get
        \begin{multline*}
            |f(q)_i-f(r)_j|\leq|f_i(q)_i-f_i(q)|+|f_i(q)-f_N(q)|+|f_N(q)-f_N(r)|\\
            {}+|f_N(r)-f_j(r)|+|f_j(r)-f_j(r)_j|<\varepsilon,
        \end{multline*}
        as required by Definition~\ref{def:continuous}. Let us also note that we get $|f(q)_i-f_n(q)_i|\leq 3\varepsilon/5$ for all rationals $q\in B_{1/N}(x)$ and all~$i,n\geq N$, so that we get
        \begin{equation*}
        |f(x)-f_n(x)|<\varepsilon\quad\text{for all }n\geq N.
        \end{equation*}
        This shows that the $f_n$ converge to $f$ pointwise (as we chose $N$ depending on~$x$). To get convergence with respect to the supremum norm, consider any $\varepsilon>0$ and take $K\in\mathbb N$ with $|f_m(x)-f_n(x)|<\varepsilon/2$ for any $x\in D$ and all $m,n\geq N$. Given $n\geq K$ and $x\in D$, use pointwise convergence to find an $m\geq K$ with $|f(x)-f_m(x)|<\varepsilon/2$. The triangle inequality yields $|f(x)-f_n(x)|<\varepsilon$, as required.
\end{proof}

Finally, we come to the Arzel\`a-Ascoli theorem. The following is related to a result of Kreuzer~\cite{kreuzer-ascoli}, though the latter works with the classical representation of continuous functions. As a consequence, Kreuzer obtains an equivalence with the conjunction of $\mathsf{StCOH}$ and weak K\H{o}nig's lemma, while we get $\mathsf{StCOH}$ by itself.

\begin{theorem}\label{thm:arzela}
    The following are equivalent over $\mathsf{RCA}_0$:
    \begin{enumerate}[label=(\roman*)]
        \item The strong cohesive principle $\mathsf{StCOH}$.
        \item The Arzel\`a-Ascoli theorem: If a sequence of continuous $f_n\colon[0,1]\to[-B,B]$ with $B\in\mathbb N$ is uniformly equicontinuous, i.e., if any $\varepsilon>0$ admits $\delta>0$~with
        \begin{equation*}
            \qquad\quad|f_n(x)-f_n(y)|<\varepsilon\quad\text{for all }n\in\mathbb N\text{ and }x,y\in[0,1]\text{ with }|x-y|<\delta,
        \end{equation*}
        then a subsequence of $(f_n)$ is uniformly convergent (with continuous limit).
    \end{enumerate}
\end{theorem}
\begin{proof}
    To see that~(ii) implies~(i), it suffices to note that Arzel\`a-Ascoli is a generalization of Bolzano-Weierstrass. Specifically, assume that $(x_n)$ with $x_n=(x_{ni})_{i\in\mathbb N}$ is a bounded sequence of reals. Then the rationals $f_n(q)_i:=x_{ni}$ represent a sequence of continuous functions $f_n\colon[0,1]\to\mathbb R$ (compare the conditions from Definitions~\ref{def:seq} and~\ref{def:seq-fct}). Here $f_n$ is constant with value~$x_n$, as a real $y=(y_i)$ is mapped to
    \begin{equation*}
    f_n(y)=(f_n(y_i)_i)_{i\in\mathbb N}\quad\text{with}\quad f_n(y_i)_i=x_{ni}.
    \end{equation*}
    Hence the functions $f_n$ are uniformly bounded and uniformly equicontinuous. Given that~(ii) holds, we find a strictly increasing map $i\mapsto n(i)$ so that the functions~$f_{n(i)}$ converge uniformly to some~$f$ as $i$ grows. So the reals $x_{n(i)}=f_n(0)$ converge to the value~$f(0)$. We can conclude via Theorem~\ref{thm:BolzanoWeierstrass}.

    For the converse direction, we first note that the condition from Definition~\ref{def:seq-fct} becomes uniform in the presence of the strong cohesive principle: Any $\varepsilon>0$ admits an~$N\in\mathbb N$ such that we have
    \begin{equation*}
        |f_n(q)_i-f_n(q)_j|<\varepsilon\quad\text{for all }n\in\mathbb N\text{ and }q\in[0,1]\cap\mathbb Q\text{ and }i,j\geq N.
    \end{equation*}
    This is derived like statement~(\ref{eq:unif-cont}) in the proof of Theorem~\ref{thm:StCADSContinuous}, using Heine-Borel.
    
    We now use Proposition~\ref{prop:compact_family_rationals} to find a strictly increasing map $i\mapsto n(i)$ such that the sequence $(f_{n(i)}(q)_{n(i)})_{i\in\mathbb N}$ is Cauchy for each rational~$q$. Due to Proposition~\ref{prop:cauchy-fct}, it is enough to show that the sequence~$(f_{n(i)})$ of functions is uniformly Cauchy. Given any $\varepsilon>0$, let $\delta>0$ witness the uniform equicontinuity that is assumed by Arzel\`a-Ascoli. As shown at the beginning of this paragraph, we have an $N\geq1/\delta$ with $|f_n(q)_i-f_n(q)|\leq\varepsilon$ for all $n\in\mathbb N$ and $q\in[0,1]\cap\mathbb Q$ and $i\geq N$. Now consider the rationals $q_k=k/N$ for $k\leq N$. Using $\mathsf{IPP}$, we find an $N'\geq N$ with
    \begin{equation*}
        |f_{n(i)}(q_k)_{n(i)}-f_{n(j)}(q_k)_{n(j)}|<\varepsilon\quad\text{for all }k\leq N\text{ and }i,j\geq N'.
    \end{equation*}
    Given any $x\in[0,1]$, pick $k\leq N$ with $|x-q_k|<1/N\leq\delta$. For $i,j\geq N'$, we learn that $|f_{n(i)}(x)-f_{n(j)}(x)|$ is bounded by
    \begin{align*}
        |f_{n(i)}(x)-f_{n(i)}(q_k)|+|f_{n(i)}(q_k)-f_{n(i)}(q_k)_{n(i)}|+|f_{n(i)}(q_k)_{n(i)}-f_{n(j)}(q_k)_{n(j)}|&\\
        {}+|f_{n(j)}(q_k)_{n(j)}-f_{n(j)}(q_k)|+|f_{n(j)}(q_k)-f_{n(j)}(x)|&
    \end{align*}
    and hence by $5\varepsilon$.
\end{proof}

In the rest of this section, we discuss applications of our convergence results to the fixed-point theorems of Banach and Caristi.

\begin{remark}
In order to use the Bolzano-Weierstrass theorem and other convergence results, we need to construct sequences of reals that validate the uniformity condition from Definition~\ref{def:seq}. In Example~\ref{ex:x^n}, we have seen that this is possible in some but not all situations. The following applications provide further evidence that we can construct sequences when they are relevant. This is an important justification for our approach, though the construction of sequences remains a subtle issue, which calls for more research in the future. In the worst case -- if future work should show that several relevant applications require arithmetical comprehension -- we would still have an interesting new picture of analysis, where abstract theorems are weak while concrete applications have logical strength.
\end{remark}

In contrast to Caristi's theorem (discussed below), the Banach fixed-point theorem is already weak in the classical setting (see Theorem~2.1 of~\cite{peng-yamazaki}). Nevertheless, it is interesting that we can derive it from the convergence of Cauchy sequences in~$\mathbb R$, which is classically equivalent to arithmetical comprehension.

\begin{proposition}[$\mathsf{RCA}_0+\mathsf{StCOH}$]\label{prop:Banach}
Consider a continuous $f\colon[0,1]\to[0,1]$ that admits a real $\rho\in[0,1)$ with
\begin{equation*}
|f(x)-f(y)|\leq\rho\cdot |x-y|\quad\text{for all }x,y\in[0,1].
\end{equation*}
Starting with any~$x_0\in[0,1]$, we then have a sequence of reals~$x_n\in[0,1]$ such that $x_{n+1}=f(x_n)$ holds for all~$n\in\mathbb N$. It converges to the unique fixed-point of~$f$.
\end{proposition}
\begin{proof}
    In our setting, the crucial task is to construct $(x_n)$ as a sequence in the sense of Definition~\ref{def:seq}. As in the proof of Lemma~\ref{lem:inf-sum-fct}(b), we may assume that we have $f(q)_i\in[0,1]$ for the rational approximations that determine~$f$. Let us also show that any~$\varepsilon>0$ admits an $N\in\mathbb N$ such that all rationals $p,q\in[0,1]$ validate
    \begin{equation*}
        |p-q|\leq\varepsilon\quad\Rightarrow\quad|f(p)_i-f(q)_j|\leq\varepsilon\quad\text{for all }i,j\geq N.
    \end{equation*}
    Set $\delta=(1-\rho)\cdot\varepsilon$. Statement (\ref{eq:unif-cont}) from the proof of Theorem~\ref{thm:StCADSContinuous} yields an~$N\in\mathbb N$ such that $|f(r)_i-f(r)_j|\leq\delta/2$ holds for any rational~$r\in[0,1]$ and all~$i,j\geq N$. Given $|p-q|\leq\varepsilon$ and $i,j\geq N$, we learn that $|f(p)_i-f(q)_j|$ is bounded by
    \begin{equation*}
        |f(p)_i-f(p)|+|f(p)-f(q)|+|f(q)-f(q)_j|\leq\frac\delta2+\rho\cdot\varepsilon+\frac\delta2=\varepsilon.
    \end{equation*}
    We may assume that the given real $x_0=(x_{0i})$ has approximations $x_{0i}\in[0,1]$. To define $x_n=(x_{ni})_{i\in\mathbb N}$ by recursion, we set $x_{n+1,i}=f(x_{ni})_i\in[0,1]$. The desired equality $x_{n+1}=f(x_n)$ is immediate by Definition~\ref{def:continuous} once it is confirmed that each~$x_n$ is a real number. To achieve the latter and to show that the $x_n$ form a sequence in our sense, we verify the uniform Cauchy condition from Definition~\ref{def:seq}. Given $\varepsilon>0$, take an~$N\in\mathbb N$ such that we have $|x_{0i}-x_{0j}|\leq\varepsilon$ for all~$i,j\geq N$. We may assume that the displayed implication from the beginning of this proof holds as well (possibly for increased~$N$). For any fixed~$i,j\geq N$, we then get $|x_{ni}-x_{nj}|\leq\varepsilon$ by induction on~$n\in\mathbb N$.

    We now show that the sequence $(x_n)$ is Cauchy. Note that we cannot directly use induction to get $|x_n-x_{n+1}|\leq\rho^n$, since the latter is a $\Pi^0_2$-statement. So we work with approximations. Given any $\varepsilon>0$, use statement~(\ref{eq:unif-cont}) from the proof of Theorem~\ref{thm:StCADSContinuous} to find an $N\in\mathbb N$ such that $i\geq N$ entails $|f(q)_i-f(q)|\leq\varepsilon/2$ for any rational $q\in[0,1]$ and hence
    \begin{multline*}
        |x_{n+1,i}-x_{n+2,i}|=|f(x_{ni})_i-f(x_{n+1,i})_i|\\
        {}\leq|f(x_{ni})_i-f(x_{ni})|+|f(x_{ni})-f(x_{n+1,i})|+|f(x_{n+1,i})-f(x_{n+1,i})_i|\\
        {}\leq\rho\cdot|x_{ni}-x_{n+1,i}|+\varepsilon,
    \end{multline*}
    which inductively yields
    \begin{equation*}
        |x_{ni}-x_{n+1,i}|\leq\rho^n+\varepsilon\cdot\sum_{i=0}^{n-1}\rho^i<\rho^n+\frac\varepsilon{1-\rho}.
    \end{equation*}
    For $m<n$, we get
    \begin{equation*}
        |x_{mi}-x_{ni}|\leq\sum_{i=m}^{n-1}\rho^i+\frac{n-m}{1-\rho}\cdot\varepsilon<\frac{\rho^m}{1-\rho}+\frac{n-m}{1-\rho}\cdot\varepsilon.
    \end{equation*}
    Since $i\geq N$ and $\varepsilon>0$ were arbitrary, this shows $|x_m-x_n|\leq\rho^m/(1-\rho)$, as needed to conclude that the sequence is Cauchy. Now Proposition~\ref{lem:cauchy-converges} guarantees that there is a limit $x=\lim_{n\to\infty}x_n$. In view of $x_{n+1}=f(x_n)$, we have
    \begin{equation*}
        |x-f(x)|\leq|x-x_n|+|x_n-x_{n+1}|+|f(x_n)-f(x)|.
    \end{equation*}
    Here the right side becomes arbitrarily small as $n$ grows (see Lemma~\ref{lem:eps-delta}), which shows that $x$ is a fixed-point. For any fixed-point~$x'$ we get
    \begin{equation*}
    |x-x'|=|f(x)-f(x')|\leq\rho\cdot|x-x'|
    \end{equation*}
    and hence~$x'=x$.
\end{proof}

Under the assumptions of the previous result, we have been able to construct a sequence of reals by iterated application of a function~$f$. Without any assumption on~$f$, this is not possible over a weak theory, as the following result shows.

\begin{proposition}\label{prop:iterates-ACA}
    The following are equivalent over~$\mathsf{RCA}_0$:
    \begin{enumerate}[label=(\roman*)]
        \item The principle of arithmetical comprehension holds.
        \item For any continuous function $f\colon[0,1]\to[0,1]$ and any $x\in[0,1]$, there is a sequence~$(x_n)_{n\in\mathbb N}$ of reals with $x_0=x$ and $x_{n+1}=f(x_n)$ for all~$n\in\mathbb N$.
    \end{enumerate}
\end{proposition}
\begin{proof}
    To show that (i) implies (ii), we define $x_{n+1,i} := f(x_{ni})_i$ for all $n,i\in \mathbb{N}$. By $\Pi^0_3$-induction, we see that $x_n=(x_{ni})$ is Cauchy for each~$n\in\mathbb N$. By construction, we have $x_{n+1}=f(x_n)$. Finally, arithmetical comprehension allows us to speed up all sequences $(x_{ni})_{i\in\mathbb N}$ into fast Cauchy sequences (cf.~Remark~\ref{rem:FastCauchyACA}), so that $(x_n)_{n\in\mathbb N}$ becomes a sequence in the sense of Definition~\ref{def:seq}.

    For the converse direction, consider a $\Sigma^0_1$-statement $\varphi(n) = \exists m\ \theta(m, n)$. Aiming at a contradiction, we assume that~(ii) holds while $\{n\in\mathbb N:\varphi(n)\}$ does not exist as a set. Consider the continuous function $f\colon[0,1]\to[0,1]$ with
    \begin{equation*}
        f(x) =
        \begin{cases}
            3\cdot x & \text{if $0 \leq x \leq 1/3$,}\\
            2 - 3\cdot x & \text{if $1/3 < x < 2/3$,}\\
            3\cdot x - 2 & \text{if $2/3 \leq x \leq 1$}.
        \end{cases}
    \end{equation*}
    To obtain a representation in the sense of Definition~\ref{def:continuous}, we declare that $f(q)_i$ is equal to~$f(q)$ (defined like $f(x)$ above). As the starting point of our iteration, we take the real number $x_0=(x_{0j})\in[0,1]$ with
    \begin{equation*}
        x_{0j} = \frac{2}{3}\cdot\sum_{i=0}^js_{ij}\quad\text{for}\quad s_{ij}=
        \begin{cases}
            3^{-i} & \text{if there is an $m \leq j$ with $\theta(m,i)$,}\\
            0 & \text{otherwise}.
        \end{cases}
    \end{equation*}
    Note that $(x_{0j})$ is non-decreasing and hence indeed Cauchy. By~(ii), we obtain a sequence of iterations~$x_{n+1}=f(x_n)$. We will show that all $n\in\mathbb N$ validate
    \begin{equation*}
        x_n \leq \frac{1}{3}\ \Leftrightarrow\ \lnot\varphi(n) \qquad \text{and} \qquad x_n \geq \frac{2}{3} \ \Leftrightarrow\ \varphi(n).
    \end{equation*}
    Once this is achieved, we invoke Definition~\ref{def:seq} to find an~$I\in\mathbb N$ with $|x_n-x_{nI}|<1/6$ for all~$n\in\mathbb N$, which allows us to form the set
    \begin{equation*}
        \{n\in\mathbb N:\varphi(n)\}=\{n\in\mathbb N:x_{nI}>1/2\}.
    \end{equation*}
    To establish the open claim, we consider an arbitrary~$n\in\mathbb N$. We find an~$N>n$ such that $\varphi(N)$ holds, since we could otherwise form the set $\{n\in\mathbb N:\varphi(n)\}$ by bounded $\Sigma^0_1$-comprehension (see Theorem~II.3.9 in~\cite{simpson09}). The latter also gives us access to the rational
    \begin{equation*}
        s=\frac{2}{3}\cdot\sum_{i=0}^N s_i\quad\text{for}\quad s_i=
        \begin{cases}
            3^{-i} & \text{if $\varphi(i)$ holds,}\\
            0 & \text{otherwise}.
        \end{cases}
    \end{equation*}
    Consider the (rational) iterates $f^i(s)$ with $f^0(s)=s$ and $f^{i+1}(s)=f(f^i(s))$. We use induction on $i\leq n$ to prove
    \begin{equation*}
        f^i(s)\leq\frac13-\frac1{3^{N+1-i}}\ \Leftrightarrow\ s_i=0\qquad\text{and}\qquad f^i(s)\geq\frac23+\frac1{3^{N+1-i}}\ \Leftrightarrow\ s_i=3^{-i}.
    \end{equation*}
    As part of the same induction, we prove the auxiliary claim
    \begin{equation*}
        f^i(s)=\frac23\cdot 3^i\cdot\sum_{j=0}^{N-i} s_{i+j}.
    \end{equation*}
    If the latter holds and we have $s_i=0$, then we get
    \begin{equation*}
        f^i(s)\leq\frac23\cdot 3^i\cdot\sum_{j=1}^{N-i}3^{-i-j}=\frac13-\frac1{3^{N+1-i}}.
    \end{equation*}
    When we have $s_i=3^{-i}$ (as well as $s_N=3^{-N}$ by construction), we obtain
    \begin{equation*}
        f^i(s)\geq\frac23\cdot 3^i\cdot(s_i+s_N)=\frac23+\frac2{3^{N+1-i}}.
    \end{equation*}
    In the induction step for the auxiliary claim, the simultaneous induction hypothesis ensures that $f^i(s)$ is smaller than $1/3$ or larger than~$2/3$. If we have $f^i(s)<1/3$ and consequently $s_i=0$, the definition of $f$ yields
    \begin{equation*}
        f^{i+1}(s)=\frac23\cdot 3^{i+1}\cdot\sum_{j=1}^{N-i}s_{i+j}=\frac23\cdot 3^{i+1}\cdot\sum_{j=0}^{N-(i+1)}s_{i+1+j}.
    \end{equation*}
    When we have $f^i(s)>2/3$ and hence $s_i=3^{-i}$, we can compute
    \begin{equation*}
        f^{i+1}(s)=\frac23\cdot 3^{i+1}\cdot\left(3^{-i}+\sum_{j=0}^{N-(i+1)} s_{i+1+j}\right)-2=\frac23\cdot 3^{i+1}\cdot\sum_{j=0}^{N-(i+1)}s_{i+1+j},
    \end{equation*}
    which completes the induction.
    
    Back to our sequence of reals, another application of Definition~\ref{def:seq} yields a~$J\in\mathbb N$ with $|x_m-x_{mJ}|\leq 3^{-N-2}$ for all $m\in\mathbb N$. Write $q_m=x_{mJ}$ for $m>0$ and put
    \begin{equation*}
        q_0=\begin{cases}
            s+2\cdot 3^{-N-1} & \text{if $\varphi(N+1)$ holds},\\
            s & \text{otherwise}.
        \end{cases}
    \end{equation*}
    Note that we have
    \begin{equation*}
        q_0\leq x_0\leq q_0+\frac23\cdot\sum_{i=N+2}^\infty 3^{-i}=q_0+3^{-N-2}
    \end{equation*}
    and hence $|x_m-q_m|\leq3^{-N-2}$ also for $m=0$. We now inductively prove
    \begin{equation*}
        |q_i-f^i(s)|\leq 3^{-N+i}-3^{-N-1}\qquad\text{for }i\leq n.
    \end{equation*}
    For $i=0$, this holds by the choice of~$q_0$. In the induction step, we distinguish two cases. Let us first assume that we have $s_i=0$. By the induction hypothesis and the above, we get
    \begin{equation*}
        x_i\leq q_i+3^{-N-2}\leq f^i(s)+3^{-N-i}-3^{-N-1}+3^{-N-2}<\frac13
    \end{equation*}
    and hence $x_{i+1}=f(x_i)=3\cdot x_i$ as well as $f^{i+1}(s)=f(f^i(s))=3\cdot f^i(s)$. This yields
    \begin{align*}
        |q_{i+1}-f^{i+1}(s)|&\leq|q_{i+1}-x_{i+1}|+|x_{i+1}-f^{i+1}(s)|\leq 3^{-N-2}+3\cdot|x_i-f^i(s)|\\
        {}&\leq 3^{-N-2}+3\cdot\big(|x_i-q_i|+|q_i-f^i(s)|\big)\\
        {}&\leq 3^{-N-2}+3^{-N-1}+3^{-N+i+1}-3^{-N}<3^{-N+i+1}-3^{-N-1}.
    \end{align*}
    When we have $s_i=3^{-i}$, we similarly get $x_i>2/3$, which yields $x_{i+1}=3\cdot x_i-2$ as well as $f^{i+1}(s)=3\cdot f^i(s)-2$. So the same chain of inequalities applies. Finally, when we apply this to $i=n$, the same argument as above shows that $\neg\varphi(n)$ implies $s_n=0$ and hence $x_n<1/3$ while $\varphi(n)$ implies $s_n=1$ and hence $x_n>2/3$.
\end{proof}

To establish a version of Caristi's fixed point theorem, we introduce a representation that covers many though not all semi-continuous functions (cf.~Remark~\ref{rmk:semi-cont}). Consider any family of rationals $g_i(q)\geq 0$ indexed by~$i\in\mathbb N$ and $q\in\mathbb Q$. In a sufficiently strong meta theory, this induces a function $g\colon\mathbb R\to[0,\infty]$ via
\begin{equation*}
    g(x)=\sup_{n\in\mathbb N}\hat g_n(x)\quad\text{with}\quad\hat g_n(x)=\inf\left\{g_i(p):i\in\mathbb N\text{ and }p\in B_{2^{-n}}(x)\cap\mathbb Q\right\}.
\end{equation*}
This function is lower semi-continuous, i.e., for any~$y<g(x)$ there is an~$N\in\mathbb N$ with $y<g(x')$ for all $x'\in B_{2^{-N}}(x)$. Indeed, the latter is satisfied whenever we have~$y<\hat g_N(x)$. For any $x'\in B_{2^{-N}}(x)$, we then get $B_{2^{-n}}(x')\subseteq B_{2^{-N}}(x)$ for some~$n\in\mathbb N$, which yields $y<\hat g_N(x)\leq\hat g_n(x')\leq g(x')$.

In the weak theories that we consider, it does not seem possible to construct~$g(x)$ or even just $\hat g_n(x)$ as reals, due to the complexity of the condition $p\in B_{2^{-n}}(x)$. However, certain expressions that involve values~$g(x)$ can be interpreted as abbreviations in a canonical way. To make this explicit for a case that we will need below, we assume that $g$ maps into~$[0,\infty)$. We then have
\begin{equation*}
    g(x)\leq g(y)+z\quad\Leftrightarrow\quad\forall\varepsilon>0\forall m\exists n:\hat g_m(x)\leq\hat g_n(y)+z+\varepsilon.
\end{equation*}
In addition, we obtain
\begin{multline*}
    \hat g_m(x)\leq\hat g_n(y)+z+\frac\varepsilon2\\
    \begin{aligned}
        {}&\Rightarrow\quad\forall j\forall q\in B_{2^{-n}}(y)\exists i\exists p\in B_{2^{-m}}(x):g_i(p)\leq g_j(q)+z+\varepsilon\\
        {}&\Rightarrow\quad\hat g_m(x)\leq\hat g_n(y)+z+\varepsilon.
    \end{aligned}
\end{multline*}
These implications remain valid if we replace $g_k(r)$ by $g'_k(r):=\min_{i\leq k}g_i(r)$, which has monotonicity properties that will become relevant later. We thus get
\begin{multline}\label{eq:lower-semicont-abbrev}
    g(x)\leq g(y)+z\quad\Leftrightarrow\\
    \forall\varepsilon>0\forall m\exists n\forall j\forall q\in B_{2^{-n}}(y)\exists i\exists p\in B_{2^{-m}}(x):g'_i(p)\leq g'_j(q)+z+\varepsilon.
\end{multline}
Let us now establish our version of Caristi's theorem.

\begin{theorem}[$\mathsf{RCA}_0$]\label{thm:caristi}
Consider a continuous function $f\colon\mathbb R\to\mathbb R$ and a lower semi-continuous function $g\colon\mathbb R\to[0,\infty)$, represented as above. If we have
\begin{equation*}
    g(f(x))\leq g(x)-|x-f(x)|\quad\text{for all }x\in\mathbb R,
\end{equation*}
then~$f$ has a fixed point.
\end{theorem}
\begin{proof}
Consider $q\in\mathbb Q$ and interpret $g(f(q))\leq g(q)-|q-f(q)|$ according to (\ref{eq:lower-semicont-abbrev}). For an arbitrary~$m$ and $\varepsilon=2^{-m}$, we find an~$n$ with
\begin{equation*}
    \forall j\forall q'\in B_{2^{-n}}(q)\exists i\exists p\in B_{2^{-m}}(f(q)):g_i'(p)\leq g_j'(q')-|q-f(q)|+2^{-m}.
\end{equation*}
If we take~$q'=q$, then the condition $q'\in B_{2^{-n}}(q)$ is satisfied independently of~$n$. For any~$j$, we thus find an~$i$ and a $p\in B_{2^{-m}}(f(q))$ with
\begin{equation*}
    g_i'(p)\leq g_j'(q)-|q-f(q)|+2^{-m}\leq g_j'(q)-|q-p|+2^{-m+1}.
\end{equation*}
Since $g_i'(p)$ is non-increasing in~$i$, we may assume~$i>j$ and $|f(q)-f(q)_i|\leq 2^{-m}$, which yields~$p\in B_{2^{-m+1}}(f(q)_i)$. We have thus established
\begin{equation*}
    \forall q\in\mathbb Q\forall m,j\exists i>j\exists p\in B_{2^{-m+1}}(f(q)_i):g_i'(p)\leq g_j'(q)-|q-p|+2^{-m+1}.
\end{equation*}
Starting with $i(0)=0$ and $q_0=0$, we can now recursively search for $i(m+1)>i(m)$ and $q_{m+1}\in B_{2^{-m+1}}(f(q_m)_{i(m+1)})$ with
\begin{equation*}
    g'_{i(m+1)}(q_{m+1})\leq g'_{i(m)}(q_m)-|q_m-q_{m+1}|+2^{-m+1}.
\end{equation*}
Let us abbreviate $p_m=g'_{i(m)}(q_m)$, so that we get
\begin{equation*}
|q_m-q_{m+1}|\leq p_m-p_{m+1}+2^{-m+1}.
\end{equation*}
For $m<n$, we now obtain
\begin{equation*}
|q_m-q_n|=\sum_{k=m}^{n-1}|q_k-q_{k+1}|=\sum_{k=m}^{n-1}p_k-p_{k+1}+2^{-k+1}\leq p_m-p_n+2^{-m+2}.
\end{equation*}
In particular, this entails that the sequence $(p_m)$ is almost decreasing in the sense that we have $p_n\leq p_m+2^{-m+2}$ for $m<n$. One can conclude that $(p_m)$ and hence $(q_m)$ is Cauchy, which allows us to form the real number $x=(q_m)$.

Let us show that $x$ is a fixed point of~$f$. Towards a contradiction, we assume that we have~$|x-f(x)|>\varepsilon$ for some~$\varepsilon>0$. By the continuity of~$f$, we find an~$M\in\mathbb N$ with $|f(q_m)_{i(m+1)}-f(x)|\leq\varepsilon/3$ for all~$m\geq M$. If we choose $m$ large enough, we also get $|x-q_m|\leq\varepsilon/3$ and
\begin{multline*}
    |q_m-f(q_m)_{i(m+1)}|\leq|q_m-q_{m+1}|+|q_{m+1}-f(q_m)_{i(m+1)}|\\
    {}\leq p_m-p_{m+1}+2^{-m+2}\leq\frac\varepsilon3.
\end{multline*}
This yields
\begin{equation*}
    |x-f(x)|\leq|x-q_m|+|q_m-f(q_m)_{i(m+1)}|+|f(q_m)_{i(m+1)}-f(x)|\leq\varepsilon,
\end{equation*}
which contradicts the choice of~$\varepsilon$.
\end{proof}

To conclude this section, we comment on the representation of lower semi continuous functions.

\begin{remark}\label{rmk:semi-cont}
    Consider a continuous function $g\colon\mathbb R\to[0,\infty)$ that is represented by Cauchy sequences~$(g(q)_i)_{i\in\mathbb N}$ as in Definition~\ref{def:continuous}. We may assume that all rationals $g(q)_i$ are non-negative. By taking $h_i(q)=g(q)_i$, we get a representation of a lower semi-continuous function~$h\colon\mathbb R\to[0,\infty]$ as above. However, the functions~$g$ and~$h$ do not coincide in general. To see this, note that the represented function $g$ does not change when we set $g(q)_0=0$ for all~$q$, while this makes~$h$ constant zero. In contrast, we can achieve $g=h$ when the Cauchy sequences $(g(q)_i)$ are fast in the sense that we always have $|g(q)-g(q)_i|\leq 2^{-i}$. In this case, we consider the modified approximations $g^\star(q)_i=g(q)_i+2^{-i+2}$, which satisfy
    \begin{align*}
        g^\star(q)_{i+1}=g(q)_{i+1}+2^{-i+1}\leq g(q)+2^{-i-1}&+2^{-i+1}\\
        \leq g(q)_i+2^{-i}+2^{-i-1}&+2^{-i+1}<g(q)_i+2^{-i+2}=g^\star(q)_i.
    \end{align*}
    The sequences $(g^\star(q)_i)$ are still Cauchy and represent the same function~$g$. For notational convenience, we assume that we have $g(q)_{i+1}\geq g(q)_i$ to begin with. Under this assumption, we get
    \begin{equation*}
        \hat h_n(x)=\inf\{g(p)_i:i\in\mathbb N\text{ and }p\in B_{2^{-n}}(x)\cap\mathbb Q\}=\inf\{g(p):p\in B_{2^{-n}}(x)\cap\mathbb Q\}.
    \end{equation*}
    Given that $g$ is continuous, this yields
    \begin{equation*}
        h(x)=\sup_{n\in\mathbb N}\hat h_n(x)=g(x).
    \end{equation*}
    To summarize, without arithmetical comprehension we have no proof that every continuous function is lower semi-continuous. As a consequence, we cannot apply Theorem~\ref{thm:caristi} to an arbitrary continuous function~$g$. This is related to the crucial monotonicity property $g'_i(q)\geq g'_{i+1}(q)$ that was used in the proof of the theorem. At the same time, the theorem applies to a wide range of lower semi-continuous functions~$g$. This includes all continuous~$g$ such that the values $g(q)$ on arguments $q\in\mathbb Q$ can be given by fast Cauchy sequences or even just by non-increasing Cauchy sequences (think of right-computable reals). Let us also note that there are no restrictions on the continuous function $f$ from Theorem~\ref{thm:caristi}. So our theorem (proved in~$\mathsf{RCA}_0$) subsumes all instances that are covered by the classical approach (where Caristi's theorem for continuous functions needs arithmetical comprehension, by Theorem~2.3 of~\cite{peng-yamazaki}; see \cite{caristi} for further important results on the reverse mathematics of Caristi's theorem).
\end{remark}

\bibliographystyle{amsplain}
\bibliography{Slow-Cauchy}

\end{document}